\documentclass[a4paper,reqno,11pt]{amsart}
\usepackage{mathrsfs}
\usepackage{txfonts}
\usepackage{amssymb}
\usepackage{amsfonts,amsmath,amsthm,amssymb,stmaryrd,esint,extarrows,chemarrow,esint,color}
\usepackage[numbers,sort&compress]{natbib}
\usepackage{booktabs,tabularx}

\newtheorem{Theorem}{Theorem}[section]
\newtheorem{Lemma}{Lemma}[section]
\newtheorem{Proposition}{Proposition}[section]
\newtheorem{Remark}{Remark}[section]

\numberwithin{equation}{section}
\usepackage[left=1 in, right=1 in,top=1 in, bottom=1 in]{geometry}

\def\XXint#1#2#3{{\setbox0=\hbox{$#1{#2#3}{\int}$ }
\vcenter{\hbox{$#2#3$ }}\kern-.6\wd0}}

\DeclareMathOperator{\trace}{tr}
\DeclareMathOperator{\diver}{div}
\DeclareMathOperator{\cur}{curl}
\def\r3{\mathbb{R}^3}
\newcommand{\na}{\nabla}
\newcommand{\la}{\lambda}
\newcommand{\va}{\varepsilon}
\newcommand{\de}{\delta}
\newcommand{\De}{\Delta}
\newcommand{\al}{\alpha}
\newcommand{\pa}{\partial}
\newcommand{\ga}{\gamma}

\def\le{\leqslant}
\def\ge{\geqslant}

\let\sss= \scriptscriptstyle

\makeatletter
\@namedef{subjclassname@2020}{\textup{2020} Mathematics Subject Classification}
\makeatother

\allowdisplaybreaks
\begin{document}
\bibliographystyle{plain}

\title[Elastic Navier-Stokes-Poisson equations]{On the half-space or exterior problems of the $3D$ compressible elastic Navier-Stokes-Poisson equations}
\author{Wenpei Wu}
\address{Academy of Mathematics and Systems Science, Chinese Academy of Sciences, Beijing, 100190, China.}
\email[W.P. Wu]{wenpeiwu16@163.com}

\author{Yong Wang}
\address{South China Research Center for Applied Mathematics and Interdisciplinary Studies, School of Mathematical Sciences, South China Normal University, Guangzhou, 510631, China.}
\email[Y. Wang]{wangyongxmu@163.com}
\thanks{Corresponding author: Yong Wang.}

\begin{abstract}
We study the three-dimensional compressible elastic Navier-Stokes-Poisson equations induced by a new bipolar viscoelastic model derived here, which model the motion of the compressible electrically conducting fluids. The various boundary conditions for the electrostatic potential including the Dirichlet and Neumann boundary conditions are considered. By using a unified energy method, we obtain the unique global $H^2$ solution near a constant equilibrium state in the half-space or exterior of an obstacle. The elasticity plays a crucial role in establishing the $L^2$ estimate for the electrostatic field.
\end{abstract}

\maketitle

\setcounter{tocdepth}{3}

\pagenumbering{Roman}
\tableofcontents

\pagenumbering{arabic}

\section{Introduction}

In this paper, we focus on a half-space or an exterior problem, in which the considered domain is occupied by a compressible electrically conducting fluid. The half-space or exterior problems in fluid dynamics have attracted a lot of attention in mathematics and physics. In particular, the exterior problems are involved with two kinds of important physical flows: the motion of a rigid body through a fluid and the flow past an obstacle, cf. \cite{Berdichevsky2009,Galdi2011}. The unboundedness could be seen as an idealization of large fluid domains in the real world. To raise issues that we care about, we propose a three-dimensional damped elastic Navier-Stokes-Poisson system derived in Appendix \ref{appendix}. The derived system (see \eqref{1.1} and \eqref{14-20220629}) incorporates four features: viscosity, elasticity, electrostaticity and friction. The fixed physical boundary provides a friction which induces a damping effect. So a friction-based damping term $\al\rho u$ appears in the motion equation as below.

To be precise, we will study the half-space or exterior problems of the compressible damped elastic Navier-Stokes-Poisson equations:
\begin{align}\label{1.1}
\begin{cases}
\rho_t +\diver (\rho u)=0,\\
(\rho u)_t+\diver (\rho u\otimes u)+\na P(\rho)=\mu \De u+(\mu+\lambda)\na \diver u+c^2\diver (\rho \mathbb{F}\mathbb{F}^{T})+\rho\na\phi-\al\rho u,\\
\mathbb{F}_t+u\cdot\na \mathbb{F}=\na u\mathbb{F},\\
\De\phi=\rho-\rho_{+},\quad\quad\quad\quad\quad\quad\quad\quad\quad\quad\quad
\quad\quad\quad\quad\quad\quad\quad\quad (x,t)\in\Omega\times \mathbb{R}^{+}.
\end{cases}
\end{align}
Here, $\Omega\subset\r3$ is a half-space $\mathbb{R}^3_+=\{x=(x_1,x_2,x_3)\in\mathbb{R}^3:x_3>0\}$ or an exterior domain $\mathbb{R}^3\backslash \overline{D}$ ($D\subset \mathbb{R}^3$ is a bounded domain with boundary $\pa D$ and $\overline{D}=D\cup\pa D$). Then we denote the boundary of $\Omega$ by
\begin{align*}
\pa\Omega=\{x_3=0\}
\end{align*}
if $\Omega=\mathbb{R}^3_+$ and
\begin{align*}
\pa\Omega=\pa D
\end{align*}
if $\Omega=\mathbb{R}^3\backslash \overline{D}$. We supplement the system \eqref{1.1} with the initial and boundary conditions
\begin{align}\label{1.2'}
(\rho, u, \mathbb{F})(x,t)\mid_{t=0}=(\rho_0, u_0, \mathbb{F}_0)(x), \quad x\in\Omega
\end{align}
and
\begin{align}\label{1.2}
\begin{cases}
\mbox{no-slip\ boundary\ condition:}\qquad\quad u\mid_{\pa\Omega}=0, & t>0;\\
\mbox{Dirichlet\ or\ Neumann\ condition:}\quad\phi\mid_{\pa\Omega}=0\quad
\mbox{or}\quad\na\phi\cdot\nu\mid_{\pa\Omega}=0, & t>0;\\
\mbox{far-field\ behaviors:}\quad(\rho,u,\mathbb{F},\phi)(x,t)\to (1,0,\mathbb{I},0) \quad as \quad x\to \infty, & t>0.
\end{cases}
\end{align}
where the symbol $\nu$ denotes the unit outward normal to $\pa\Omega$ and $\mathbb{I}$ is the identity matrix. The unknown variables $\rho=\rho(x,t)>0$, $u=u(x,t)\in\r3$, $\mathbb{F}=\mathbb{F}(x,t)\in \mathbb{M}^{3\times3}$ (the set of $3\times3$ matrices with positive determinants) denote the density, the velocity and the deformation gradient of viscoelastic electrically conducting fluids, respectively. The variable $\mathbb{F}$ is also called the deformation tensor or deformation matrix in some references. The electrostatic potential $\phi=\phi(x,t)$ is coupled with the density through the Poisson equation. The pressure $P=P(\rho)$ is a smooth function satisfying $P'(\rho)>0$ for $\rho>0$.  Two constant viscosity coefficients $\mu$ and $\lambda$ satisfy the usual physical constraints $\mu>0$ and $3\la+2\mu\ge0$. The constant $\al>0$ is the friction coefficient. In the motion of fluids, we use $\rho_{+}$ to model a positive background charge distribution. In this paper, we shall consider two cases:
\begin{enumerate}
  \item $\rho_{+}\equiv \bar{\rho}>0$;
  \item $\rho_{+}=e^{-\phi}$.
\end{enumerate}
The above case $(2)$ is known to be the Boltzmann distribution. For simplicity, we only deal with the case $(1)$ $\rho_{+}\equiv \bar{\rho}$ later. In fact, the case $(2)$ can be solved similarly and some comments or explanations about this will be given where appropriate. Without loss of generality, we assume $\bar\rho=1$.

From a PDE point of view, the system \eqref{1.1} is closely related to two systems: one is called the compressible viscoelastic system (also called elastic Navier-Stokes system)
\begin{align}\label{viscoelastic-system}
\begin{cases}
\rho_t +\diver (\rho u)=0,\\
(\rho u)_t+\diver (\rho u\otimes u)+\na P(\rho)=\mu \De u+(\mu+\lambda)\na \diver u+c^2\diver  (\rho \mathbb{F}\mathbb{F}^{T}),\\
\mathbb{F}_t+u\cdot\na \mathbb{F}=\na u\mathbb{F};
\end{cases}
\end{align}
the other is called the compressible Navier-Stokes-Poisson system
\begin{align}\label{NSP-system}
\begin{cases}
\rho_t +\diver (\rho u)=0,\\
(\rho u)_t+\diver (\rho u\otimes u)+\na P(\rho)=\mu\De u+(\mu+\lambda)\na \diver u+\rho\na\phi-\al\rho u,\\
\Delta \phi=\rho-\bar\rho.
\end{cases}
\end{align}
In the following, we will review the research status about the above two systems.

The viscoelastic system \eqref{viscoelastic-system} is used to describe the macroscopic dynamics of a kind of Oldroyd-B type non-Newtonian fluids. The incompressible version of \eqref{viscoelastic-system} was first introduced in \cite{Lin-Liu-Zhang2005} and then the compressible model \eqref{viscoelastic-system} was considered in \cite{Qian-Zhang2010}. Unlike other kinds of Oldroyd-B type models investigated in \cite{Oldroyd1950,Chemin-Masmoudi2001,Elgindi-Rousset2015,La2020,DiIorio-Marcati-Spirito2020,Elgindi-Masmoudi2020}, the characteristic of this model \eqref{viscoelastic-system} is that it uses the deformation gradient $\mathbb{F}$ to characterize the internal elasticity of fluids. Since then, a lot of works on the well-posedness or asymptotics of solutions to that model were made by many researchers. As for the compressible system \eqref{viscoelastic-system}, Hu and Wang \cite{Hu-Wang2010} proved the existence and uniqueness of the local-in-time large strong solution to the Cauchy problem. Later, Hu and Wu \cite{Hu-Wu2013} obtained the unique global-in-time small strong solution by constructing the suitable a priori estimates, meanwhile, they showed the optimal time-decay rates of the solution and its lower-order derivatives by using semigroup methods developed in \cite{Ponce1985,Hoff-Zumbrun1995}, where the initial data belong to $L^1(\mathbb{R}^3)$. Under the weaker assumption that the initial data lie in $\dot{B}^{-3/2}_{2,\infty}(\mathbb{R}^3)$ (noting that $L^1(\mathbb{R}^3)\subset\dot{B}^{-3/2}_{2,\infty}(\mathbb{R}^3)$), Wu et al. \cite{Wu-Gao-Tan2017} used a pure energy method firstly introduced in \cite{Guo-Wang2012} to show the optimal time-decay rates of arbitrary spatial derivatives but the highest-order. It should be noted that the optimal time-decay rate of the highest-order spatial derivatives can be obtained by using a time-weighted argument as in \cite{Gao-Li-Yao2021}. Recently, Hu and Zhao \cite{Hu-Zhao2020-3,Hu-Zhao2020-2} proved the global existence of classical solutions to the system \eqref{viscoelastic-system} with zero shear viscosity but positive volume viscosity (namely, $\mu=0$ and $\la>0$). For the initial-boundary value problem of the compressible system \eqref{viscoelastic-system}, Qian \cite{Qian2011} obtained the unique global-in-time small strong solution and then Chen and Wu \cite{Chen-Wu2018} showed the exponential decay rates. About more related results, readers can refer to \cite{Qian-Zhang2010,Li-Wei-Yao2016,
Jiang-Jiang2021,Jiang-Jiang-Zhan2020,Han-Zi2020} and the literature therein. As for the incompressible case of the system \eqref{viscoelastic-system}, we refer the readers to \cite{Lin-Liu-Zhang2005,Chen-Zhang2006,Lin-Zhang2008,Lei-Liu-Zhou2007,Lei-Zhou2005,
Lei2010,Hu-Lin2016,Zhang-Fang2012,Jiang-Wu-Zhong2016,Jiang-Jiang-Wu2017,Lei-Liu-Zhou2008,Cai-Lei-Lin-Masmoudi2019} and the references cited therein. As mentioned above, many research developments have been achieved, however, the global existence of the large strong (or smooth) solution even in two dimensions is still open whether for the incompressible or compressible case, see open problems listed in \cite{Hu-Lin-Liu2018}.

For the Navier-Stokes-Poisson system \eqref{NSP-system} with $\al=0$, there are a wealth of research results on the Cauchy problem, cf. \cite{Li-Matsumura-Zhang2010,Wang-Wu2010,Wang2012,Tan-Yang-Zhao-Zou2013,Wang-Wang2015,Bie-Wang-Yao2017} and the references therein. Compared with the Navier-Stokes equations \cite{Matsumura-Nishida1979}, it is proved that the electrostatic field plays a good role in the global well-posedness and large-time behaviors of solutions, see \cite{Wang2012}. However, for its initial-boundary value problem on a bounded domain, the situation is totally different since the Poisson term $\rho\na\phi$ brings essential difficulties when making energy estimates. The unmanageable boundary integrals will appear when integrating by parts. In fact, the energy estimates depend on the type of the boundary condition of the electrostatic potential, say Dirichlet, Neumann, or other else. Recently, based on the method introduced in \cite{Matsumura-Nishida1983} together with a new Stokes-type estimate, the Neumann problem on a bounded domain for the system \eqref{NSP-system} with $\al=0$ has been solved by Liu and Zhong \cite{Liu-Zhong2021}, while, the Dirichlet problem is still open. Mathematically, the main difficulty in the case of the Dirichlet boundary condition is the lack of a priori estimates on $\na\phi$, which in turn is due to the lack of the control for the boundary integral terms involving the electrostatic potential. When introducing the elasticity in system \eqref{NSP-system} with $\al=0$, things changed as stated in comments below.

{\it Comments on the elasticity}. The elasticity indeed brings a good effect. When the Navier-Stokes-Poisson system \eqref{NSP-system} with $\al=0$ is coupled with the transport equations for the deformation gradient $\mathbb{F}$, it becomes the elastic Navier-Stokes-Poisson system. Then the elasticity is proved to be helpful for solving the initial-boundary value problems on bounded domains regardless of that the electrostatic potential possesses Dirichlet-type, Neumann-type or mixed Dirichlet-Neumann boundary conditions. With the help of the elastic variable, that is, the deformation $\varphi=X(x,t)-x$ not the deformation gradient $\mathbb{F}$ itself, the authors \cite{Wang-Wu2021,Wang-Shen-Wu-Zhang2022} established the effective dissipation estimates of $\na\phi$ by making good use of a reduction for the original system and a relation between the electrostatic potential and the deformation. Hence, they showed the global well-posedness and exponential stability of the initial-boundary value problems on a $3D$ bounded domain. By the way, the stabilizing effect of the elasticity in the Rayleigh-Taylor instability was verified in \cite{Jiang-Jiang-Wu2017,Jiang-Wu-Zhong2016}.

{\it Comments on the friction-based damping term $\al\rho u$ in \eqref{1.1}}. From a modeling point of view, as derived in Appendix \ref{appendix}, the damping effect induced by the friction shall occur in fluid dynamics due to the relative motion to the fixed physical boundary. Thus the friction-based damping term $\al\rho u$ in \eqref{1.1} has its physical significance in the motion of a fluid across the surface of a body, cf. \cite{Persson2000}. From a technical point of view, the damping term together with the viscous terms provide a strong dissipation mechanism so that the global well-posedness of the system \eqref{1.1} is available in this paper. So, is it possible to remove the damping term? So far it is open even for the exterior or half-space problem of the three-dimensional Navier-Stokes-Poisson equations, which can be seen in \cite{Liu-Luo-Zhong2020}. The main obstacle arises from the electrostatic field $E=-\na\phi$. It is key to derive the $L^2$ dissipation estimate for $E$ or $\na\phi$. For this purpose, one needs to establish the $L^2$ dissipation estimate for $u_t$, which cannot be realized for the exterior or half-space problem without the damping term. Note that the damping term can be removed for the problem on a bounded domain, cf. \cite{Liu-Zhong2021,Wang-Wu2021,Wang-Shen-Wu-Zhang2022}. The point in the case of bounded domains is that the viscous terms additionally possess a damping effect with the help of Poincar\'{e} inequality, that is, $\|u\|_{L^2}\lesssim\|\na u\|_{L^2}$.

{\it Our main contributions in this paper}. We develop energetic variational approaches to derive a bipolar viscoelastic system in Appendix \ref{appendix} which induces the unipolar system \eqref{1.1} mainly considered in this paper. It is natural to see that the positive background charge distribution $\rho_{+}$ has two states: the constant distribution $\rho_{+}=\bar\rho$ and the Boltzmann distribution $\rho_{+}=e^{-\phi}$. We develop a unified energy method to solve the $3D$ initial-boundary value problem of the hyperbolic-parabolic-elliptic system (say \eqref{1.1}) on a half-space or an exterior domain under various boundary conditions including the Dirichlet-type, Neumann-type and the mixed type. Our methods used in this paper are clean and effective, which will shed light on the problems for complex systems of partial differential equations under a variety of physical boundary conditions.

In short, the electrostatic field together with the elasticity bring a good effect for the Cauchy problem (cf. \cite{Tan-Wang-Wu2020}) or the initial-boundary problem (cf. \cite{Wang-Wu2021,Wang-Shen-Wu-Zhang2022}) on bounded domains but not for the exterior or half-space problem. To overcome the bad effects mainly from the electrostatic field, the damping term $\al\rho u$ in the system \eqref{1.1} is needed for now. How to solve the exterior or half-space problem \eqref{1.1}--\eqref{1.2} with $\al=0$? It will be a challenging problem in mathematics. To clarify these related results mentioned above, we make a comparative analysis in the following Table $1$.

\begin{table}[htbp]
\caption{Comparative Analysis}
\centering
\begin{tabular}{p{60pt}p{85pt}p{125pt}p{125pt}}
\toprule
\small Equations & \small Cauchy problems & \small Problems on bounded domains & \small Problems on exterior domains\\
\midrule
\eqref{viscoelastic-system} & cf. \cite{Qian-Zhang2010,Hu-Wang2010,Hu-Wu2013} &  cf. \cite{Chen-Wu2018} & cf. \cite{Qian2011} \\
\eqref{NSP-system}($\al>0$) & -- & -- & cf. \cite{Liu-Luo-Zhong2020} \\
\eqref{NSP-system}($\al=0$) & cf. \cite{Li-Matsumura-Zhang2010} & cf. \cite{Liu-Zhong2021} & Unsolved\\
\eqref{1.1}($\al>0$) & -- & -- & Solved in this paper \\
\eqref{1.1}($\al=0$) & cf. \cite{Tan-Wang-Wu2020} & cf.  \cite{Wang-Wu2021,Wang-Shen-Wu-Zhang2022} & Unsolved\\
\bottomrule
\end{tabular}
\end{table}

Note that all the results for $\al=0$ in the above table still hold for $\al>0$.

In this paper, we shall study the initial-boundary value problem \eqref{1.1}--\eqref{1.2} with constraints
\begin{align}\label{compatible'}
\begin{cases}
\rho\det \mathbb{F}=1,\smallskip \\
\mathbb{F}^{lk}\na_l\mathbb{F}^{ij}=\mathbb{F}^{lj}\na_l\mathbb{F}^{ik}.
\end{cases}
\end{align}
These constrains in \eqref{compatible'} are very natural for viscoelastic fluid models.
In fact, the constraints in \eqref{compatible'} hold for all $t>0$ by Lemma \ref{A.3} if they are valid initially. Moreover, noting Lemma \ref{A.3'}, \eqref{compatible'} infers for all $t\ge0$,
\begin{align}\label{compatible''}
\diver  (\rho \mathbb{F}^T)=\na_{j}(\rho \mathbb{F}^{jk})=0.
\end{align}

In the sequel, we state the main results on the existence and uniqueness of the global solution. Note that the deformation gradient $\mathbb{F}$ and the deformation $\varphi$ satisfy the relation $\varphi=\De^{-1}\diver \mathbb{F}^{-1}$ as \eqref{'2.4}.
\begin{Theorem}\label{th1.2}
Let $\Omega\subset\r3$ be a half-space or an exterior domain with a compact boundary $\pa\Omega\in\mathcal{C}^3$. Assume that the initial data $(\rho_{0}-1, u_{0}, \mathbb{F}_{0}-\mathbb{I})\in H^2(\Omega)$ and $\varphi_0:=\De^{-1}\diver \mathbb{F}_0^{-1}\in L^2(\Omega)$  satisfying
\begin{align}\label{compatible}
\begin{cases}
\rho_{0}\det \mathbb{F}_{0}=1,\smallskip\\
\mathbb{F}_{0}^{lk}\na_l\mathbb{F}_{0}^{ij}=\mathbb{F}_{0}^{lj}\na_l\mathbb{F}_{0}^{ik},\smallskip \\
\mathbb{F}_0=(\mathbb{I}+\na\varphi_0)^{-1}
\end{cases}
\end{align}
and for some small constant $\de_{0}>0$,
\begin{align*}
\|(\rho_{0}-1, u_{0}, \mathbb{F}_{0}-\mathbb{I})\|_{H^2}+\|\varphi_0\|_{L^2}<\de_{0}.
\end{align*}

(i) Then the problem \eqref{1.1}--\eqref{1.2} with $\rho_{+}=1$ admits a unique global solution $(\rho, u, \mathbb{F}, \na\phi)$ satisfying
\begin{align}\label{regularity-1}
\begin{cases}
\rho\in \mathcal{C}([0,+\infty);H^2(\Omega)),\quad \rho_{t}\in \mathcal{C}([0,+\infty);H^1(\Omega)),\\
u\in \mathcal{C}([0,+\infty);H^2(\Omega)\cap H^1_0(\Omega))\cap L^2([0,+\infty);H^3(\Omega)),\\
u_{t}\in \mathcal{C}([0,+\infty);L^2(\Omega))\cap L^2([0,+\infty);H^1_0(\Omega)),\\
\mathbb{F}\in \mathcal{C}([0,+\infty);H^2(\Omega)),\quad \mathbb{F}_{t}\in \mathcal{C}([0,+\infty);H^1(\Omega)),
\end{cases}
\end{align}
and
\begin{align*}
\na\phi\in \mathcal{C}([0,+\infty);H^3(\Omega)),\quad \na\phi_{t}\in \mathcal{C}([0,+\infty);H^2(\Omega)).
\end{align*}
Moreover, it holds that for all $t\ge0$,
\begin{align*}
\|(\rho-1, u, \mathbb{F}-\mathbb{I})(t)\|_{H^2}+ \|\na\phi(t)\|_{H^3}+\|(\rho_{t},u_{t},\mathbb{F}_t,\na\phi_t)(t)\|_{L^2}
+\|\varphi(t)\|_{L^2}\le C_0.
\end{align*}

(ii) Then the problem \eqref{1.1}--\eqref{1.2} with $\rho_{+}=e^{-\phi}$ admits a unique global solution $(\rho, u, \mathbb{F}, \phi)$ satisfying the same regularity \eqref{regularity-1}
and
\begin{align*}
\phi\in \mathcal{C}([0,+\infty);H^4(\Omega)),\quad \phi_{t}\in \mathcal{C}([0,+\infty);H^3(\Omega)).
\end{align*}
Moreover, it holds that for all $t\ge0$,
\begin{align*}
\|(\rho-1, u, \mathbb{F}-\mathbb{I})(t)\|_{H^2}+\|\phi(t)\|_{H^4}
+\|(\rho_{t},u_{t},\mathbb{F}_t,\phi_t,\na\phi_t)(t)\|_{L^2}+\|\varphi(t)\|_{L^2}\le C_0.
\end{align*}

The above $C_0>0$ depends only on the initial data.
\end{Theorem}

We give some remarks in the following.

\begin{Remark}
We only prove $(i)$ of Theorem \ref{th1.2} in this paper. Then $(ii)$ of Theorem \ref{th1.2} can be obtained by making some obvious modifications in light of \eqref{phi-20220706}.
\end{Remark}

\begin{Remark}
Compared with the initial-boundary value problem \eqref{1.1}--\eqref{1.2} on a bounded domain \cite{Wang-Wu2021}, the case of the half-space or exterior domain is much more complicated. This is because we cannot directly use Poincar\'{e}'s inequality to get the dissipation estimate of the velocity field $u$ itself.
\end{Remark}

\begin{Remark}
Theorem \ref{th1.2} still holds if the boundary condition for the electrostatic potential $\phi$ in \eqref{1.2} is replaced by the Dirichlet-Neumann mixed boundary condition
\begin{align*}
\phi\mid_{S_1}=0,\quad \na\phi\cdot\nu\mid_{S_2}=0,\quad t>0,
\end{align*}
where $S_1\cup S_2=\pa\Omega$ and $S_1\cap S_2=\emptyset$.
Thus Theorem \ref{th1.2} still holds for an annulus with different conditions for $\phi$ on inner and outer boundaries with the help of Poincar\'{e}'s inequality. Moreover, the exterior domain $\Omega$ in Theorem \ref{th1.2} can be more general:
\begin{align*}
\Omega=\r3\backslash(\cup_{i=1}^N\overline{D_i}),\quad D_i\subset\r3,\quad i=1,2,\dots,N,
\end{align*}
where $D_i$ $(i=1,2,\dots,N)$ are bounded domains with $C^3$-smooth boundary $\pa D_i$, respectively. Note that $\pa\Omega=\cup_{i=1}^N\pa D_i$ consists of $N$ disjoint components.
\end{Remark}

\begin{Remark}
The viscoelastic two-fluid system \eqref{14-20220629} will be studied in a forthcoming paper.
\end{Remark}

\noindent{\textbf{Notation}}. In this paper, we use $a\lesssim b$ if $a\le Cb$ for a generic constant $C>0$. The relation $a\sim b$ represents that $a\lesssim b$ and $b\lesssim a$. We denote the gradient operator $\na=\pa_{x}=(\pa_{x_1},\pa_{x_2},\pa_{x_3})^T$ and $\na_j:=\pa_{x_j}$ $(j=1,2,3)$. We denote the Frobenius inner product of two matrices $\mathbb{A},\ \mathbb{B}\in \mathbb{R}^{3\times 3}$ by $\mathbb{A}:\mathbb{B}:=\sum_{i,j=1}^3\mathbb{A}^{ij}\mathbb{B}^{ij}$. Particularly, $|\mathbb{A}|^2=\mathbb{A}:\mathbb{A}$. The usual Lebesgue spaces are represented by $L^p$ ($1\le p\le\infty$) equipped with the norm $\|\cdot\|_{L^p}$. The usual Sobolev spaces are denoted by $W^{k,p}=\{u\in L_{\rm loc}^1: D^{\alpha}u\in L^p \ \mbox{for\ all}\ |\alpha|\le k\}$, with the norm $\|\cdot\|_{W^{k,p}}$. When $p=2$, we simply write $H^{k}=W^{k,2}\ (k=1,2,...)$ equipped with the norm $\|\cdot\|_{H^k}$. And we denote $\widehat{W}^{k,p}=\{u\in L_{\rm loc}^{1}: D^{\alpha}u\in L^p,\ |\alpha|=k\}$. The spaces involving time $L^{p}([0,T];Z)$ represent all the measurable functions $f:[0,T]\to Z$ with the norm $\|f\|_{L^{p}([0,T];Z)}:=(\int_{0}^{T}\|f(t)\|_{Z}^p\,dt)^{1/p}<\infty$ for $1\le p<\infty.$ The spaces involving time $\mathcal{C}([0,T];Z)$ represent all the continuous functions $f:[0,T]\to Z$ with the norm $\|f\|_{\mathcal{C}([0,T];Z)}:=\max\limits_{0\le t\le T}\|f(t)\|_{Z}<\infty.$

The outline of this paper is as follows. In Section 2, we will make a reformulation for the original problem \eqref{1.1}--\eqref{1.2} and list some auxiliary lemmas needed in the following sections. In Section 3, we will derive the unified lower-order energy estimates of solutions for the linearized system in the half-space and the exterior domain with a compact boundary. Then the higher-order energy estimates of solutions for the linearized system in the half-space and the exterior domain are established in Section 4 and Section 5, respectively. In Section 6, we use the energy estimates obtained in Sections 3-5 to establish the a priori estimates and then finish the proof of Theorem \ref{th1.2}. In Appendix \ref{appendix}, a new bipolar viscoelastic system, which describes the dynamics of two kinds of viscoelastic electrically conducting fluids, is derived by using an energetic variational approach, which further infers the unipolar system \eqref{1.1} considered in this paper.

\section{Preliminaries}\label{se2}

In the section, we first make a reformulation of the original problem and then list some auxiliary lemmas frequently used in the next sections.

\subsection{Reformulation}

\

For a given velocity field $u(x(X,t),t)$, the flow map $x(X,t)$ can be determined by the initial value problem:
\begin{align*}
\begin{cases}
\displaystyle\frac{d}{dt}x(X,t)=u(x(X,t),t),\quad t>0,\smallskip\\
x(X,0)=X,
\end{cases}
\end{align*}
where $x,X$ are the current spatial (Eulerian) coordinate and the material (Lagrangian) coordinate for fluid particles, respectively. The deformation gradient $\widetilde{\mathbb{F}}$ can be defined as
\begin{align*}
\widetilde{\mathbb{F}}(X,t)=\frac{\pa x}{\pa X}(X,t)
\end{align*}
in the Lagrangian coordinate, while in the Eulerian coordinate, the deformation gradient $\mathbb{F}(x,t)$ can be written as
\begin{align*}
\mathbb{F}(x(X,t),t)=\widetilde{\mathbb{F}}(X,t).
\end{align*}
Moreover, we can prove that $\mathbb{F}(x,t)$ satisfies the following transport equations:
\begin{align*}
\mathbb{F}_t+u\cdot\na \mathbb{F}=\na u\mathbb{F}
\end{align*}
by using the chain rule directly (see \cite{Renardy-Hrusa-Nohel1987} for instance).

In order to study its well-posedness effectively, we need to reformulate the system (\ref{1.1}). For this purpose, we introduce the inverse of $\mathbb{F}$ denoted by
\begin{align}\label{2.1}
\mathbb{E}:=\frac{\pa X}{\pa x}=\mathbb{F}^{-1},
\end{align}
where $X=X(x,t)$ is the inverse mapping of $x(X,t)$. We define the quantity
\begin{align}\label{2.2}
\mathbb{K}:=\mathbb{E}-\mathbb{I},
\end{align}
which was first proposed by Sideris and Thomases \cite{Sideris-Thomases2004}. Note that the matrix $\mathbb{K}=(\mathbb{K}^{ij})$ is curl free (cf. \cite{Lin-Zhang2008}), so there exists a vector valued function $\varphi=(\varphi^1,\varphi^2,\varphi^3)^T$ such that $(\mathbb{K}^{i1},\mathbb{K}^{i2},\mathbb{K}^{i3})^T=\na\varphi^i$ ($i=1,2,3$). In fact, the function $\varphi$ can be selected as $\varphi(x,t)=X(x,t)-x$, which implies
\begin{align}\label{2.3}
\varphi_t+u\cdot\na\varphi+u=0.
\end{align}
The benefit of introducing $\varphi$ lies in that we can easily deduce $\varphi\mid_{\pa\Omega}=0$ from $u\mid_{\pa\Omega}=0$. By \eqref{2.1}--\eqref{2.2} and the Taylor's expansion, we have
\begin{align}\label{'2.4}
\varphi=\De^{-1}\diver\mathbb{F}^{-1}
\end{align}
and
\begin{align}\label{2.4}
\mathbb{F}=(\mathbb{I}+\mathbb{K})^{-1}=\sum_{i=0}^\infty(-1)^i\mathbb{K}^i=\mathbb{I}-\mathbb{K}+O(|\mathbb{K}|^2)
=\mathbb{I}-\mathbb{\na\varphi}+O(|\mathbb{\na\varphi}|^2),
\end{align}
where the absolute convergence of the matrix series can be guaranteed according to the fact $\|\na\varphi\|_{H^2}\ll1$ in the a priori estimates. By \eqref{compatible''}, the term $\diver (\rho\mathbb{F}\mathbb{F}^T)$ can be decomposed as
\begin{align}\label{2.5}
\na_{j}(\rho \mathbb{F}^{ik}\mathbb{F}^{jk})&=\mathbb{F}^{ik}\na_{j}(\rho \mathbb{F}^{jk})+ \rho\mathbb{F}^{jk}\na_{j}\mathbb{F}^{ik}=\rho\mathbb{F}^{jk}\na_{j}\mathbb{F}^{ik}\nonumber\\
 &=\rho(\delta^{jk}-\mathbb{K}^{jk}+O(|\mathbb{K}|^2))\na_{j}(\delta^{ik}-\mathbb{K}^{ik}+O(|\mathbb{K}|^2))\nonumber\\
 &=-\rho\na_{j}\mathbb{K}^{ij}+\rho O(|\mathbb{K}|)\na O(|\mathbb{K}|)\nonumber\\
 &=-\rho\Delta\varphi^{i}+\rho O(|\na\varphi|)\na O(|\na\varphi|).
\end{align}
Next, combining the fact $\rho\det \mathbb{F}=1$ for all $t\ge0$ in Lemma \ref{A.3} and the determinant expansion theorem,
we have
\begin{align*}
\rho=\det \mathbb{F}^{-1}=\det (\mathbb{I}+\na\varphi)=1+\diver \varphi+\frac{1}{2}\{(\diver \varphi)^2-\trace[(\na\varphi)^2]\}+\det(\na\varphi),
\end{align*}
which implies
\begin{align}\label{2.6}
\rho-1=\diver \varphi+O(|\na\varphi|^2).
\end{align}
For the sake of simplicity, let us take $\al=P'(1)=1.$ Thus, together with \eqref{2.3} and \eqref{2.4}--\eqref{2.6}, we reduce \eqref{1.1}--\eqref{1.2} with $\rho_{+}\equiv1$ to the following problem
\begin{align}\label{2.8}
\begin{cases}
L_1:=u_{t}-\mu\Delta u-(\mu+\lambda)\na\diver u+\Delta\varphi+\na\diver \varphi-\na\phi+u=R_1,\\
L_2:=\varphi_t+u=R_2,\\
\Delta \phi=\diver \varphi+O(|\na\varphi|^2),
\end{cases}
\end{align}
which is subject to the initial and boundary conditions
\begin{align}\label{2.8'}
\begin{cases}
(u,\varphi,\na\phi)(x,t)\mid_{t=0}=(u_0,\varphi_0,\na\phi_0)(x), & x\in\Omega,\\
u\mid_{\pa\Omega}=\varphi\mid_{\pa\Omega}=0, & t>0,\\
\phi\mid_{\pa\Omega}=0\quad\mbox{or}\quad\na\phi\cdot\nu\mid_{\pa\Omega}=0,  & t>0,\\
(u,\varphi,\phi)\to(0,0,0) \quad \mbox{as} \quad x\to \infty.
\end{cases}
\end{align}
The above terms $R_1$ and $R_2$ are defined by
\begin{align}\label{R12}
\begin{cases}
R_1:=-u\cdot\na u-(1-\tfrac{1}{\rho})[\mu \Delta u+(\mu+\lambda) \na\diver u]\\
\qquad\ -(\tfrac{P'(\rho)}{\rho}-1)\na [\diver \varphi+ O(|\na\varphi|^2)]+O(|\na\varphi|)\na O(|\na\varphi|),\\
R_2:=-u\cdot\na \varphi.
\end{cases}
\end{align}
If $\rho_{+}=e^{-\phi}$, then we can replace Eq. $\eqref{2.8}_3$ with
\begin{align}\label{phi-20220706}
\De\phi-\phi=\diver\varphi+O(|\na\varphi|^2)+O(\phi^2),
\end{align}
where used Taylor's formula:
\begin{align*}
e^{-\phi}=1-\phi+O(\phi^2).
\end{align*}

\begin{Remark}
Since we assume that $\|\na\varphi\|_{H^2}\ll1$ in the following a priori estimates, by \eqref{2.6}, we have
$
\frac{1}{2}\le\rho\le\frac{3}{2}.
$
By the Taylor's expansion, we get
$
1-\frac{1}{\rho},\ \frac{P'(\rho)}{\rho}-1\sim \rho-1=\diver \varphi+O(|\na\varphi|^2).
$
\end{Remark}

\subsection{Auxiliary Lemmas}

\

In the following, we list some useful lemmas which are frequently used in later sections. First, we recall the classical Gagliardo-Nirenberg-Sobolev inequality in a half-space or an exterior domain with a compact boundary.

\begin{Lemma}\label{le2.1}
Let $m$ be a positive integer, $p,q,r\in[1,+\infty]$, and $\Omega\subset\mathbb{R}^n$ be a half-space or an exterior domain with a $C^m$ compact boundary $\pa\Omega$. If $u\in L^q(\Omega)\cap\widehat{W}^{m,p}(\Omega)$, then for any integer $k\in[0,m]$,
\begin{align*}
\|\na^ku\|_{L^r(\Omega)}\le C\|\na^mu\|_{L^p(\Omega)}^{\al}\|u\|_{L^q(\Omega)}^{1-\al},
\end{align*}
where
\begin{align*}
\frac{1}{r}-\frac{k}{n}=\al(\frac{1}{p}-\frac{m}{n})+(1-\al)\frac{1}{q},
\end{align*}
with
\begin{align*}
\begin{cases}
\al\in[\frac{k}{m}, 1), & \mbox{if}\ p\in(1,+\infty)\ \mbox{and}\ m-k-\frac{n}{p}\in \mathbb{N}\cup\{0\},\smallskip\\
\al\in[\frac{k}{m}, 1], & \mbox{otherwise}.
\end{cases}
\end{align*}
Moreover, when $k=0$ and $mp<n$, the additional condition is needed: $u(x)\to 0$, as $x\to\infty$ or $u\in L^{\ga}(\Omega)$ for some finite $\ga\ge1$. The above positive constant $C$ depends only on $n$, $m$, $k$, $p$, $q$, $\al$ and $\Omega$.

As a special case often used in this paper, it holds that for $n=3$, $k=0$, $m=1$, $\al=1$, $p=2$,
\begin{align*}
\|u\|_{L^{6}}\le C\|\na u\|_{L^2}.
\end{align*}
\end{Lemma}
\begin{proof}
One refers to \cite{Nirenberg1959} for the half-space and to \cite{Francesca-Paolo2004} for the exterior domain with a compact boundary.
\end{proof}

We can use Lemma \ref{le2.1} to prove the following commutator and product estimates in the case of a half-space or an exterior domain.

\begin{Lemma}\label{le2.2}
Let $l\ge 1$ be an integer and define the commutator
\begin{align*}
[\na^l,g]h=\na^l(gh)-g\na^lh.
\end{align*}
Then we have
\begin{align*}
\|[\na^l,g]h\|_{L^{p_0}} \lesssim\|\na g\|_{L^{p_1}}
\|\na^{l-1}h\|_{L^{p_2}}+\|\na^l g\|_{L^{p_3}}\|h\|_{L^{p_4}}.
\end{align*}
In addition, we have that for $l\ge0$,
\begin{align*}
\|\na^l(gh)\|_{L^{p_0}} \lesssim\|g\|_{L^{p_1}}
\|\na^{l}h\|_{L^{p_2}} +\|\na^l g\|_{L^{p_3}} \|h\|_{L^{p_4}}.
\end{align*}
In the above, $p_0,p_1,p_2,p_3,p_4\in[1,+\infty]$ such that
\begin{align*}
\frac{1}{p_0}=\frac{1}{p_1}+\frac{1}{p_2}=\frac{1}{p_3}+\frac{1}{p_4}.
\end{align*}
\end{Lemma}
\begin{proof}
The detailed proof is similar to \cite[Lemma 3.1]{Ju2004} or \cite[Lemma A.4]{Wang-Liu-Tan2016}.
\end{proof}

Next, we give some important time-invariant identities for the density $\rho$ and the deformation gradient $\mathbb{F}$ which can be seen in \cite{Qian-Zhang2010}.

\begin{Lemma}\label{A.3}
If the initial data $(\rho_{0},\mathbb{F}_{0})$ satisfy three constraints
\begin{align*}
\diver (\rho_{0}\mathbb{F}_{0}^{T})=0, \quad \mathbb{F}_{0}^{lk}\na_{l}\mathbb{F}_{0}^{ij}=\mathbb{F}_{0}^{lj}\na_{l}\mathbb{F}_{0}^{ik},
\quad \rho_{0}\det \mathbb{F}_{0}=1,
\end{align*}
then the solution $(\rho, \mathbb{F})$ of \eqref{1.1} satisfies for all $t>0$,
\begin{align*}
\diver (\rho \mathbb{F}^{T})=0, \quad \mathbb{F}^{lk}\na_{l}\mathbb{F}^{ij}=\mathbb{F}^{lj}\na_{l}\mathbb{F}^{ik},
\quad \rho \det \mathbb{F}=1.
\end{align*}
\end{Lemma}

In addition, we have the following relationships for the above three constraints.

\begin{Lemma}\label{A.3'}
Let $\Omega\subset\mathbb{R}^3$ be a half-space or an exterior domain with a compact boundary $\pa\Omega$. If $(\rho-1,\mathbb{F}-\mathbb{I})\in H^2(\Omega)$ with $(\rho,\mathbb{F})\to (1,\mathbb{I})$  as $x\to \infty$, then
\begin{align*}
\left.
\begin{array}{ll}
\diver  (\rho \mathbb{F}^{T})=0,\smallskip\\
\mathbb{F}^{lk}\na_{l}\mathbb{F}^{ij}=\mathbb{F}^{lj}\na_{l}\mathbb{F}^{ik}
\end{array}
\right\}\Rightarrow \rho \det \mathbb{F}=1
\end{align*}
and
\begin{align*}
\left.
\begin{array}{ll}
\rho \det \mathbb{F}=1,\smallskip\\
\mathbb{F}^{lk}\na_{l}\mathbb{F}^{ij}=\mathbb{F}^{lj}\na_{l}\mathbb{F}^{ik}
\end{array}
\right\}\Rightarrow \diver  (\rho \mathbb{F}^{T})=0.
\end{align*}
\end{Lemma}
\begin{proof}
Referring to \cite{Hu-Zhao2020-3,Hu-Zhao2020-2}, the details are omitted.
\end{proof}
Finally, we will state some useful regularity estimates.

\begin{Lemma}\label{le2.3}
Let $k=0,1$. Let $\Omega\subset\mathbb{R}^n$ be a half-space or an exterior domain with a $C^3$ compact boundary $\pa\Omega$. For the Stokes problem
\begin{align*}
\begin{cases}
\diver  v=h,  &  \mbox{in}\ \Omega,\\
-\Delta v+\na q=g, & \mbox{in}\ \Omega,\\
v\mid_{\pa\Omega}=a,
\end{cases}
\end{align*}
it holds that
\begin{align*}
\|\na^{k+2}v\|_{L^2(\Omega)}+\|\na^{k+1}q\|_{L^2(\Omega)}\le C(\|h\|_{H^{k+1}(\Omega)}+\|g\|_{H^{k}(\Omega)}+\|a\|_{H^{k+\frac{3}{2}}(\pa\Omega)}+\|\na v\|_{L^2(\Omega)}).
\end{align*}
\end{Lemma}
\begin{proof}
Refer to \cite{Temam1977,Matsumura-Nishida1983}.
\end{proof}

\begin{Lemma}\label{le2.4}
Let $k=0,1$. Let $\Omega\subset\mathbb{R}^n$ be a half-space or an exterior domain with a $C^3$ compact boundary $\pa\Omega$. The symbol $\nu$ is the unit outward normal to $\pa\Omega$. Given an $f\in H^k(\Omega)$, there exists a $\phi\in H^{k+2}(\Omega)$ satisfying
\begin{align*}
\begin{cases}
\De \phi=f,  &  \mbox{in}\ \Omega,\\
\phi\mid_{\pa\Omega}=0\quad \mbox{or}\quad \na \phi\cdot\nu\mid_{\pa\Omega}=0
\end{cases}
\end{align*}
and
\begin{align*}
\|\na^{k+2}\phi\|_{L^2}\le C (\|f\|_{H^k}+\|\na\phi\|_{L^2}).
\end{align*}
\end{Lemma}
\begin{Remark}
For the case $\na \phi\cdot\nu\mid_{\pa\Omega}=0$, we additionally assume
\begin{align*}
\int_\Omega f\,dx=0.
\end{align*}
\end{Remark}

\begin{proof}
Refer to \cite{Bourguignon-Brezis1974,Matsumura-Nishida1983}.
\end{proof}

\section{Lower-order Energy Estimates}\label{se3}

Note that the relations \eqref{2.4} and \eqref{2.6}
\begin{align}\label{equivalent}
\begin{cases}
\mathbb{F}-\mathbb{I}=-\na\varphi+O(|\na\varphi|^2),\\
\rho-1=\diver \varphi+O(|\na\varphi|^2).
\end{cases}
\end{align}
It suffices to derive the energy estimates of the solution $(u,\varphi,\na\phi)$ to the linearized system \eqref{2.8}. Meanwhile, we assume that for some sufficiently small $\de>0$ and some $T>0$,
\begin{align*}
\sup_{0\le t\le T}\big[\|(\rho-1, u, \mathbb{F}-\mathbb{I})(t)\|_{H^2}+\|\De^{-1}\diver \mathbb{F}^{-1}(t)\|_{L^2}\big]<\de,
\end{align*}
which implies
\begin{align}\label{equivalent'}
\sup_{0\le t\le T}\big[\|u(t)\|_{H^2}+\|(\na\phi,\varphi)(t)\|_{H^3}\big]<\de.
\end{align}
In this section, we shall give the lower-order energy estimates of solution in the half-space or the exterior domain with a compact boundary. For the sake of convenience, we give some estimates for $R_1$ defined in \eqref{R12}. By the a priori assumptions, we have
\begin{align}\label{R-1}
R_1\approx u\na u+\na\varphi\na^2u+\na\varphi\na^2\varphi+|\na\varphi|^2\na^2u+|\na\varphi|^2\na^2\varphi,
\end{align}
which implies
\begin{align*}
\na R_1&\approx \na u\na u+u\na^2u+\na^2\varphi\na^2u+\na\varphi\na^3u+\na^2\varphi\na^2\varphi+\na\varphi\na^3\varphi\nonumber\\
&\quad+\na\varphi\na^2\varphi\na^2u+|\na\varphi|^2\na^3u+\na\varphi\na^2\varphi\na^2\varphi+|\na\varphi|^2\na^3\varphi.
\end{align*}
Then, we have
\begin{align}\label{R_1-L2}
\|R_1\|_{L^2}&\lesssim\|u\|_{L^6}\|\na u\|_{L^3}+\|\na\varphi\|_{L^\infty}\|\na^2 u\|_{L^2}+\|\na\varphi\|_{L^\infty}\|\na^2\varphi\|_{L^2}+\|\na\varphi\|_{L^\infty}^2\|\na^2u\|_{L^2}
+\|\na\varphi\|_{L^\infty}^2\|\na^2\varphi\|_{L^2}\nonumber\\
&\lesssim\de(\|\na u\|_{H^1}+\|\na^2\varphi\|_{L^2}).
\end{align}
Similarly, we have
\begin{align}\label{R_1-L6/5}
\|R_1\|_{L^{6/5}}\lesssim\de(\|\na u\|_{L^2}+\|\na\varphi\|_{H^1})
\end{align}
and
\begin{align}\label{na-R_1-L2}
\|\na R_1\|_{L^2}\lesssim\de(\|\na u\|_{H^2}+\|\na^2\varphi\|_{H^1}).
\end{align}

First, we shall construct the dissipation estimate for $u, \na u$.

\begin{Lemma}\label{le3.1}
Let $\Omega\subset\r3$ be a half-space or an exterior domain with a compact boundary $\pa\Omega\in\mathcal{C}^3$. It holds that
\begin{align}\label{3.1}
\frac{d}{dt}\|(u,\diver\varphi,\na\varphi,\na\phi)\|_{L^2}^2+\|(u,\na u)\|_{L^2}^2\lesssim \de\|\na\varphi\|_{H^1}^2.
\end{align}
\end{Lemma}
\begin{proof}
From Eq. $\eqref{2.8}_{3}$, we have
\begin{align}\label{1-20-1}
\diver \varphi_t=\Delta \phi_t-O(|\na\varphi|^2)_t.
\end{align}
Taking $\diver $ to Eq. $\eqref{2.8}_{2}$, we have
\begin{align}\label{1-20-2}
\diver  u=-\diver \varphi_t-\diver (u\cdot\na\varphi),
\end{align}
together with $(\ref{1.1})_{1}$ and $\eqref{2.6}$, which infers
\begin{align}\label{1-20-3}
O(|\na\varphi|^2)_{t}=-\diver [(\diver \varphi+O(|\na\varphi|^2))u]+\diver (u\cdot\na\varphi).
\end{align}
Integrating the resulting identity $u\cdot(L_1-R_1)-\Delta\varphi\cdot(L_2-R_2)=0$ over $\Omega$ by parts and using the boundary conditions for $(u,\varphi,\phi)$ in \eqref{2.8'}, by \eqref{1-20-1}--\eqref{1-20-3}, we obtain
\begin{align}\label{3.2}
\frac{1}{2}&\frac{d}{dt}\int_{\Omega}(|u|^2+|\diver \varphi|^2+|\na\varphi|^2+|\na\phi|^2)\,dx
+\int_{\Omega}[\mu|\na u|^2+(\mu+\lambda)|\diver  u|^2+|u|^2]\,dx\nonumber\\
&=\int_{\Omega}[R_1\cdot u-R_2\cdot\De\varphi-\diver (u\cdot\na\varphi)\diver\varphi] \,dx-
\int_\Omega [\diver \varphi+O(|\na\varphi|^2)]\na\phi\cdot u\,dx\nonumber\\
&:=I_1+I_2+I_3+I_4,
\end{align}
where we have used the facts
\begin{align*}
\int_{\Omega}u\cdot\na\diver\varphi \,dx&=-\int_{\Omega}\diver u\diver\varphi \,dx=\int_{\Omega}[\diver \varphi_{t}+\diver (u\cdot\na\varphi)]\diver\varphi\,dx\nonumber\\
&=\frac{1}{2}\frac{d}{dt}\int_{\Omega}|\diver\varphi|^2\,dx+\int_{\Omega}\diver(u\cdot\na\varphi)\diver\varphi \,dx
\end{align*}
and
\begin{align*}
-\int_{\Omega}u\cdot\na\phi \,dx&=\int_{\Omega}\diver u\phi \,dx
=\int_{\Omega}[-\diver \varphi_{t}-\diver (u\cdot\na\varphi)]\phi \,dx\nonumber\\&
=-\int_{\Omega}\Delta\phi_t\phi \,dx+\int_{\Omega} [O(|\na\varphi|^2)_{t}\phi+(u\cdot\na\varphi)\cdot\na\phi] \,dx\nonumber\\
&=\frac{1}{2}\frac{d}{dt}\int_{\Omega}|\na\phi|^2\,dx+\int_{\Omega} [O(|\na\varphi|^2)_{t}\phi +(u\cdot\na\varphi)\cdot\na\phi] \,dx\nonumber\\
&=\frac{1}{2}\frac{d}{dt}\int_{\Omega}|\na\phi|^2\,dx+\int_{\Omega} [\diver \varphi+O(|\na\varphi|^2)]\na\phi\cdot u \,dx.
\end{align*}
Then, we need to estimate the right-hand side of \eqref{3.2}. By \eqref{equivalent'}, \eqref{R_1-L6/5}, H\"{o}lder's inequality and Lemma \ref{le2.1}, we have
\begin{align*}
I_1&=\int_{\Omega}R_1\cdot u \,dx\lesssim\|R_1\|_{L^{6/5}}\|u\|_{L^6}\lesssim\de(\|\na\varphi\|_{H^1}^2+\|\na u\|_{L^2}^2);\\
I_2&=\int_{\Omega}-R_2\cdot\Delta\varphi \,dx=\int_{\Omega}u\cdot\na\varphi\cdot\De\varphi \,dx\lesssim\|u\|_{L^6}\|\na\varphi\|_{L^3}\|\na^2\varphi\|_{L^2}\lesssim\de(\|\na\varphi\|_{H^1}^2+\|\na u\|_{L^2}^2);
\\
I_3&=\int_{\Omega}-\diver (u\cdot\na\varphi)\diver\varphi\,dx\lesssim\|\na u\|_{L^3}\|\na\varphi\|_{L^6}\|\na\varphi\|_{L^2}+\|u\|_{L^6}\|\na^2\varphi\|_{L^2}\|\na \varphi\|_{L^3}\lesssim\de\|\na\varphi\|_{H^1}^2;
\\
I_4&=-
\int_\Omega [\diver \varphi+O(|\na\varphi|^2)]\na\phi\cdot u\,dx\lesssim\|\na\varphi\|_{L^3}\|\na\phi\|_{L^2}\|u\|_{L^6}\lesssim\de(\|\na\varphi\|_{H^1}^2+\|\na u\|_{L^2}^2).
\end{align*}
Plugging the above four estimates $I_{1}-I_{4}$ into \eqref{3.2}, we can obtain \eqref{3.1}.
\end{proof}
Next, we construct the dissipation estimate for $u_t, \na u_t$.
\begin{Lemma}\label{le3.2}
Let $\Omega\subset\r3$ be a half-space or an exterior domain with a compact boundary $\pa\Omega\in\mathcal{C}^3$. It holds that
\begin{align}\label{3.9}
\frac{d}{dt}\|(u_{t},\diver\varphi_{t},\na\varphi_{t},\na\phi_{t})\|_{L^2}^2+\|(u_{t},\na u_{t})\|_{L^2}^2\lesssim \de\|\na(u,\varphi)\|_{H^1}^2.
\end{align}
\end{Lemma}
\begin{proof}
From \eqref{1-20-1}--\eqref{1-20-3}, we obtain
\begin{align}\label{1-20-4}
\begin{cases}
\diver \varphi_{tt}=\Delta \phi_{tt}-O(|\na\varphi|^2)_{tt},\\
\diver  u_t=-\diver \varphi_{tt}-\diver (u\cdot\na\varphi)_t,\\
O(|\na\varphi|^2)_{tt}=-\diver [(\diver \varphi+O(|\na\varphi|^2))u]_t+\diver (u\cdot\na\varphi)_t.
\end{cases}
\end{align}
From the boundary conditions for $(u,\varphi,\phi)$ in \eqref{2.8'}, we have
\begin{align}\label{1-20-5}
\begin{cases}
u_t\mid_{\pa\Omega}=\varphi_t\mid_{\pa\Omega}=0,\\
\phi_t\mid_{\pa\Omega}=0\quad\mbox{or}\quad\na\phi_{tt}\cdot\nu\mid_{\pa\Omega}=0.
\end{cases}
\end{align}
By \eqref{1-20-4}--\eqref{1-20-5}, we can integrate the identity $u_{t}\cdot(\pa_tL_{1}-\pa_tR_{1})-\Delta\varphi_{t}\cdot(\pa_tL_{2}-\pa_tR_{2})=0$ over $\Omega$ by parts to obtain
\begin{align}\label{3.10}
\frac{1}{2}&\frac{d}{dt}\int_{\Omega}(|u_{t}|^2+|\diver \varphi_{t}|^2+|\na\varphi_{t}|^2+|\na\phi_{t}|^2)\,dx+\int_{\Omega}[\mu|\na u_{t}|^2+(\mu+\lambda)|\diver  u_{t}|^2+|u_{t}|^2]\,dx\nonumber\\
&=\int_{\Omega}[\pa_tR_{1}\cdot u_{t}-\pa_tR_{2}\cdot\De\varphi_{t}-\diver (u\cdot\na\varphi)_{t} \diver \varphi_{t}]\,dx-
\int_\Omega \na\phi_{t}\cdot[(\diver \varphi+O(|\na\varphi|^2))u]_{t}\,dx\nonumber\\
&:=J_1+J_2+J_3+J_4,
\end{align}
where we have used the facts
\begin{align*}
\int_{\Omega}u_{t}\cdot\na\diver\varphi_{t} \,dx&=-\int_{\Omega}\diver u_{t}\diver\varphi_{t} \,dx=\int_{\Omega}[\diver \varphi_{tt}+\diver (u\cdot\na\varphi)_{t}]\diver\varphi_{t} \,dx\nonumber\\
 &=\frac{1}{2}\frac{d}{dt}\int_{\Omega}
 |\diver \varphi_{t}|^2\,dx+\int_{\Omega}\diver(u\cdot\na\varphi)_{t}\diver\varphi_{t} \,dx
\end{align*}
and
\begin{align*}
-\int_{\Omega}u_{t}\cdot\na\phi_{t}\,dx
&=\int_{\Omega}\diver  u_{t}\phi_{t} \,dx=\int_{\Omega}[-\diver \varphi_{tt}-\diver (u\cdot\na\varphi)_{t}]\phi_{t} \,dx\nonumber\\
&=-\int_{\Omega}\Delta\phi_{tt}\phi_{t}\,dx+\int_{\Omega} O(|\na\varphi|^2)_{tt}\phi_{t} \,dx+\int_{\Omega}(u\cdot\na\varphi)_{t}\cdot\na\phi_{t} \,dx\nonumber\\
&=\frac{1}{2}\frac{d}{dt}\int_{\Omega}|\na\phi_{t}|^2\,dx+\int_{\Omega} [O(|\na\varphi|^2)_{tt}\phi_{t}+(u\cdot\na\varphi)_{t}\cdot\na\phi_{t}] \,dx\nonumber\\
&=\frac{1}{2}\frac{d}{dt}\int_{\Omega}|\na\phi_{t}|^2\,dx+\int_{\Omega} \na\phi_{t}\cdot[(\diver \varphi+O(|\na\varphi|^2))u]_{t} \,dx.
\end{align*}
Before estimating the above terms $J_1-J_4$, we derive some auxiliary estimates. Multiplying \eqref{1-20-1} by $-\phi_{t}$, and integrating the identity over $\Omega$, by \eqref{1-20-3}, \eqref{1-20-5} and Lemma \ref{le2.1}, we obtain
\begin{align*}
\|\na\phi_t\|_{L^2}^2=\int_{\Omega}|\na\phi_{t}|^2\,dx&=-\int_{\Omega}[\diver \varphi_{t}+O(|\na\varphi|^2)_{t}]\phi_{t}\,dx\nonumber\\
&=-\int_{\Omega}\{\diver \varphi_{t}-\diver [(\diver \varphi+O(|\na\varphi|^2))u]+\diver (u\cdot\na\varphi)\}\phi_{t}\,dx\nonumber\\
&=\int_{\Omega}[\varphi_{t}-(\diver \varphi+O(|\na\varphi|^2))u+u\cdot\na\varphi]\cdot\na\phi_{t}\,dx\nonumber\\
&\lesssim\|\varphi_{t}\|_{L^2} \|\na\phi_{t}\|_{L^2}+\|\na\varphi\|_{L^3}\|u\|_{L^6}\|\na\phi_{t}\|_{L^2}
+\|\na\varphi\|_{L^6}^2\|u\|_{L^6}\|\na\phi_{t}\|_{L^2}\nonumber\\
&\lesssim\|\varphi_{t}\|_{L^2}\|\na\phi_{t}\|_{L^2}+\|\na\varphi\|_{L^3}\|\na u\|_{L^2} \|\na\phi_{t}\|_{L^2}+\|\na\varphi\|_{L^6}^2\|\na u\|_{L^2}\|\na\phi_{t}\|_{L^2},
\end{align*}
which, together with Cauchy's inequality and $\|\varphi_{t}\|_{L^2}\lesssim\|u\|_{L^2}+\|\na\varphi\|_{L^\infty}\|u\|_{L^2}\lesssim\|u\|_{L^2}$, infers
\begin{align}\label{e1'}
\|\na\phi_t\|_{L^2}\lesssim\|u\|_{L^2}+\de\|\na u\|_{L^2}.
\end{align}
By $\eqref{2.8}_2$, H\"{o}lder's inequality and Lemma \ref{le2.1}, we easily obtain
\begin{align}\label{e1}
\begin{cases}
\|\na\varphi_t\|_{L^2}\lesssim \|\na u\|_{L^2},\\
\|\na^2\varphi_t\|_{L^2}\lesssim \|\na^2u\|_{L^2}+\de\|\na u\|_{H^1},\\
\|\na\varphi_{tt}\|_{L^2}\lesssim \|\na u_t\|_{L^2}+\de\|\na u\|_{H^1}.
\end{cases}
\end{align}
By Lemmas \ref{le2.3}--\ref{le2.4}, \eqref{1-20-1}, \eqref{e1'} and $\eqref{e1}_1$, we can deduce
\begin{align}\label{e1''}
\|\na^2\phi_t\|_{L^2}&\lesssim \|\De\phi_t\|_{L^2}+\|\na\phi_t\|_{L^2}\nonumber\\
&\lesssim \|\na\varphi_t\|_{L^2}+\|\na\varphi\|_{L^{\infty}}\|\na\varphi_t\|_{L^2}+\|\na \phi_t\|_{L^2}\nonumber\\
&\lesssim \|\na\varphi_t\|_{L^2}+\|\na \phi_t\|_{L^2}\nonumber\\
&\lesssim \|\na u\|_{L^2}+\|u\|_{L^2}+\de\|\na u\|_{L^2}\lesssim \|u\|_{H^1}.
\end{align}
Note that from \eqref{R12} and \eqref{R-1}
\begin{align}\label{e2}
\begin{cases}
\pa_tR_1\approx u_t\cdot\na u+u\cdot\na u_t+\na\varphi_t\na^2u+\na\varphi\na^2u_t+\na\varphi_t\na^2\varphi+\na\varphi\na^2\varphi_t+(|\na\varphi|^2)_t\na^2u\\
\qquad\quad+(|\na\varphi|^2)\na^2u_t+(|\na\varphi|^2)_t\na^2\varphi+(|\na\varphi|^2)\na^2\varphi_t,\\
\pa_tR_2=-u_t\cdot\na \varphi-u\cdot\na \varphi_t.
\end{cases}
\end{align}
Then, by \eqref{e1'}--\eqref{e2} and integrating by parts, we can use H\"{o}lder's inequality and Lemma \ref{le2.1} to bound the terms $J_1-J_4$:
\begin{align*}
J_1&=\int_{\Omega}\pa_tR_1\cdot u_t \,dx\lesssim\|u_t\|_{L^6}\|u_t\|_{L^2}\|\na u\|_{L^3}+\|u_t\|_{L^6}\|u\|_{L^3}\|\na u_t\|_{L^2}+\|u_t\|_{L^6}\|\na\varphi_t\|_{L^3}\|\na^2u\|_{L^2}\nonumber\\
&\quad+\|u_t\|_{L^6}\|\na^2\varphi\|_{L^3}\|\na u_t\|_{L^2}+\|\na u_t\|_{L^2}^2\|\na \varphi\|_{L^\infty}+\|u_t\|_{L^6}\|\na\varphi_t\|_{L^2} \|\na^2\varphi\|_{L^3}+\|u_t\|_{L^6}\|\na\varphi\|_{L^3}\|\na^2\varphi_t\|_{L^2}\nonumber\\
&\quad+\|\na\varphi\|_{L^\infty}\|\na\varphi_t\|_{L^3}\|\na^2u\|_{L^2}\|u_t\|_{L^6}
+\|\na\varphi\|_{L^\infty}\|\na^2\varphi\|_{L^3}\|\na u_t\|_{L^2}\|u_t\|_{L^6}
+\|\na\varphi\|_{L^\infty}^2\|\na u_t\|_{L^2}^2\\
&\quad+\|\na\varphi\|_{L^\infty}\|\na\varphi_t\|_{L^2}\|\na^2\varphi\|_{L^3}\|u_t\|_{L^6}
+\|\na\varphi\|_{L^\infty}\|\na\varphi\|_{L^3}\|u_t\|_{L^6}\|\na^2\varphi_t\|_{L^2}\nonumber\\
&\lesssim\de(\|\na u\|_{H^1}^2+\|\na u_t\|_{L^2}^2);\\
J_2&=-\int_{\Omega}\pa_tR_2\cdot\Delta\varphi_t \,dx\lesssim\|u_t\|_{L^6}\|\na\varphi\|_{L^3}\|\na^2\varphi_t\|_{L^2}+\|\na \varphi_t\|_{L^2}\|\na^2\varphi_t\|_{L^2} \|u\|_{L^\infty}\lesssim\de(\|\na u\|_{H^1}^2+\|\na u_t\|_{L^2}^2);\\
J_3&=\int_{\Omega}-\diver (u\cdot\na\varphi)_t\diver\varphi_t \,dx\lesssim\de(\|\na u\|_{H^1}^2+\|\na u_t\|_{L^2}^2);\\
J_4&=-\int_\Omega \na\phi_t\cdot[(\diver \varphi+O(|\na\varphi|^2))u]_t\,dx\lesssim\de(\|\na(u,\varphi)\|_{H^1}^2+\|\na u_t\|_{L^2}^2).
\end{align*}
Plugging the estimates for $J_1-J_4$ into \eqref{3.10}, since $\de$ is small, we get \eqref{3.9}.
\end{proof}

Next, we give a crucial lemma about the dissipation estimate for $\na\phi,\na\varphi$.

\begin{Lemma}\label{le3.3}
Let $\Omega\subset\r3$ be a half-space or an exterior domain with a compact boundary $\pa\Omega\in\mathcal{C}^3$. It holds that
\begin{align}\label{3.17}
\frac{d}{dt}\int_{\Omega}(\frac12|\varphi|^2-u\cdot\varphi)\,dx
+\|\na(\varphi,\phi)\|_{L^2}^2\lesssim\|u\|_{H^1}^2+\de\|\na^2(u,\varphi)\|_{L^2}^2.
\end{align}
\end{Lemma}
\begin{proof}
Multiplying Eq. $\eqref{2.8}_{1}$, Eq. $\eqref{2.8}_{2}$ by $-\varphi$, $\varphi-u$, respectively, summing them up and integrating the resulting identity over $\Omega$ by parts, by the boundary conditions in \eqref{2.8'}, \eqref{R_1-L6/5}, H\"{o}lder's and Cauchy's inequalities and Lemma \ref{le2.1}, we obtain for any $\va>0$,
\begin{align}\label{3.20}
&\frac{d}{dt}\int_{\Omega}(\frac12|\varphi|^2-u\cdot\varphi)\,dx+\int_{\Omega}(|\na\varphi|^2+|\diver \varphi|^2+|\na\phi|^2)\,dx\nonumber\\
&=\int_{\Omega}u^2\,dx-\int_{\Omega}\phi O(|\na\varphi|^2)\,dx-\int_{\Omega}[\mu\De u+(\mu+\lambda)\na\diver  u]\cdot\varphi\,dx\nonumber\\
&\quad-\int_{\Omega}R_1\cdot\varphi\,dx+\int_{\Omega}R_2\cdot(\varphi-u)\,dx\nonumber\\
&\lesssim\|u\|_{L^2}^2+\|\phi\|_{L^6}\|\na\varphi\|_{L^3}\|\na\varphi\|_{L^2}+\|\na u\|_{L^2}\|\na \varphi\|_{L^2}+\|R_1\|_{L^{6/5}}\|\varphi\|_{L^6}+\|R_2\|_{L^{6/5}}\|\varphi-u\|_{L^6}\nonumber\\
&\lesssim\|u\|_{L^2}^2+C_\va\|\na u\|_{L^2}^2+\va\|\na\varphi\|_{L^2}^2+\de\|\na(u,\varphi)\|_{H^1}^2.
\end{align}
Here, using the boundary conditions for $\phi$ in \eqref{2.8'}, by $\eqref{2.8}_3$, we have calculated
\begin{align*}
\int_{\Omega}\na\phi\cdot\varphi \,dx&=-\int_{\Omega}\phi\diver\varphi \,dx\nonumber\\
&=-\int_{\Omega}\phi\Delta\phi \,dx+\int_{\Omega}\phi O(|\na\varphi|^2) \,dx\nonumber\\
&=\int_{\Omega}|\na\phi|^2\,dx+\int_{\Omega}\phi O(|\na\varphi|^2) \,dx.
\end{align*}
Thus, taking $\va>0$ to be small enough, since $\de$ is small, we deduce \eqref{3.17} from \eqref{3.20}.
\end{proof}

\section{Higher-order Energy Estimates in Half-spaces}\label{se5}

In this section, we focus on the higher-order energy estimates for $(u,\varphi,\phi)$ in the half-space $\Omega=\r3_{+}=\{x\in\mathbb{R}^3: x_3>0\}$. Note that the tangential derivatives of a function still vanish on the boundary $\pa\Omega=\{x_3=0\}$ if the function values zero on $\pa\Omega$. So, we divide the energy estimates into the tangential derivatives estimates and the normal derivatives estimates. We denote the tangential derivatives $\pa=(\pa_{x_1},\pa_{x_2})$. We first establish the higher-order tangential derivatives estimates for $(u,\varphi,\phi)$ in the half-space $\Omega=\r3_{+}$.

\begin{Lemma}\label{le3.5'}
Let $\Omega=\r3_{+}$. It holds that
\begin{align}
&\frac{d}{dt}\|\pa(u,\diver\varphi,\na\varphi,\na\phi)\|_{L^2}^2
+\|\pa u\|_{L^2}^2+\|\pa\na(u,\varphi,\phi)\|_{L^2}^2\lesssim \|\na(u,\varphi,u_{t})\|_{L^2}^2+\de(\|\na\phi\|_{L^2}^2+\|\na(u,\varphi)\|_{H^1}^2);\label{3.26'}\\
&\frac{d}{dt}\|\pa^2(u,\diver \varphi,\na\varphi,\na\phi)\|_{L^2}^2
 +\|\pa^2u\|_{L^2}^2+\|\pa^2\na(u,\varphi,\phi)\|_{L^2}^2\lesssim \|\na (u,u_{t})\|_{L^2}^2+\de\|\na(u,\varphi)\|_{H^2}^2.\label{3.27'}
\end{align}
\end{Lemma}
\begin{proof}
{\it Step 1.}
For the case of the half-space $\Omega=\r3_{+}$, we deduce from the boundary conditions in \eqref{2.8'} that for $k=1,2$,
\begin{align}\label{tangential}
\begin{cases}
\pa^{k}u\mid_{\pa\Omega}=\pa^{k}\varphi\mid_{\pa\Omega}=0,\\
\pa^{k}\phi\mid_{\pa\Omega}=0\quad\mbox{or}\quad\pa^{k}\na\phi\cdot\nu\mid_{\pa\Omega}=0.
\end{cases}
\end{align}
Integrating the identity $\pa(L_1-R_1)\cdot\pa u+\pa(L_2-R_2)\cdot(-\pa\Delta \varphi)=0$ over $\Omega$ by parts, by \eqref{1-20-1}, \eqref{1-20-2} and \eqref{tangential}, we obtain
\begin{align}\label{3.30'}
\frac{1}{2}&\frac{d}{dt}\int_{\Omega}(|\pa u|^2+|\pa \diver \varphi|^2+|\pa \na\varphi|^2+|\pa \na\phi|^2)\,dx
+\int_{\Omega}[\mu|\pa \na u|^2+(\mu+\lambda)|\pa\diver  u|^2+|\pa u|^2]\,dx\nonumber\\
&=-\int_{\Omega}\pa\diver\varphi\cdot\pa\diver(u\cdot\na\varphi)\,dx
-\int_{\Omega}\pa[O(|\na\varphi|^2)_{t}-\diver (u\cdot\na\varphi)]\cdot \pa\phi\,dx\nonumber\\
&\quad+\int_{\Omega}\pa R_1\cdot\pa u\,dx-\int_{\Omega}\pa R_2\cdot\pa\Delta\varphi\,dx\nonumber\\
&:=H_1+H_2+H_3+H_4,
\end{align}
where we have computed
\begin{align*}
\int_{\Omega}&\pa\na\diver\varphi\cdot\pa u \,dx=-\int_{\Omega}\pa\diver\varphi\cdot\pa\diver u\,dx\\
&=\int_{\Omega}\pa\diver\varphi\cdot\pa [\diver  \varphi_{t}+\diver (u\cdot\na\varphi)]\,dx\\
&=\frac{1}{2}\frac{d}{dt}\int_{\Omega}|\pa\diver\varphi|^2\,dx
+\int_{\Omega}\pa\diver\varphi\cdot\pa\diver(u\cdot\na\varphi)\,dx
\end{align*}
and
\begin{align*}
-&\int_{\Omega}\pa\na\phi\cdot\pa u \,dx=\int_{\Omega}\pa\phi\cdot\pa\diver  u\,dx\nonumber\\
&=-\int_{\Omega}\pa[\Delta\phi_{t}-O(|\na\varphi|^2)_{t}+\diver (u\cdot\na\varphi)]\cdot \pa\phi \,dx\nonumber\\
&=\frac{1}{2}\frac{d}{dt}\int_{\Omega}|\pa\na\phi|^2\,dx+\int_{\Omega}\pa[O(|\na\varphi|^2)_{t}-\diver (u\cdot\na\varphi)]\cdot \pa\phi\,dx.
\end{align*}
Then, by integrating by parts, \eqref{R_1-L2}, \eqref{e1}, H\"{o}lder's and Cauchy's inequalities and Lemma \ref{le2.1}, we can bound the right-hand side of \eqref{3.30'}:
\begin{align*}
&H_1=-\int_{\Omega}\pa\diver\varphi\cdot\pa\diver (u\cdot\na\varphi)\,dx\lesssim\|\na^2\varphi\|_{L^2}(\|\na\varphi\|_{L^{\infty}}\|\na^2u\|_{L^2}
+\|\na u\|_{L^3}\|\na^2\varphi\|_{L^{6}}+\|u\|_{L^{\infty}}\|\na^3\varphi\|_{L^2})\nonumber\\
&\quad\lesssim\de(\|\na u\|_{H^1}^2+\|\na^2\varphi\|_{L^2}^2);
\\
&H_2=-\int_{\Omega}\pa[O(|\na\varphi|^2)_{t}-\diver (u\cdot\na\varphi)]\cdot \pa\phi \,dx\nonumber\\
&\quad\lesssim\|\na\phi\|_{L^2}(\|\na^2\varphi\|_{L^6}\|\na\varphi_t\|_{L^3}
+\|\na\varphi\|_{L^\infty}\|\na^2\varphi_t\|_{L^2}+\|\na^2\varphi\|_{L^6}\|\na u\|_{L^3}
+\|\na \varphi\|_{L^\infty}\|\na^2u\|_{L^2}+\|u\|_{L^{\infty}}\|\na^3\varphi\|_{L^2})\nonumber\\
&\quad\lesssim\de(\|\na u\|_{H^1}^2+\|\na\phi\|_{L^2}^2);
\\
&H_3=\int_{\Omega}\pa R_1\cdot\pa u\,dx=-\int_{\Omega}R_1\cdot\pa^2u\,dx\lesssim\int_{\Omega}|R_1|\cdot|\na^2u|\,dx\nonumber\\
&\quad\lesssim\|R_1\|_{L^2}\|\na^2u\|_{L^2}\lesssim\de(\|\na u\|_{H^1}+\|\na^2\varphi\|_{L^2})\|\na^2u\|_{L^2}\lesssim\de(\|\na u\|_{H^1}^2+\|\na^2\varphi\|_{L^2}^2);
\\
&H_4=-\int_{\Omega}\pa R_2\cdot\pa\Delta\varphi\,dx\lesssim\|\na u\|_{L^6}\|\na \varphi\|_{L^3}\|\na^3 \varphi\|_{L^2}+\|u\|_{L^\infty}\|\na^2\varphi\|_{L^2}\|\na^3\varphi\|_{L^2}\lesssim\de(\|\na u\|_{H^1}^2+\|\na \varphi\|_{H^1}^2).
\end{align*}
Putting the above estimates for $H_1,H_2,H_3,H_4$ into \eqref{3.30'}, we have
\begin{align}\label{7-15-1}
&\frac{d}{dt}\|(\pa u,\pa\diver\varphi,\pa\na\varphi,\pa\na\phi)\|_{L^2}^2
 +\|(\pa u,\pa\na u)\|_{L^2}^2\lesssim \de(\|\na\phi\|_{L^2}^2+\|\na(u,\varphi)\|_{H^1}^2).
\end{align}
Integrating the identity $\pa^2(L_1-R_1)\cdot\pa^2 u+\pa^2(L_2-R_2)\cdot(-\pa^2\Delta \varphi)=0$ over $\Omega$, like \eqref{7-15-1}, we have
\begin{align}\label{7-15-2}
&\frac{d}{dt}\|(\pa^2 u,\pa^2\diver\varphi,\pa^2\na\varphi,\pa^2\na\phi)\|_{L^2}^2
 +\|(\pa^2u,\pa^2\na u)\|_{L^2}^2\lesssim \de(\|\na u\|_{H^2}^2+\|\na\varphi\|_{H^2}^2).
\end{align}

{\it Step 2.}
Integrating the identity $\pa(L_1-R_1)\cdot\pa\varphi=0$ over $\Omega$ by parts, by $\eqref{2.8}_3$, \eqref{R_1-L2}, \eqref{1-20-1}, \eqref{1-20-2}, \eqref{tangential}, H\"{o}lder's and Cauchy's inequalities and Lemma \ref{le2.1}, we obtain for any $\va>0$,
\begin{align}\label{3.44'}
\int_{\Omega}&(|\pa\na\varphi|^2+|\pa\diver \varphi|^2+|\pa\na\phi|^2)\,dx\nonumber\\
&=-\int_{\Omega}\pa\phi\cdot \pa O(|\na\varphi|^2)\,dx+\int_{\Omega}\pa(u_{t}+u)\cdot\pa\varphi \,dx+\mu\int_{\Omega}\pa\na u\cdot\pa\na\varphi \,dx\nonumber\\
&\quad+(\mu+\lambda)\int_{\Omega}\pa\diver u\cdot\pa\diver\varphi \,dx+\int_{\Omega}R_1\cdot\pa^2\varphi \,dx\nonumber\\
&\lesssim\int_{\Omega}|\na\phi|\cdot|\na\varphi|\cdot|\na^2\varphi|\,dx+\int_{\Omega}(|\na u_{t}|+|\na u|)\cdot|\na\varphi| \,dx+\int_{\Omega}|\pa\na u|\cdot|\pa\na\varphi| \,dx+\int_{\Omega}|R_1|\cdot|\pa\na\varphi| \,dx\nonumber\\
&\lesssim\|\na\phi\|_{L^6}\|\na\varphi\|_{L^3}\|\na^2\varphi\|_{L^2}+\|\na u_{t}\|_{L^2}\|\na\varphi\|_{L^2}+\|\na u\|_{L^2} \|\na\varphi\|_{L^2}+\|\pa\na u\|_{L^2}\|\pa\na\varphi\|_{L^2}
+\|R_1\|_{L^2}\|\pa\na\varphi\|_{L^2}\nonumber\\
&\lesssim\|(\na\varphi,\na u_{t},\na u)\|_{L^2}^2+C_\va\|\pa\na u\|_{L^2}^2+\va\|\pa\na\varphi\|_{L^2}^2+\de(\|\pa\na\varphi\|_{L^2}^2+\|\na(u,\varphi)\|_{H^1}^2),
\end{align}
where we have computed
\begin{align*}
-&\int_{\Omega}\pa\na\phi\cdot\pa\varphi \,dx=\int_{\Omega}\pa\phi\cdot\pa\diver \varphi \,dx\nonumber\\
&=\int_{\Omega}\pa\phi\cdot\pa\De\phi \,dx-\int_{\Omega}\pa\phi\cdot \pa O(|\na\varphi|^2)\,dx\nonumber\\
&=-\int_{\Omega}|\pa\na\phi|^2\,dx-\int_{\Omega}\pa\phi\cdot \pa O(|\na\varphi|^2)\,dx.
\end{align*}
So, letting $\va>0$ be small enough, since $\de$ is small, we deduce from \eqref{3.44'} that
\begin{align}\label{7-15-3}
\|(\pa\na\varphi,\pa\na\phi)\|_{L^2}^2\lesssim \|(\na\varphi,\na u_{t},\na u,\pa\na u)\|_{L^2}^2+\de\|\na^2(u,\varphi)\|_{L^2}^2.
\end{align}
Integrating the identity $\pa^2(L_1-R_1)\cdot\pa^2\varphi=0$ over $\Omega$, similar to the proof of \eqref{7-15-3}, we can obtain
\begin{align}\label{7-15-4}
\|(\pa^2\na\varphi,\pa^2\na\phi)\|_{L^2}^2\lesssim \|(\na u_{t},\na u,\pa^2\na u)\|_{L^2}^2+\de\|\na(u,\varphi)\|_{H^2}^2.
\end{align}

{\it Step 3.} Let $\eta>0$ be a small but fixed constant. Computing $\eqref{7-15-3}\times\eta+\eqref{7-15-1}$ and $\eqref{7-15-4}\times\eta+\eqref{7-15-2}$, respectively, we deduce \eqref{3.26'} and \eqref{3.27'}.
\end{proof}

Next, we derive the estimates of the normal derivatives of $\diver \varphi$ in the half-space.
\begin{Lemma}\label{le3.8'}
Let $\Omega=\r3_{+}$. It holds that
\begin{align}\label{3.56'}
\frac{d}{dt}\|\pa_{x_3}\diver\varphi\|_{L^2}^2+\|\pa_{x_3}(\tfrac{D\diver\varphi}{Dt},\diver \varphi)\|_{L^2}^2\lesssim\|(u_{t},u,\na\phi)\|_{L^2}^2+\|\pa\na (u,\varphi)\|_{L^2}^2+\de\|(\na u,\na^2\varphi)\|_{H^1}^2
\end{align}
and
\begin{align}\label{3.57'}
\frac{d}{dt}\|\pa^{\kappa}\pa_{x_3}^{\iota+1}\diver\varphi\|_{L^2}^{2}
+\|\pa^{\kappa}\pa_{x_3}^{\iota+1}(\tfrac{D\diver\varphi}{Dt},\diver\varphi)\|_{L^2}^{2}\lesssim\|\na( u,u_{t},\na\phi)\|_{L^2}^2+\|\pa^{\kappa+1}\pa_{x_3}^{\iota}\na (u,\varphi)\|_{L^2}^2+\de\|\na(u,\varphi)\|_{H^2}^2,
\end{align}
where $\kappa+\iota=1$.
\end{Lemma}
\begin{proof}
Denote the material derivative
\begin{align*}
\tfrac{D f}{Dt}:=\pa_tf+u\cdot\na f.
\end{align*}
Applying the divergence operator $\diver $ to both sides of $\eqref{2.3}$, we have
\begin{align}\label{ncu}
\diver  u=-\tfrac{D\diver\varphi}{Dt}-(\na u)^{T}:\na\varphi.
\end{align}
Plugging \eqref{ncu} into Eq. $\eqref{2.8}_{1}$, we obtain
\begin{align*}
u_{t}-\mu\Delta (u-\tfrac{\varphi}{\mu})+\mu\na\diver  (u-\tfrac{\varphi}{\mu})+(2\mu+\lambda)\na[\tfrac{D\diver \varphi}{Dt}+(\na u)^{T}:\na\varphi]+2\na \diver \varphi-\na\phi+u=R_1.
\end{align*}
Choosing the third component of the above identity, we have
\begin{align}\label{3.63'}
&(2\mu+\lambda)\pa_{x_{3}}(\tfrac{D\diver\varphi}{Dt})+2\pa_{x_3}\diver\varphi=\mathcal{R},
\end{align}
where
\begin{align*}
\mathcal{R}
&:=-\pa_tu_3+\mu[\pa_{x_1x_1} (u-\frac{\varphi}{\mu})_3+\pa_{x_2x_2} (u-\frac{\varphi}{\mu})_3-\pa_{x_1x_3} (u-\frac{\varphi}{\mu})_1-\pa_{x_2x_3} (u-\frac{\varphi}{\mu})_2]\nonumber\\
&\quad-(2\mu+\lambda)\pa_{x_{3}}[(\na u)^{T}:\na\varphi]+\pa_{x_{3}}\phi-u_3+(R_1)_3.
\end{align*}
Multiplying \eqref{3.63'} by $\pa_{x_{3}}(\tfrac{D\diver\varphi}{Dt}+\diver\varphi)$ and then integrating it over $\Omega$, we obtain
\begin{align}\label{3.64'}
&\frac{2+2\mu+\lambda}{2}\frac{d}{dt}\|\pa_{x_3}\diver\varphi\|_{L^2}^2+(2\mu+\lambda)\|\pa_{x_3}(\tfrac{D\diver \varphi}{Dt})\|_{L^2}^2+2\|\pa_{x_3}\diver\varphi\|_{L^2}^2\nonumber\\
&\quad=-(2+2\mu+\lambda)\int_{\Omega}(u\cdot\na\diver\varphi)_{x_3}\diver \varphi_{x_3}\,dx+\int_{\Omega}\pa_{x_{3}}(\tfrac{D\diver\varphi}{Dt}+\diver\varphi)\mathcal{R}\,dx\nonumber\\
&\quad:=N_1+N_2.
\end{align}
Then, we can easily estimate the right-hand side of \eqref{3.64'} as follows:
\begin{align}\label{3.65'}
N_1&\lesssim\int_{\Omega}|u_{x_3}|\cdot|\na\diver\varphi|\cdot|\diver\varphi_{x_3}|\,dx+\int_{\Omega}|\diver u|\diver \varphi_{x_3}^2\,dx\nonumber\\
&\lesssim\|\na u\|_{H^1}\|\na\diver\varphi\|_{H^1}^2\lesssim\de\|\na\diver\varphi\|_{H^1}^2
\end{align}
and
\begin{align}\label{3.66'}
N_2&\le\frac{2\mu+\la}{2}\|\pa_{x_3}(\tfrac{D\diver\varphi}{Dt})\|_{L^2}^2+\|\pa_{x_3}\diver\varphi\|_{L^2}^2+
C\|\mathcal{R}\|_{L^2}^2\nonumber\\
&\le\frac{2\mu+\la}{2}\|\pa_{x_3}(\tfrac{D\diver\varphi}{Dt})\|_{L^2}^2
+\|\pa_{x_3}\diver\varphi\|_{L^2}^2\nonumber\\
&\quad+C\|(u_{t},u,\na\phi)\|_{L^2}^2+C\|\pa\na (u-\frac{\varphi}{\mu})\|_{L^2}^2+C\|\na u\|_{L^6}^2\|\na^2\varphi\|_{L^3}^2\nonumber\\
&\quad+C\|\na\varphi\|_{L^\infty}^2\|\na^2 u\|_{L^2}^2+C\|R_1\|_{L^2}^2\nonumber\\
&\le\frac{2\mu+\la}{2}\|\pa_{x_3}(\tfrac{D\diver\varphi}{Dt})\|_{L^2}^2
+\|\pa_{x_3}\diver\varphi\|_{L^2}^2\nonumber\\
&\quad+C\|(u_{t},u,\na\phi)\|_{L^2}^2+C\|\pa\na(u,\varphi)\|_{L^2}^2+C\de\|(\na u,\na^2\varphi)\|_{H^1}^2.
\end{align}
Substituting \eqref{3.65'}--\eqref{3.66'} into \eqref{3.64'}, we obtain \eqref{3.56'}.

Next, applying $\pa^{k}\pa^{\iota}_{x_3}(k+\iota=1)$ to \eqref{3.63'}, multiplying the identity by $\pa^{k}\pa^{\iota+1}_{x_3}(\frac{D\diver\varphi}{Dt}+\diver\varphi),$ integrating over $\Omega$ by parts, as in the proof of \eqref{3.56'}, we can obtain \eqref{3.57'}.
\end{proof}

Then, we shall give the energy estimate about $\pa\na u$ in the half-space.
\begin{Lemma}\label{le3.7'}
Let $\Omega=\r3_{+}$. It holds that
\begin{align}\label{3.53'}
\frac{d}{dt}\|\pa(u,\na u,\diver u)\|_{L^2}^2+\|\pa u_{t}\|_{L^2}^2\lesssim\|(\na^2\phi,\na^3\varphi)\|_{L^2}^2+\de(\|\na u\|_{H^2}^2+\|\na^2\varphi\|_{H^1}^2).
\end{align}
\end{Lemma}
\begin{proof}
Applying $\pa$ to Eq. $\eqref{2.8}_{1}$, multiplying the resulting identity by $\pa u_{t}$, and then integrating it over $\Omega$ by parts, by \eqref{na-R_1-L2}, H\"{o}lder's and Cauchy's inequalities and Lemma \ref{le2.1}, we have for any $\va>0$,
\begin{align*}
\frac{1}{2}&\frac{d}{dt}\int_{\Omega}[\mu|\pa\na u|^2+(\mu+\lambda)|\pa\diver  u|^2+|\pa u|^2]\,dx+\int_{\Omega}|\pa u_{t}|^2\,dx\nonumber\\
&=-\int_{\Omega}\pa u_{t}\cdot\pa[\De\varphi+\na\diver\varphi-\na\phi]\,dx+\int_{\Omega}\pa u_{t}\cdot\pa R_1\,dx\nonumber\\
&\lesssim\|\pa u_{t}\|_{L^2}(\|\na^3\varphi\|_{L^2}+\|\na^2\phi\|_{L^2})+\|\pa u_{t}\|_{L^2}\|\na R_1\|_{L^2}\nonumber\\
&\lesssim(\de+\va)\|\pa u_{t}\|_{L^2}^2+\|(\na^2\phi,\na^3\varphi)\|_{L^2}^2+\de(\|\na u\|_{H^2}^2+\|\na^2\varphi\|_{H^1}^2),
\end{align*}
which, by letting $\va>0$ be small enough, since $\de$ is small, infers \eqref{3.53'}.
\end{proof}

Now, we can derive the desired higher-order dissipation estimates for $(u,\varphi)$ in the half-space.

\begin{Lemma}\label{le3.10'}
Let $\Omega=\r3_{+}$. It holds that
\begin{align}
&\frac{d}{dt}\|\na^2\varphi\|_{L^2}^2+\|\na^2(u,\varphi)\|_{L^2}^2\lesssim\|(u_{t},u,\na u,\na \varphi,\pa\na u,\pa\na\varphi)\|_{L^2}^2+\|\pa_{x_3}(\tfrac{D\diver\varphi}{Dt},\diver\varphi)\|_{L^2}^2;\label{3.80'}\\
&\frac{d}{dt}\|\na^3\varphi\|_{L^2}^2+\|\na^3(u,\varphi)\|_{L^2}^2\lesssim\|(u_{t},u,\na u,\na\varphi)\|_{H^1}^2+\|\pa\na^2(u,\varphi)\|_{L^2}^2+\|\pa_{x_3}\na(\tfrac{D\diver \varphi}{Dt},\diver\varphi)\|_{L^2}^2;\label{3.81''}\\
&\frac{d}{dt}\|\pa\na^2\varphi\|_{L^2}^2+\|\pa\na^2(u,\varphi)\|_{L^2}^2\lesssim\|(u_{t},\na u,\na\varphi)\|_{H^{1}}^2+\|(\na\phi,\pa^2\na u,\pa\pa_{x_3}(\tfrac{D\diver\varphi}{Dt}),\pa \na\diver \varphi)\|_{L^2}^2+\de\|\na^3(u,\varphi)\|_{L^2}^2.\label{3.81'''}
\end{align}
\end{Lemma}
\begin{proof}
Firstly, we can construct the following high-order dissipation estimates of $\varphi$:
\begin{align}
&\frac{d}{dt}\|\na^2\varphi\|_{L^2}^2+\|\na^2\varphi\|_{L^2}^2\lesssim\|\na^2(u-\frac{\varphi}{\mu})\|_{L^2}^2+\de\|\na u\|_{H^1}^2;\label{3.18'}\\
&\frac{d}{dt}\|\na^3\varphi\|_{L^2}^2+\|\na^3\varphi\|_{L^2}^2\lesssim\|\na^3(u-\frac{\varphi}{\mu})\|_{L^2}^2+\de\|\na u\|_{H^2}^2;\label{3.19'''}\\
&\frac{d}{dt}\|\pa\na^2\varphi\|_{L^2}^2+\|\pa\na^2\varphi\|_{L^2}^2\lesssim\|\pa\na^2(u-\frac{\varphi}{\mu})\|_{L^2}^2+\de\|\na u\|_{H^2}^2.\label{3.19''}
\end{align}
In fact, applying $\na^2$ to Eq. $\eqref{2.8}_{2}$, multiplying the resulting identity by $\na^2\varphi$ and integrating it over $\Omega$, by H\"{o}lder's and Cauchy's inequalities and Lemma \ref{le2.2}, we obtain for any $\va>0$,
\begin{align*}
\frac{1}{2}&\frac{d}{dt}\int_{\Omega}|\na^2\varphi|^2 \,dx+\frac{1}{\mu}\int_{\Omega}|\na^2\varphi|^2 \,dx\\
&=-\int_{\Omega}\na^2(u-\frac{\varphi}{\mu})\cdot\na^2\varphi \,dx-\int_{\Omega}\na^2(u\cdot\na\varphi)\cdot\na^2\varphi \,dx\\
&\lesssim\|\na^2(u-\frac{\varphi}{\mu})\|_{L^2}\|\na^2\varphi\|_{L^2}+\|\na^2(u\cdot\na\varphi)\|_{L^2} \|\na^2\varphi\|_{L^2}\nonumber\\
&\lesssim\|\na^2(u-\frac{\varphi}{\mu})\|_{L^2}\|\na^2\varphi\|_{L^2}+(\|\na^2u\|_{L^2} \|\na\varphi\|_{L^\infty}+\|u\|_{L^\infty}\|\na^3\varphi\|_{L^2})\|\na^2\varphi\|_{L^2}\nonumber\\
&\lesssim(\de+\va)\|\na^2\varphi\|_{L^2}^2+\|\na^2(u-\frac{\varphi}{\mu})\|_{L^2}^2+\de\|\na u\|_{H^1}^2,
\end{align*}
which, by letting $\va>0$ be small enough, infers \eqref{3.18'}. Similarly, we can prove \eqref{3.19'''}--\eqref{3.19''}.

Secondly, we can construct the following high-order dissipation estimates of $(u-\frac{\varphi}{\mu})$:
\begin{align}
&\|\na^{2}(u-\frac{\varphi}{\mu})\|_{L^2}^2\lesssim \|(u_{t},u,\na u,\na \varphi,\pa\na u,\pa_{x_3}(\tfrac{D\diver \varphi}{Dt}))\|_{L^2}^2+\|\diver \varphi\|_{H^{1}}^2+\de\|\na^2(u,\varphi)\|_{L^2}^2;\label{3.72'}\\
&\|\na^{3}(u-\frac{\varphi}{\mu})\|_{L^2}^2
\lesssim \|(u_{t},u,\na u)\|_{H^1}^2+\|(\na\varphi,\pa\na^2 u,\pa_{x_3}\na(\tfrac{D\diver \varphi}{Dt}))\|_{L^2}^2+\|\diver\varphi\|_{H^{2}}^2+\de(\|\na^3 u\|_{L^2}^2+\|\na^2\varphi\|_{H^{1}}^2);\label{3.73'}\\
&\|\pa\na^{2}(u-\frac{\varphi}{\mu})\|_{L^2}^2\lesssim \|(u_{t},\na u,\na\varphi)\|_{H^{1}}^2+\|(\na\phi,\pa^2\na u,\pa\pa_{x_3}(\tfrac{D\diver\varphi}{Dt}),\pa\na\diver\varphi)\|_{L^2}^2+\de\|\na^3(u,\varphi)\|_{L^2}^2.\label{3.74''}
\end{align}
In fact, by collecting $\eqref{ncu}$ and Eq. $\eqref{2.8}_{1}$, we get
\begin{align}\label{3.75'}
\begin{cases}
\diver (u-\frac{\varphi}{\mu})=-\frac{D\diver\varphi}{Dt}-\frac{1}{\mu}\diver\varphi-(\na u)^{T}:\na\varphi,\\
-\mu\De(u-\frac{\varphi}{\mu})-\na\phi=-u_{t}-u+(\mu+\lambda)\na\diver u-\na \diver\varphi+R_1,\\
(u-\frac{\varphi}{\mu})\mid_{\pa\Omega}=0.
\end{cases}
\end{align}
Applying Lemma \ref{le2.3} to the boundary-value problem \eqref{3.75'}, we obtain
\begin{align}\label{1-30-1}
\|\na^{2}(u-\tfrac{\varphi}{\mu})\|_{L^2}\lesssim \|(u_{t},u,R_1)\|_{L^2}+\|\tfrac{D\diver \varphi}{Dt}\|_{H^{1}}+\|\diver\varphi\|_{H^{1}}+\|(\na u)^{T}:\na\varphi\|_{H^{1}}+\|(\na u,\na\varphi)\|_{L^2},
\end{align}
where we have used
\begin{align*}
\|\na\diver u\|_{L^2}\le\|\na(\tfrac{D\diver\varphi}{Dt})\|_{L^2}+\|\na[(\na u)^{T}:\na\varphi]\|_{L^2}.
\end{align*}
By \eqref{ncu} and H\"{o}lder's inequality, we have
\begin{align}
&\|\tfrac{D\diver\varphi}{Dt}\|_{L^2}\le\|\diver u\|_{L^2}+\|(\na u)^{T}:\na\varphi\|_{L^2}\lesssim\|\na u\|_{L^2};\label{1-30-2}\\
&\|\na(\tfrac{D\diver\varphi}{Dt})\|_{L^2}\lesssim \|\pa(\tfrac{D\diver\varphi}{Dt})\|_{L^2}+\|\pa_{x_3}(\tfrac{D\diver \varphi}{Dt})\|_{L^2}\nonumber\\
&\qquad\qquad\quad\ \lesssim\|\pa\na u\|_{L^2}+\|\pa[(\na u)^{T}:\na\varphi]\|_{L^2}+\|\pa_{x_3}(\tfrac{D\diver \varphi}{Dt})\|_{L^2}\nonumber\\
&\qquad\qquad\quad\ \lesssim\|\pa\na u\|_{L^2}+\|\na^2u\|_{L^2}\|\na\varphi\|_{L^\infty}+\|\na u\|_{L^3}\|\na^2\varphi\|_{L^6}+\|\pa_{x_3}(\tfrac{D\diver\varphi}{Dt})\|_{L^2}\nonumber\\
&\qquad\qquad\quad\ \lesssim\|\pa\na u\|_{L^2}+\|\pa_{x_3}(\tfrac{D\diver\varphi}{Dt})\|_{L^2}+\de\|\na u\|_{H^1};\label{1-30-3}\\
&\|(\na u)^{T}:\na\varphi\|_{H^{1}}\lesssim\|\na u\|_{L^6}\|\na\varphi\|_{L^3}+\|\na^2u\|_{L^2} \|\na\varphi\|_{L^\infty}+\|\na u\|_{L^3}\|\na^2\varphi\|_{L^6}\lesssim\de\|\na u\|_{H^1}.\label{1-30-4}
\end{align}
Plugging the estimates \eqref{1-30-2}--\eqref{1-30-4} and \eqref{R_1-L2} into \eqref{1-30-1}, we obtain
\begin{align*}
\|\na^{2}(u-\tfrac{\varphi}{\mu})\|_{L^2}\lesssim \|(u_{t},u,\na u,\na\varphi,\pa\na u,\pa_{x_3}(\tfrac{D\diver \varphi}{Dt}))\|_{L^2}+\|\diver\varphi\|_{H^{1}}+\de\|\na^2(u,\varphi)\|_{L^2},
\end{align*}
which infers \eqref{3.72'}.

Similarly, we have
\begin{align}\label{3.77'}
\|\na^{3}(u-\tfrac{\varphi}{\mu})\|_{L^2}\lesssim\|(u_{t},u,R_1)\|_{H^1}+\|\tfrac{D\diver\varphi}{Dt}\|_{H^2}+\|\diver  \varphi\|_{H^2}+\|(\na u)^{T}:\na\varphi\|_{H^2}+\|(\na u,\na\varphi)\|_{L^2},
\end{align}
where we have used
\begin{align*}
\|\na\diver u\|_{H^1}\le\|\na(\tfrac{D\diver\varphi}{Dt})\|_{H^1}+\|\na[(\na u)^{T}:\na\varphi]\|_{H^1}.
\end{align*}
Similar to \eqref{1-30-2}--\eqref{1-30-4}, by Lemma \ref{le2.2}, we have
\begin{align}
&\|\tfrac{D\diver\varphi}{Dt}\|_{H^1}\lesssim \|\diver u\|_{H^1}+\|(\na u)^{T}:\na\varphi\|_{H^1}\lesssim\|\diver  u\|_{H^1}+\de\|\na u\|_{H^1}\lesssim\|\na u\|_{H^1};\label{1-30-5}\\
&\|\na^2[(\na u)^{T}:\na\varphi]\|_{L^2}\lesssim\|\na^3 u\|_{L^2}\|\na\varphi\|_{L^\infty}+\|\na u\|_{L^\infty}\|\na^3 \varphi\|_{L^2}\lesssim\de\|\na u\|_{H^2};\label{1-30-6}\\
&\|\na^2(\tfrac{D\diver\varphi}{Dt})\|_{L^2}\lesssim \|\pa\na(\tfrac{D\diver \varphi}{Dt})\|_{L^2}+\|\pa_{x_3}\na(\tfrac{D\diver\varphi}{Dt})\|_{L^2}\nonumber\\
&\qquad\qquad\quad\ \ \lesssim\|\pa\na\diver u\|_{L^2}+\|\pa\na[(\na u)^{T}:\na\varphi]\|_{L^2}+\|\pa_{x_3}\na(\tfrac{D\diver \varphi}{Dt})\|_{L^2}\nonumber\\
&\qquad\qquad\quad\ \ \lesssim\|\pa\na^2u\|_{L^2}+\|\na^2[(\na u)^{T}:\na\varphi]\|_{L^2}+\|\pa_{x_3}\na(\tfrac{D\diver \varphi}{Dt})\|_{L^2}\nonumber\\
&\qquad\qquad\quad\ \ \lesssim\|\pa\na^2u\|_{L^2}+\|\pa_{x_3}\na(\tfrac{D\diver\varphi}{Dt})\|_{L^2}+\de\|\na u\|_{H^2}.\label{1-30-7}
\end{align}
Plugging the estimates \eqref{R_1-L2}, \eqref{na-R_1-L2}, \eqref{1-30-4} and \eqref{1-30-5}--\eqref{1-30-7} into \eqref{3.77'}, we obtain
\begin{align*}
\|\na^{3}(u-\tfrac{\varphi}{\mu})\|_{L^2}\lesssim\|(u_{t},u,\na u)\|_{H^1}+\|(\na\varphi,\pa\na^2u,\pa_{x_3}\na(\tfrac{D\diver\varphi}{Dt}))\|_{L^2}+\|\diver  \varphi\|_{H^2}+\de(\|\na^3u\|_{L^2}+\|\na^2\varphi\|_{H^1}),
\end{align*}
which infers \eqref{3.73'}.

To estimate the term $\pa\na^2 u$ on the right-hand side of \eqref{3.73'}, by applying $\pa$ to Eq. $\eqref{3.75'}_2$, we study
\begin{align*}
\begin{cases}
-\mu\De[\pa(u-\frac{\varphi}{\mu})]=\pa[-u_{t}-u+(\mu+\lambda)\na\diver u-\na\diver \varphi+\na\phi+R_1],\\
\pa(u-\frac{\varphi}{\mu})\mid_{\pa\Omega}=0.
\end{cases}
\end{align*}
By Lemma \ref{le2.4}, we have
\begin{align}\label{1-31-1}
\|\pa\na^2(u-\frac{\varphi}{\mu})\|_{L^2}
&=\|\na^2\pa(u-\frac{\varphi}{\mu})\|_{L^2}\nonumber\\
&\lesssim \|(\pa u_{t},\pa u,\pa\na \phi,\pa R_1)\|_{L^2}+\|\pa\na\diver  u\|_{L^2}+\|\pa\na\diver \varphi\|_{L^2}+\|\na\pa(u-\frac{\varphi}{\mu})\|_{L^2}\nonumber\\
&\lesssim \|(\na u_{t},\na u,\na^2\phi,\na R_1)\|_{L^2}+\|\pa\na(\tfrac{D\diver\varphi}{Dt})\|_{L^2}+\|\na^2[(\na u)^{T}:\na\varphi]\|_{L^2}\nonumber\\
&\quad+\|\pa\na\diver\varphi\|_{L^2}+\|\na^2(u,\varphi)\|_{L^2}.
\end{align}
By $\eqref{2.8}_3$, \eqref{ncu} and \eqref{1-30-6}, we estimate
\begin{align}\label{1-31-2}
\|\pa\na(\tfrac{D\diver\varphi}{Dt})\|_{L^2}&\lesssim\|\pa^2(\tfrac{D\diver \varphi}{Dt})\|_{L^2}+\|\pa\pa_{x_3}(\tfrac{D\diver\varphi}{Dt})\|_{L^2}\nonumber\\
&\lesssim\|\pa^2\diver u\|_{L^2}+\|\pa^2[(\na u)^{T}:\na\varphi]\|_{L^2}+\|\pa\pa_{x_3}(\tfrac{D\diver \varphi}{Dt})\|_{L^2}\nonumber\\
&\lesssim\|\pa^2\na u\|_{L^2}+\|\pa\pa_{x_3}(\tfrac{D\diver\varphi}{Dt})\|_{L^2}+\de\|\na u\|_{H^2}
\end{align}
and
\begin{align}\label{1-31-3}
\|\na^2\phi\|_{L^2}\lesssim\|\De\phi\|_{L^2}+\|\na\phi\|_{L^2}\lesssim\|\na\varphi\|_{L^2}
+\|\na\varphi\|_{L^\infty}\|\na\varphi\|_{L^2}
+\|\na\phi\|_{L^2}\lesssim\|\na\varphi\|_{L^2}+\|\na\phi\|_{L^2}.
\end{align}
Plugging \eqref{na-R_1-L2}, \eqref{1-30-6} and \eqref{1-31-2}--\eqref{1-31-3} into \eqref{1-31-1}, we have
\begin{align*}
\|\pa\na^2(u-\frac{\varphi}{\mu})\|_{L^2}\lesssim \|(u_{t},\na u,\na\varphi)\|_{H^{1}}+\|(\na\phi,\pa^2\na u,\pa\pa_{x_3}(\tfrac{D\diver\varphi}{Dt}),\pa \na\diver\varphi)\|_{L^2}+\de\|\na^3(u,\varphi)\|_{L^2},
\end{align*}
which implies \eqref{3.74''}.

Finally, combining \eqref{3.18'}--\eqref{3.19''} with \eqref{3.72'}--\eqref{3.74''}, we can deduce the estimates \eqref{3.80'}--\eqref{3.81'''}.
\end{proof}

\section{Higher-order Energy Estimates in Exterior Domains}\label{se6}

\

In this section, we shall derive the higher-order interior and boundary estimates in an exterior domain with a compact boundary, cf. \cite{Matsumura-Nishida1983}. First, we establish the higher-order interior estimates for $(u,\varphi,\phi)$.
\begin{Lemma}\label{le3.5}
Let $\Omega\subset\r3$ be an exterior domain with a compact boundary and $\chi_{\sss 0}\in C_{0}^{\infty}(\Omega)$ be any fixed function. It holds that
\begin{align}\label{3.26}
\frac{d}{dt}&\|\chi_{\sss 0}(\na u,\na\diver \varphi,\De\varphi,\De\phi)\|_{L^2}^2
+\|\chi_{\sss 0}\na u\|_{L^2}^2+\|\chi_{\sss 0}\na^2(u,\varphi,\phi)\|_{L^2}^2\nonumber\\
&\lesssim \|\na(u,\varphi,\phi,u_t)\|_{L^2}^2+\de\|\na(u,\varphi)\|_{H^1}^2
\end{align}
and
\begin{align}\label{3.27}
\frac{d}{dt}&\|\chi_{\sss 0}\na(\na u,\na\diver\varphi,\De\varphi,\De\phi)\|_{L^2}^2+\|\chi_{\sss 0}\na^2u\|_{L^2}^2+\|\chi_{\sss 0}\na^3(u,\varphi,\phi)\|_{L^2}^2
\nonumber\\
&\lesssim \|\na(u,u_t)\|_{L^2}^2+\|\na^2(u,\varphi,\phi)\|_{L^2}^2
+\de(\|\na\varphi\|_{H^2}^2+\|\na^3u\|_{L^2}^2).
\end{align}
\end{Lemma}
\begin{proof}
{\it Step 1.} For the first two equations in \eqref{2.8}, we integrate the identity $\na(L_1-R_1):\na u\chi_{\sss 0}^2+\na(L_2-R_2):(-\na\Delta \varphi\chi_{\sss 0}^2)=0$ over $\Omega$ by parts to obtain
\begin{align}\label{3.30}
\frac{1}{2}&\frac{d}{dt}\|\chi_{\sss 0}(\na u,\na\diver\varphi,\De\varphi,\De\phi)\|_{L^2}^2+\|\chi_{\sss 0}\nabla u\|_{L^2}^2+\mu\|\chi_{\sss 0}\De u\|_{L^2}^2+(\mu+\lambda)\|\chi_{\sss 0}\na\diver u\|_{L^2}^2\nonumber\\
&=-\mu\int_{\Omega}2\chi_{\sss 0}\De u\cdot\na u\cdot\na\chi_{\sss 0}\,dx
-(\mu+\lambda)\int_{\Omega}2\chi_{\sss 0}
\left[\na\diver u\cdot\na u\cdot\na\chi_{\sss 0}-\cur u\cdot(\na\diver u\times\na\chi_{\sss 0})\right]\,dx\nonumber\\
&\quad-\int_{\Omega}\na\diver\varphi\cdot\na\diver (u\cdot\na\varphi)\chi_{\sss 0}^2\,dx
+\int_{\Omega}2\chi_{\sss 0}\left[\na\diver\varphi\cdot\na u\cdot\na\chi_{\sss 0}
-\cur u\cdot (\na\diver\varphi\times\na\chi_{\sss 0}) \right]\,dx\nonumber\\
&\quad-\int_{\Omega}2\chi_{\sss 0}\left[\na\phi\cdot\na u\cdot\na\chi_{\sss 0}-\cur u\cdot (\na\phi\times\na\chi_{\sss 0})\right] \,dx-\int_{\Omega}2\chi_{\sss 0}\De \phi_{t}\na \phi\cdot\na\chi_{\sss 0}\,dx\nonumber\\
&\quad-\int_{\Omega}\chi_{\sss 0}^2\na[O(|\na\varphi|^2)_{t}-\diver (u\cdot\na\varphi)]\cdot \na\phi \,dx-\int_{\Omega}2\chi_{\sss 0}\na \varphi_{t}\cdot\De \varphi\cdot\na\chi_{\sss 0}\,dx\nonumber\\
&\quad+\int_{\Omega}\chi_{\sss 0}^2\na R_1:\na u\,dx-\int_{\Omega}\chi_{\sss 0}^2\na R_2:\na\Delta\varphi\,dx\nonumber\\
&:=\sum_{i=1}^{10}E_{i}.
\end{align}
Now we estimate the terms $E_1$--$E_{10}$. By \eqref{R_1-L2}, \eqref{e1}, H\"{o}lder's and Cauchy's inequalities and Lemma \ref{le2.1}, we have
\begin{align*}
&E_{1}
=-\mu\int_{\Omega}2\chi_{\sss 0}\De u\cdot\na u\cdot\na\chi_{\sss 0}\,dx\lesssim\|\na u\|_{L^2}\|\chi_{\sss 0}\De u\|_{L^2};
\\
&E_{2}=-(\mu+\lambda)\int_{\Omega}2\chi_{\sss 0}\left[\na\diver u\cdot\na u\cdot\na\chi_{\sss 0}-\cur u\cdot(\na\diver u\times\na\chi_{\sss 0})\right]\,dx\lesssim\|\na u\|_{L^2}\|\chi_{\sss 0}\na\diver u\|_{L^2};
\\
&E_{3}
=-\int_{\Omega}\na\diver\varphi\cdot\na\diver(u\cdot\na\varphi)\chi_{\sss 0}^2\,dx=-\int_{\Omega}\na\diver\varphi \cdot\na[(\na u)^T:\na\varphi+u\cdot\na\diver\varphi]\chi_{\sss 0}^2\,dx\nonumber\\
&\quad=-\int_{\Omega}\na\diver\varphi \cdot\na[(\na u)^T:\na\varphi]\chi_{\sss 0}^2\,dx-\int_{\Omega}\na\diver\varphi \cdot\na u\cdot\na\diver\varphi\chi_{\sss 0}^2\,dx\nonumber\\
&\qquad-\frac12\int_{\Omega}\chi_{\sss 0}^2u\cdot\na|\na\diver\varphi|^2\,dx\nonumber\\
&\quad=-\int_{\Omega}\na\diver\varphi \cdot\na[(\na u)^T:\na\varphi]\chi_{\sss 0}^2\,dx-\int_{\Omega}\na\diver\varphi \cdot\na u\cdot\na\diver\varphi\chi_{\sss 0}^2\,dx\nonumber\\
&\qquad+\frac12\int_{\Omega}\diver (\chi_{\sss 0}^2u)|\na\diver\varphi|^2\,dx\nonumber\\
&\quad\lesssim\|\na\diver\varphi\|_{L^2}\|\na[(\na u)^T:\na\varphi]\|_{L^2}+\|\na\diver\varphi\|_{L^2}\|\na u\|_{L^6}\|\na\diver\varphi\|_{L^3}\nonumber\\
&\qquad+\|u\|_{L^\infty}\|\na\diver\varphi\|_{L^2}^2+\|\diver  u\|_{L^3}\|\na\diver\varphi\|_{L^6}\|\na\diver\varphi\|_{L^2}\nonumber\\
&\quad\lesssim\de(\|\na^2\varphi\|_{L^2}^2+\|\na u\|_{H^1}^2);
\\
&E_{4}
=\int_{\Omega}2\chi_{\sss 0}\left[\na\diver\varphi\cdot\na u\cdot\na\chi_{\sss 0}
-\cur u\cdot (\na\diver\varphi\times\na\chi_{\sss 0}) \right]\,dx\lesssim\|\na u\|_{L^2}\|\na^2\varphi\|_{L^2};
\\
&E_{5}
=-\int_{\Omega}2\chi_{\sss 0}\left[\na\phi\cdot\na u\cdot\na\chi_{\sss 0}-\cur u\cdot (\na\phi\times\na\chi_{\sss 0})\right] \,dx\lesssim\|\na u\|_{L^2}\|\na\phi\|_{L^2};
\\
&E_{6}
=-\int_{\Omega}2\chi_{\sss 0}\De \phi_{t}\na \phi\cdot\na\chi_{\sss 0} \,dx
\lesssim\|\De \phi_{t}\|_{L^2}\|\na\phi\|_{L^2}\lesssim\|\na u\|_{L^2}\|\na\phi\|_{L^2};
\\
&E_{7}
=-\int_{\Omega}\chi_{\sss 0}^2\na[O(|\na\varphi|^2)_{t}-\diver (u\cdot\na\varphi)]\cdot \na\phi \,dx\lesssim\de(\|\na u\|_{H^1}^2+\|\na\phi\|_{L^2}^2);
\\
&E_{8}
=-\int_{\Omega}2\chi_{\sss 0}\na \varphi_{t}\cdot\De\varphi\cdot\na\chi_{\sss 0}\,dx\lesssim\|\na \varphi_{t}\|_{L^2} \|\na^2\varphi\|_{L^2}\lesssim\|\na u\|_{L^2}\|\na^2\varphi\|_{L^2};
\\
&E_{9}
=\int_{\Omega}\chi_{\sss 0}^2\na R_1:\na u\,dx=-\int_{\Omega}R_1\cdot\diver (\chi_{\sss 0}^2\na u)\,dx\lesssim \|R_1\|_{L^2} \|\diver (\chi_{\sss 0}^2\na u)\|_{L^2}\nonumber\\
&\quad\lesssim \de(\|\na u\|_{H^1}+\|\na^2\varphi\|_{L^2})\|\na u\|_{H^1}\lesssim\de(\|\na u\|_{H^1}^2+\|\na^2\varphi\|_{L^2}^2);\\
&E_{10}=-\int_{\Omega}\chi_{\sss 0}^2\na R_2:\na\De\varphi\,dx=\int_{\Omega}\chi_{\sss 0}^2\na (u\cdot\na\varphi):\na\Delta\varphi\,dx\lesssim\|\na(u\cdot\na\varphi)\|_{L^2}\|\na^3\varphi\|_{L^2}\nonumber\\
&\quad\ \lesssim\de(\|\na u\|_{L^3}\|\na\varphi\|_{L^6}+\|u\|_{L^\infty}\|\na^2\varphi\|_{L^2})\lesssim\de\|\na(u,\varphi)\|_{H^1}^2.
\end{align*}
From the above estimates for $E_{1}$--$E_{10}$, we obtain
\begin{align}\label{5-23-1}
\sum_{i=1}^{10}E_{i}\lesssim\de(\|\na(u,\varphi)\|_{H^1}^2+\|\na\phi\|_{L^2}^2)+\|\na u\|_{L^2}\|(\na\phi,\na^2 \varphi,\chi_{\sss 0}\De u,\chi_{\sss 0}\na\diver u)\|_{L^2}.
\end{align}
Note that
\begin{align*}
\|\chi_{\sss 0}\na^2u\|_{L^2}^2&=\int_{\Omega}\na^2u:\na^2u\chi_{\sss 0}^2\,dx=-\int_{\Omega}\na u:\diver (\na^2u\chi_{\sss 0}^2)\,dx\\
&=-\int_{\Omega}\na u:\na\De u\chi_{\sss 0}^2\,dx-\int_{\Omega}2\chi_{\sss 0}\na u:\na^2u\cdot\na\chi_{\sss 0}\,dx\\
&=\|\chi_{\sss 0}\De u\|^2+\int_{\Omega}2\chi_{\sss 0}\na u\cdot\De u\cdot\na\chi_{\sss 0}\,dx-\int_{\Omega}2\chi_{\sss 0}\na u:\na^2u\cdot\na\chi_{\sss 0}\,dx\\
&\lesssim\|\chi_{\sss 0}\De u\|_{L^2}^2+\|\chi_{\sss 0}\De u\|_{L^2} \|\na u\|_{L^2}+\|\chi_{\sss 0}\na^2u\|_{L^2} \|\na u\|_{L^2},
\end{align*}
which, by Cauchy's inequality, infers
\begin{align}\label{5-23-2}
\|\chi_{\sss 0}\na^2u\|_{L^2}^2-\|\na u\|_{L^2}^2\lesssim\|\chi_{\sss 0}\De u\|_{L^2}^2.
\end{align}
Plugging \eqref{5-23-1}--\eqref{5-23-2} into \eqref{3.30}, by Cauchy's inequality, we deduce for any $\epsilon>0$,
\begin{align}\label{5-23-3}
\frac{d}{dt}&\|\chi_{\sss 0}(\na u,\na\diver \varphi,\De\varphi,\De\phi)\|_{L^2}^2
+\|\chi_{\sss 0}(\na u,\na^2u)\|_{L^2}^2\nonumber\\
&\lesssim \|\na u\|_{L^2}^2+\de(\|\na(u,\varphi)\|_{H^1}^2+\|\na\phi\|_{L^2}^2)+\epsilon\|(\na\phi,\na^2\varphi)\|_{L^2}^2.
\end{align}

Integrating the identity $\na^2(L_1-R_1):\na^2 u\chi_{\sss 0}^2+\na^2(L_2-R_2):(-\na^2\Delta \varphi\chi_{\sss 0}^2)=0$ over $\Omega$, by the similar arguments as above, we obtain
\begin{align}\label{6-02-3}
\frac{d}{dt}&\|\chi_{\sss 0}\na(\na u,\na\diver\varphi,\De\varphi,\De\phi)\|_{L^2}^2+\|\chi_{\sss 0}(\na^2u,\na^3u)\|_{L^2}^2\nonumber\\
&\lesssim \|\na u\|_{H^1}^2+\de(\|\na\varphi\|_{H^2}^2+\|\na^3u\|_{L^2}^2+\|\na^2\phi\|_{L^2}^2)
+\epsilon\|(\na^2\phi,\na^3\varphi)\|_{L^2}^2.
\end{align}

{\it Step 2.}
Applying $\na$ to $\eqref{2.8}_1$, multiplying the identity by $\na\varphi\chi_{\sss 0}^2$ and integrating it over $\Omega$ by parts, through H\"{o}lder's and Cauchy's inequalities and Lemma \ref{le2.1}, we obtain for any $\epsilon>0$,
\begin{align}\label{3.44}
&\|\chi_{\sss 0}(\De\varphi,\na\diver \varphi,\De\phi)\|_{L^2}^2\nonumber\\
&\quad=-\int_{\Omega}2\chi_{\sss 0}\De\varphi\cdot\na\varphi\cdot\na\chi_{\sss 0}\,dx-\int_{\Omega}2\chi_{\sss 0}\na\diver\varphi\cdot(\na\varphi\cdot\na\chi_{\sss 0}-\cur\varphi\times\na\chi_{\sss 0})\,dx\nonumber\\
&\qquad+\int_{\Omega}\De\phi O(|\na\varphi|^2)\chi_{\sss 0}^2\,dx+\int_{\Omega}2\chi_{\sss 0}\left[\na\phi\cdot\na\varphi\cdot\na\chi_{\sss 0}-\diver\varphi\na\phi\cdot\na \chi_{\sss 0}-\na\phi\cdot(\cur\varphi\times\na\chi_{\sss 0})\right] \,dx\nonumber\\
&\qquad+\int_{\Omega}\chi_{\sss 0}^2(\na u_{t}-\na u):\na\varphi\,dx+\int_{\Omega}\chi_{\sss 0}[\mu \Delta u+(\mu+\lambda) \na\diver u+R_1]\cdot(\chi_{\sss 0}\De\varphi+2\na\varphi\cdot\na\chi_{\sss 0})\,dx\nonumber\\
&\quad\lesssim\|\na\varphi\|_{L^2}\|\chi_{\sss 0}\De\varphi\|_{L^2}+\|\na\varphi\|_{L^2}\|\chi_{\sss 0}\na\diver \varphi\|_{L^2}
+\|\chi_{\sss 0}\De\phi\|_{L^2}\|\na\varphi\|_{L^2}\|\na\varphi\|_{L^\infty}
+\|(\na\phi,\na u_{t},\na u)\|_{L^2}\|\na\varphi\|_{L^2}\nonumber\\
&\qquad+(\|\chi_{\sss 0}\De u\|_{L^2}+\|\chi_{\sss 0}\na\diver u\|_{L^2}+\|\chi_{\sss 0}R_1\|_{L^2})(\|\chi_{\sss 0}\De\varphi\|_{L^2}+\|\na\varphi\|_{L^2})\nonumber\\
&\quad\lesssim\|(\na\varphi,\na\phi,\na u_{t},\na u,\chi_{\sss 0}\De u,\chi_{\sss 0}\na\diver u)\|_{L^2}^2
+(\de+\epsilon)\|(\chi_{\sss 0}\De\varphi,\chi_{\sss 0}\na\diver \varphi)\|_{L^2}^2+\de\|\chi_{\sss 0}\De\phi\|_{L^2}^2,
\end{align}
where we have made the following calculations:
\begin{align*}
\int_{\Omega}&\na\De\varphi:\na\varphi\chi_{\sss 0}^2\,dx=-\int_{\Omega}\De\varphi\cdot\diver (\na\varphi\chi_{\sss 0}^2)\,dx=-\|\chi_{\sss 0}\De\varphi\|_{L^2}^2-\int_{\Omega}2\chi_{\sss 0}\De\varphi\cdot\na\varphi\cdot\na\chi_{\sss 0}\,dx;\\
\int_{\Omega}&\na^2\diver\varphi:\na\varphi\chi_{\sss 0}^2\,dx=-\int_{\Omega}\na\diver\varphi\cdot\diver (\na\varphi\chi_{\sss 0}^2)\,dx\\
&\qquad\qquad\qquad\quad\ \ =-\int_{\Omega}\na\diver\varphi\cdot\De\varphi\chi_{\sss 0}^2\,dx-\int_{\Omega}2\chi_{\sss 0}\na\diver\varphi\cdot\na\varphi\cdot\na\chi_{\sss 0}\,dx\\
&\qquad\qquad\qquad\quad\ \ =-\int_{\Omega}\na\diver\varphi\cdot(\na\diver\varphi-\cur\cur\varphi)\chi_{\sss 0}^2\,dx-\int_{\Omega}2\chi_{\sss 0}\na\diver\varphi\cdot\na\varphi\cdot\na\chi_{\sss 0}\,dx\\
&\qquad\qquad\qquad\quad\ \ =-\|\chi_{\sss 0}\na\diver \varphi\|_{L^2}^2-\int_{\Omega}2\chi_{\sss 0}\na\diver\varphi\cdot(\na\varphi\cdot\na\chi_{\sss 0}-\cur\varphi\times\na\chi_{\sss 0})\,dx;\\
-&\int_{\Omega}\na^2\phi:\na\varphi\chi_{\sss 0}^2 \,dx=\int_{\Omega}\na\phi\cdot\diver
(\na\varphi\chi_{\sss 0}^2) \,dx\\
&\qquad\qquad\qquad\quad\ \ =\int_{\Omega}2\chi_{\sss 0}\na\phi\cdot\na\varphi\cdot\na\chi_{\sss 0} \,dx+
\int_{\Omega}\na\phi\cdot\De\varphi\chi_{\sss 0}^2 \,dx\nonumber\\
&\qquad\qquad\qquad\quad\ \ =\int_{\Omega}2\chi_{\sss 0}\na\phi\cdot\na\varphi\cdot\na\chi_{\sss 0} \,dx+
\int_{\Omega}\na\phi\cdot(\na\diver\varphi-\cur\cur\varphi)\chi_{\sss 0}^2 \,dx\nonumber\\
&\qquad\qquad\qquad\quad\ \ =-\|\chi_{\sss 0}\De\phi\|_{L^2}^2+\int_{\Omega}\De\phi O(|\na\varphi|^2)\chi_{\sss 0}^2\,dx\nonumber\\
&\qquad\qquad\qquad\quad\ \ \quad+\int_{\Omega}2\chi_{\sss 0}\left[\na\phi\cdot\na\varphi\cdot\na\chi_{\sss 0}-\diver\varphi\na\phi\cdot\na \chi_{\sss 0}-\na\phi\cdot(\cur\varphi\times\na\chi_{\sss 0})\right] \,dx.
\end{align*}
Taking $\epsilon>0$ to be small enough in \eqref{3.44}, since $\de$ is small, we deduce
\begin{align}\label{san-in-2}
\|\chi_{\sss 0}(\na^2\varphi,\na^2\phi)\|_{L^2}^2
&\lesssim\|(\na u,\na\varphi,\na\phi,\na u_t,\chi_{\sss 0}\na^2u)\|_{L^2}^2.
\end{align}
Similar to \eqref{san-in-2}, we can deduce
\begin{align}\label{20210714-2}
\|\chi_{\sss 0}(\na^3\varphi,\na^3\phi)\|_{L^2}^2\lesssim\|(\na^2\varphi,\na^2\phi,\na u_{t},\na u,\chi_{\sss 0}\na^3u)\|_{L^2}^2
+\de(\|\na^2u\|_{H^1}^2+\|\na^3\varphi\|_{L^2}^2).
\end{align}

{\it Step 3.} Let $\eta>0$ be a small but fixed constant and $\epsilon>0$ be small enough. Computing $\eqref{san-in-2}\times\eta+\eqref{5-23-3}$ and $\eqref{20210714-2}\times\eta+\eqref{6-02-3}$, respectively, we deduce \eqref{3.26} and \eqref{3.27}.
\end{proof}

Unlike the half-space case, we need to split the boundary and then flatten it locally by introducing a suitable coordinate transformation. By the finite covering theorem, there exist finitely many bounded open domains $\{\Theta_{j}\}_{j=1}^{N}\subset\mathbb{R}^3$, such that $\pa\Omega\subset\cup_{j=1}^{N}\Theta_{j}.$ The local coordinates $y=(y_{1},y_{2},y_{3})$ in each open set $\Theta_{j}$ will satisfy the conditions as below:

(1) The surface $\Theta_{j}\cap\pa\Omega$ is the image of a smooth vector function $z^{j}(y_{1},y_{2})=(z_{1}^{j},z_{2}^{j},z_{3}^{j})(y_{1},y_{2})$ (eg. take the local geodesic polar coordinate), satisfying
\begin{align*}
|z_{y_{1}}^{j}|=1,\ z_{y_{1}}^{j}\cdot z_{y_{2}}^{j}=0\quad \mbox{and}\quad |z_{y_{2}}^{j}|\ge c,
\end{align*}
where $c>0$ is a constant independent of $j$ $(j=1,2,\dots, N)$.

(2) Any $x=(x_{1},x_{2},x_{3})^T\in\Theta_{j}$ can be expressed as
\begin{align}\label{3.46}
x_{i}:=\Upsilon_{i}(y)=y_{3}\mathcal{N}_{i}^{j}(y_{1},y_{2})+z_{i}^{j}(y_{1},y_{2})\quad \mbox{for}\quad i=1,2,3,
\end{align}
where $\mathcal{N}_{i}^{j}(y_{1},y_{2})$ denotes the unit outward normal vector at the boundary point $(y_{1},y_{2},0)$.

Without causing any misunderstanding, we shall omit the superscript $j$ below. We define a set of orthonormal basis as follow:
\[e^{1}=z_{y_{1}},   \quad e^2=\frac{z_{y_{2}}}{|z_{y_{2}}|},\quad \mbox{and} \quad \mathcal{N}=e^{1}\times e^{2},\]
with $e^{1}=(e^{1}_i)$, $e^{2}=(e^{2}_i)$, and $\mathcal{N}=(\mathcal{N}_i)$, $i=1, 2, 3$.

Using the Frenet-Serret's formula (cf. \cite{do-Carmo2016}), there exist smooth functions $(\alpha_{1},\beta_{1},\gamma_{1},\alpha_{2},\beta_{2},\gamma_{2})$ of $(y_{1},y_{2})$ satisfying
\begin{equation*}
\frac{\pa}{\pa y_{1}}\left[
 \begin{array}{ccc}
    e^1_i\smallskip\\
    e^2_i\smallskip\\
    \mathcal{N}_i\\
 \end{array}
\right]=\left[
 \begin{array}{ccc}
    0&-\gamma_{1}&-\alpha_{1}\smallskip\\
    \gamma_{1}&0&-\beta_{1}\smallskip\\
    \alpha_{1}&\beta_{1}&0\\
 \end{array}
\right]\left[
 \begin{array}{ccc}
    e^1_i\smallskip\\
    e^2_i\smallskip\\
    \mathcal{N}_i\\
 \end{array}
\right],
\end{equation*}\smallskip
\begin{equation*}
\frac{\pa}{\pa y_{2}}\left[
 \begin{array}{ccc}
    e^1_i\smallskip\\
    e^2_i\smallskip\\
    \mathcal{N}_i\\
 \end{array}
\right]=\left[
 \begin{array}{ccc}
    0&-\gamma_{2}&-\alpha_{2}\smallskip\\
    \gamma_{2}&0&-\beta_{2}\smallskip\\
    \alpha_{2}&\beta_{2}&0\\
 \end{array}
\right]\left[
 \begin{array}{ccc}
    e^1_i\smallskip\\
    e^2_i\smallskip\\
    \mathcal{N}_i\\
 \end{array}
\right].
\end{equation*}
And the Jacobian $J$ of the transform \eqref{3.46} can be written as
\begin{align}\label{3.47}
J=\Upsilon_{y_{1}}\times\Upsilon_{y_{2}}\cdot\mathcal{N}=|z_{y_{2}}|+(\alpha_{1}|z_{y_{2}}|+\beta_{2})y_{3}
+(\alpha_{1}\beta_{2}-\beta_{1}\alpha_{2})y_{3}^2.
\end{align}
Letting $y_{3}$ to be small enough so that $J\ge c/2$ in \eqref{3.47}, thus the transform \eqref{3.46} is regular. In other words, the vector-valued function $\Upsilon(y):=(\Upsilon_{1},\Upsilon_{2},\Upsilon_{3})^T(y)$ is invertible. So the partial derivatives $(y_{1},y_{2},y_{3})_{x_{i}}(x)$ make sense and can be represented by
\begin{align}\label{3.48}
\begin{cases}
\pa_{x_{i}}y_{1}=\frac{1}{J}(\Upsilon_{y_{2}}\times\Upsilon_{y_{3}})_{i}=\frac{1}{J}
(\mathcal{A}e_{i}^{1}+\mathcal{B}e_{i}^{2})=:a_{1i},\smallskip\\
\pa_{x_{i}}y_{2}=\frac{1}{J}(\Upsilon_{y_{3}}\times\Upsilon_{y_{1}})_{i}=\frac{1}{J}
(\mathcal{C}e_{i}^{1}+\mathcal{D}e_{i}^{2})=:a_{2i},\smallskip\\
\pa_{x_{i}}y_{3}=\frac{1}{J}(\Upsilon_{y_{1}}\times\Upsilon_{y_{2}})_{i}=\mathcal{N}_{i}=:a_{3i},
\end{cases}
\end{align}
where $\mathcal{A}=|z_{y_{2}}|+\beta_{2}y_{3}$, $\mathcal{B}=-y_{3}\alpha_{2}$, $\mathcal{C}=-\beta_{1}y_{3}$, $\mathcal{D}=1+\alpha_{1}y_{3}$,
and $J=\mathcal{A}\mathcal{D}-\mathcal{B}\mathcal{C}\ge c/2$.

By direct calculations, we deduce from \eqref{3.48} that
\begin{align*}
\sum_{i=1}^{3}a_{1i}a_{3i}=\sum_{i=1}^{3}a_{2i}a_{3i}=0,\quad \sum_{i=1}^{3}a_{3i}^2=|\mathcal{N}|^2=1
\end{align*}
and
\begin{align*}
\pa_{x_{i}}=a_{1i}\pa_{y_1}+a_{2i}\pa_{y_2}+a_{3i}\pa_{y_3}=\frac{1}{J}
(\mathcal{A}e_{i}^{1}+\mathcal{B}e_{i}^{2})\pa_{y_1}+\frac{1}{J}
(\mathcal{C}e_{i}^{1}+\mathcal{D}e_{i}^{2})\pa_{y_2}+\mathcal{N}_{i}\pa_{y_3}.
\end{align*}
Define the material derivative
\begin{align*}
\frac{D}{Dt}:=\pa_t+u\cdot\nabla.
\end{align*}
Taking the divergence operator $\diver $ to both sides of $\eqref{2.3}$, we have
\begin{align}\label{2.7}
\frac{D\diver\varphi}{Dt}+\diver  u=-\trace(\na u\na\varphi)=-(\na u)^{T}:\na\varphi.
\end{align}
Therefore, in each $\Theta_{j}$, we can reformulate \eqref{2.7}, Eq. $\eqref{2.8}_{1}$ and $\diver \varphi$ in the local coordinates $(y_{1},y_{2},y_{3})$ as below:
\begin{align}
&\mathcal{L}_{1}:=\frac{D\diver\varphi}{Dt}+\frac{1}{J}[(\mathcal{A}e^1+\mathcal{B}e^2)\cdot u_{y_1}
+(\mathcal{C}e^1+\mathcal{D}e^2)\cdot u_{y_2}+J\mathcal{N}\cdot u_{y_3}]=\mathcal{R}_{1},\nonumber\\
&\mathcal{L}_{2}:=u_{t}+u-\frac{\mu}{J^2}[(\mathcal{A}^2+\mathcal{B}^2) u_{y_1y_1}+2(\mathcal{A}\mathcal{C}+\mathcal{B}\mathcal{D})u_{y_1y_2}
+(\mathcal{C}^2+\mathcal{D}^2)u_{y_2y_2}+J^2u_{y_3y_3}]\nonumber\\
&\qquad\quad + (\mbox{first\ order\ terms\ of}\ u)+\frac{1}{J}(\mathcal{A}e^1+\mathcal{B}e^2)\left[(\mu+\lambda)\frac{D\diver\varphi}{Dt}+\diver \varphi-\phi\right]_{y_1}\nonumber\\
&\qquad\quad +\frac{1}{J}(\mathcal{C}e^1+\mathcal{D}e^2)\left[(\mu+\lambda)\frac{D\diver\varphi}{Dt}+\diver \varphi-\phi\right]_{y_2}
+\mathcal{N}\left[(\mu+\lambda)\frac{D\diver\varphi}{Dt}+\diver \varphi-\phi\right]_{y_3}\nonumber\\
&\qquad\quad +\frac{1}{J^2}[(\mathcal{A}^2+\mathcal{B}^2) \varphi_{y_1y_1}+2(\mathcal{A}\mathcal{C}+\mathcal{B}\mathcal{D})\varphi_{y_1y_2}
+(\mathcal{C}^2+\mathcal{D}^2)\varphi_{y_2y_2}+J^2\varphi_{y_3y_3}]\nonumber\\
&\qquad\quad +(\mbox{first\ order\ terms\ of}\ \varphi)=\mathcal{R}_{2},\nonumber\\
&\diver \varphi-\frac{1}{J}[(\mathcal{A}e^1+\mathcal{B}e^2)\cdot \varphi_{y_1}
+(\mathcal{C}e^1+\mathcal{D}e^2)\cdot\varphi_{y_2}+J\mathcal{N}\cdot\varphi_{y_3}]=0,\label{3.47'}
\end{align}
where
\begin{align*}
\mathcal{R}_{1}:=-(\na u)^{T}:\na\varphi,
\end{align*}
\begin{align*}
\mathcal{R}_2&:=-u\cdot\na u-(1-\tfrac{1}{\rho})[\mu \Delta u+(\mu+\lambda) \na\diver u]-(\tfrac{P'(\rho)}{\rho}-1)\na [\diver \varphi+O(|\na\varphi|^2)]\nonumber\\
&\quad+O(|\na\varphi|)\na O(|\na\varphi|)+(\mu+\lambda)\na\mathcal{R}_{1}.
\end{align*}
We denote the tangential derivatives by $\pa=(\pa_{y_{1}},\pa_{y_{2}})$ and assume $\chi_{\sss j}\in C_{0}^{\infty}(\Theta_{j})$ be any fixed function. Then
\begin{align*}
\begin{cases}
\chi_{\sss j}\pa^{k}u\mid_{\pa\Omega_{j}^{-1}}=\chi_{\sss j}\pa^{k}\varphi\mid_{\pa\Omega_{j}^{-1}}=0,\\
\chi_{\sss j}\pa^{k}\phi\mid_{\pa\Omega_{j}^{-1}}=0\quad\mbox{or}\quad\chi_{\sss j}\pa^{k}\na\phi\cdot\nu\mid_{\pa\Omega_{j}^{-1}}=0,
\end{cases}
\end{align*}
where $0\le k\le2$ and $\Omega_{j}^{-1}:=\{y|y=\Upsilon^{-1}(x), x\in\Omega_{j}=\Theta_{j}\cap\Omega\}.$

Hereafter, the higher-order estimates near the boundary for $(u,\varphi,\phi)$ in an exterior domain with a compact boundary can be constructed as below.

\begin{Lemma}\label{le3.6}
Let $\chi_{\sss j}\in C_{0}^{\infty}(\Theta_{j})$ $(j=1,2,\dots,N)$ be any fixed function. It holds that
\begin{align}\label{3.49}
\frac{d}{dt}&\|\chi_{\sss j}\pa(u,\diver \varphi,\na\varphi,\na\phi)\|_{L^2}^2
+\|\chi_{\sss j}\pa u\|_{L^2}^2+\|\chi_{\sss j}\pa\na(u,\varphi,\phi)\|_{L^2}^2\nonumber\\
&\lesssim \|\na(u,\varphi,\phi,u_t)\|_{L^2}^2+\de\|\na(u,\varphi)\|_{H^1}^2
\end{align}
and
\begin{align}\label{3.50}
\frac{d}{dt}&\|\chi_{\sss j}\pa^2(u,\diver\varphi,\na\varphi,\na\phi)\|_{L^2}^2+\|\chi_{\sss j}\pa^2u\|_{L^2}^2
+\|\chi_{\sss j}\pa^2\na(u,\varphi,\phi)\|_{L^2}^2\nonumber\\
&\lesssim \|\na(u,u_t)\|_{L^2}^2+\|\na^2(u,\varphi,\phi)\|_{L^2}^2+\de(\|\na\varphi\|_{H^2}^2
+\|\na^3u\|_{L^2}^2).
\end{align}

\end{Lemma}
\begin{proof}
Similar to the proof of Lemma \ref{le3.5}, we omit the details.

\end{proof}

\begin{Lemma}\label{le3.8}
Let $\chi_{\sss j}\in C_{0}^{\infty}(\Theta_{j})$ $(j=1,2,\dots,N)$ be any fixed function. It holds that
\begin{align}\label{3.56}
\frac{d}{dt}&\|\chi_{\sss j}\pa_{y_3}\diver\varphi\|_{L^2}^2
+\|\chi_{\sss j}\pa_{y_3}(\tfrac{D\diver\varphi}{Dt},\diver\varphi)\|_{L^2}^2\nonumber\\
&\lesssim\|(u,\na u,\na\varphi,\na\phi,u_{t})\|_{L^2}^2+\|\chi_{\sss j}\pa\na( u,\varphi)\|_{L^2}^2+\de\|(\na u,\na^2\varphi)\|_{H^1}^2
\end{align}
and
\begin{align}\label{3.57}
\frac{d}{dt}&\|\chi_{\sss j}\pa^{\kappa}\pa_{y_3}^{\iota+1}\diver\varphi\|_{L^2}^{2}
+\|\chi_{\sss j}\pa^{\kappa}\pa_{y_3}^{\iota+1}(\tfrac{D\diver\varphi}{Dt},\diver\varphi)\|_{L^2}^{2}\nonumber\\
&\lesssim\|(u,\na u,\na \varphi,\na\phi,u_{t})\|_{H^1}^2+\|\chi_{\sss j}\pa^{\kappa+1}\pa_{y_3}^{\iota}\na (u,\varphi)\|_{L^2}^2+\de\|\na( u,\varphi)\|_{H^2}^2,
\end{align}
where $\kappa+\iota=1$.
\end{Lemma}
\begin{proof}
Taking $\pa_{y_{3}}$ to $\mathcal{L}_{1}-\mathcal{R}_{1}=0$ and $\pa_{y_{3}}$ to \eqref{3.47'}, respectively, and multiplying $\mathcal{L}_{2}-\mathcal{R}_{2}=0$ by $\mathcal{N}$, we shall obtain
\begin{align}
&\pa_{y_3}(\tfrac{D\diver\varphi}{Dt})+\tfrac{1}{J}[(\mathcal{A}e^1+\mathcal{B}e^2)\cdot u_{y_1y_3}
+(\mathcal{C}e^1+\mathcal{D}e^2)\cdot u_{y_2y_3}+J\mathcal{N}\cdot u_{y_3y_3}]+O(\na u)=(\mathcal{R}_{1})_{y_3};\smallskip\label{3.60}\\
&\mathcal{N}\cdot u_{t}+\mathcal{N}\cdot u-\tfrac{\mu}{J^2}[(\mathcal{A}^2+\mathcal{B}^2)\mathcal{N}\cdot  u_{y_1y_1}+2(\mathcal{A}\mathcal{C}+\mathcal{B}\mathcal{D})\mathcal{N}\cdot u_{y_1y_2}
+(\mathcal{C}^2+\mathcal{D}^2)\mathcal{N}\cdot u_{y_2y_2}+J^2\mathcal{N}\cdot u_{y_3y_3}]\nonumber\\
&\qquad+\tfrac{1}{J^2}[(\mathcal{A}^2+\mathcal{B}^2)\mathcal{N}\cdot  \varphi_{y_1y_1}+2(\mathcal{A}\mathcal{C}+\mathcal{B}\mathcal{D})\mathcal{N}\cdot \varphi_{y_1y_2}
+(\mathcal{C}^2+\mathcal{D}^2)\mathcal{N}\cdot \varphi_{y_2y_2}+J^2\mathcal{N}\cdot \varphi_{y_3y_3}]\nonumber\\
&\qquad+[(\mu+\lambda)\tfrac{D\diver\varphi}{Dt}+\diver \varphi-\phi]_{y_3}+O(\na u)+O(\na \varphi)=\mathcal{N}\cdot \mathcal{R}_{2};\smallskip\label{3.61}\\
&\tfrac{1}{J}[(\mathcal{A}e^1+\mathcal{B}e^2)\cdot \varphi_{y_1y_3}
+(\mathcal{C}e^1+\mathcal{D}e^2)\varphi_{y_2y_3}+J\mathcal{N}\cdot \varphi_{y_3y_3}]-\diver\varphi_{y_3}+O(\na \varphi)=0.\label{3.62}
\end{align}
To cancel the terms $\mathcal{N}\cdot u_{y_3y_3}$ and $\mathcal{N}\cdot \varphi_{y_3y_3}$ in \eqref{3.61}, we calculate $\mu\times\eqref{3.60}+\eqref{3.61}-\eqref{3.62}$ to obtain
\begin{align}\label{3.63}
&(2\mu+\lambda)\pa_{y_{3}}(\tfrac{D\diver\varphi}{Dt})+2\pa_{y_3}\diver\varphi\nonumber\\
&\quad=\tfrac{\mu}{J^2}[(\mathcal{A}^2+\mathcal{B}^2)\mathcal{N}\cdot u_{y_1y_1}+2(\mathcal{A}\mathcal{C}+\mathcal{B}\mathcal{D})\mathcal{N}\cdot u_{y_1y_2}
+(\mathcal{C}^2+\mathcal{D}^2)\mathcal{N}\cdot u_{y_2y_2}]\nonumber\\
&\qquad-\mathcal{N}\cdot u_{t}-\mathcal{N}\cdot u+\phi_{y_3}-\tfrac{\mu}{J}[(\mathcal{A}e^1+\mathcal{B}e^2)\cdot u_{y_1y_3}
+(\mathcal{C}e^1+\mathcal{D}e^2)\cdot u_{y_2y_3}]\nonumber\\
&\qquad+O(\na u)-\tfrac{1}{J^2}[(\mathcal{A}^2+\mathcal{B}^2)\mathcal{N} \cdot \varphi_{y_1y_1}+2(\mathcal{A}\mathcal{C}+\mathcal{B}\mathcal{D})\mathcal{N}\cdot \varphi_{y_1y_2}
+(\mathcal{C}^2+\mathcal{D}^2)\mathcal{N}\cdot \varphi_{y_2y_2}]\nonumber\\
&\qquad+\tfrac{1}{J}[(\mathcal{A}e^1+\mathcal{B}e^2)\cdot \varphi_{y_1y_3}
+(\mathcal{C}e^1+\mathcal{D}e^2)\cdot \varphi_{y_2y_3}]\nonumber\\
&\qquad+O(\na\varphi)+\mu(\mathcal{R}_{1})_{y_3}
+\mathcal{N}\mathcal{R}_{2}:=\mathcal{R}.
\end{align}
Multiplying \eqref{3.63} by $\chi_{\sss j}^2\pa_{y_{3}}(\tfrac{D\diver\varphi}{Dt}+\diver\varphi)$ and then integrating it over $\Omega_{j}^{-1}$, we obtain
\begin{align}\label{3.64}
&\frac{2+2\mu+\lambda}{2}\frac{d}{dt}\|\chi_{\sss j}\pa_{y_3}\diver\varphi\|_{L^2}^2
+(2\mu+\lambda)\|\chi_{\sss j}\pa_{y_3}(\tfrac{D\diver\varphi}{Dt})\|_{L^2}^2
+2\|\chi_{\sss j}\pa_{y_3}\diver\varphi\|_{L^2}^2\nonumber\\
&\quad=-(2+2\mu+\lambda)\int_{\Omega_{j}^{-1}}(u\cdot\na\diver\varphi)_{y_3}\diver\varphi_{y_3}\chi_{\sss j}^2\,dy+\int_{\Omega_{j}^{-1}}
\chi_{\sss j}^2\pa_{y_{3}}(\tfrac{D\diver\varphi}{Dt}+\diver\varphi)\mathcal{R}\,dy\nonumber\\
&\quad:=K_1+K_2.
\end{align}
Then, the right-hand side of \eqref{3.64} can be easily estimated as follows:
\begin{align}\label{3.65}
K_1
&\lesssim\int_{\Omega_{j}^{-1}}\chi_{\sss j}^2\big|u_{y_3}\cdot\na \diver\varphi\diver\varphi_{y_3}\big|\,dy
+\int_{\Omega_{j}^{-1}}\big|\diver\varphi_{y_3}^2\diver (u\chi_{\sss j}^2)\big|\,dy\nonumber\\
&\lesssim\|u\|_{H^2}\|\na \diver\varphi\|_{H^1}^2\lesssim\de\|\na \diver\varphi\|_{H^1}^2
\end{align}
and
\begin{align}\label{3.66}
K_2&\le\frac{2\mu+\la}{2}\|\chi_{\sss j}\pa_{y_3}(\tfrac{D\diver\varphi}{Dt})\|_{L^2}^2+\|\chi_{\sss j}\pa_{y_3}\diver\varphi\|_{L^2}^2+
C\|\chi_{\sss j}\mathcal{R}\|_{L^2}^2\nonumber\\
&\le\frac{2\mu+\la}{2}\|\chi_{\sss j}\pa_{y_3}(\tfrac{D\diver\varphi}{Dt})\|_{L^2}^2
+\|\chi_{\sss j}\pa_{y_3}\diver\varphi\|_{L^2}^2+C\|(u,\na u,\na\varphi,\na\phi,u_{t})\|_{L^2}^2+C\|\chi_{\sss j}(\pa\na u,\pa\na \varphi)\|_{L^2}^2\nonumber\\
&\quad+\|\na\varphi\|_{L^\infty}^2\|\na^2u\|_{L^2}^2+\|\na u\|_{L^6}^2\|\na^2\varphi\|_{L^3}^2+\|u\|_{L^6}^2\|\na u\|_{L^3}^2+\|\na\varphi\|_{L^\infty}^2\|\na^2\varphi\|_{L^2}^2\nonumber\\
&\le\frac{2\mu+\la}{2}\|\chi_{\sss j}\pa_{y_3}(\tfrac{D\diver\varphi}{Dt})\|_{L^2}^2
+\|\chi_{\sss j}\pa_{y_3}\diver\varphi\|_{L^2}^2+C\|(u,\na u,\na\varphi,\na\phi,u_{t})\|_{L^2}^2+C\|\chi_{\sss j}(\pa\na u,\pa\na \varphi)\|_{L^2}^2\nonumber\\
&\quad+\de\|(\na u,\na^2\varphi)\|_{H^1}^2.
\end{align}
Substituting \eqref{3.65}--\eqref{3.66} into \eqref{3.64}, we get \eqref{3.56}.

Applying $\pa^{k}\pa^{\iota}_{y_3}(k+\iota=1)$ to \eqref{3.63}, multiplying the identity by $\pa^{k}\pa^{\iota+1}_{y_3}(\frac{D\diver\varphi}{Dt}+\diver \varphi)\chi_{\sss j}^2,$ integrating over $\Omega_{j}^{-1}$ by parts and as in the proof of \eqref{3.56}, we can obtain \eqref{3.57}.
\end{proof}

\begin{Lemma}\label{le3.7}
Let $\chi_{\sss j}\in C_{0}^{\infty}(\Theta_{j})$ $(j=1,2,\dots,N)$ be any fixed function. It holds that
\begin{align}\label{3.53}
\frac{d}{dt}\|\chi_{\sss j}(\pa u,\pa\na u,\pa\diver u)\|_{L^2}^2+\|\chi_{\sss j}\pa u_{t}\|_{L^2}^2\lesssim\|(\na^2u,\na^3\varphi,\na^2\phi)\|_{L^2}^2+\de(\|\na u\|_{H^2}^2+\|\na^2\varphi\|_{H^1}^2).
\end{align}
\end{Lemma}
\begin{proof}
Taking the operator $\pa$ on Eq. $\eqref{2.8}_{1},$ multiplying it with $\pa u_{t}\chi_{\sss j}^2$, and then integrating the identity $\pa(L_1-R_1):\pa u_{t}\chi_{\sss j}^2=0$ over $\Omega_{j}^{-1}$, by H\"{o}lder's and Cauchy's inequalities and Lemma \ref{le2.1}, we get
\begin{align*}
\frac{1}{2}&\frac{d}{dt}\int_{\Omega_{j}^{-1}}\chi_{\sss j}^2[\mu|\pa\na u|^2+(\mu+\lambda)|\pa\diver  u|^2+|\pa u|^2]\,dy+\int_{\Omega_{j}^{-1}}|\chi_{\sss j}\pa u_{t}|^2\,dy\nonumber\\
&=-\int_{\Omega_{j}^{-1}}\chi_{\sss j}^2\pa u_{t}:\pa(\na \diver \varphi+\Delta\varphi-\na\phi)\,dy
+\int_{\Omega_{j}^{-1}}\chi_{\sss j}^2\pa u_{t}:\pa R_1\,dy\nonumber\\
&\quad-\mu\int_{\Omega_{j}^{-1}}2\chi_{\sss j}\pa\na u\cdot\pa u_{t}\cdot\na\chi_{\sss j}\,dy-(\mu+\lambda)\int_{\Omega_{j}^{-1}}2\chi_{\sss j}\pa\diver  u\cdot\pa u_{t}\cdot\na\chi_{\sss j}\,dy\nonumber\\
&\lesssim\|\chi_{\sss j}\pa u_{t}\|_{L^2}\big(\|\na^2\phi\|_{L^2}+\|\na^3\varphi\|_{L^2}+\|\na^2u\|_{L^2}\big)
+\|\chi_{\sss j}\pa u_{t}\|_{L^2}\|\na R_1\|_{L^2}\nonumber\\
&\lesssim\de(\|\chi_{\sss j}\pa u_{t}\|_{L^2}^2+\|\na u\|_{H^2}^2+\|\na^2\varphi\|_{H^1}^2)+\|\na^2\phi\|_{L^2}^2+\|\na^3\varphi\|_{L^2}^2+\|\na^2u\|_{L^2}^2.
\end{align*}
Thus, we immediately deduce \eqref{3.53} from the above inequality.
\end{proof}

Similar to the Lemma \ref{le3.10'}, we shall obtain the higher-order dissipation estimates for $(u,\varphi)$ in an exterior domain with a compact boundary.
\begin{Lemma}\label{le3.10}
Let $\chi_{\sss j}\in C_{0}^{\infty}(\Theta_{j})$ $(j=1,2,\dots,N)$ be any fixed function. It holds that
\begin{align}\label{3.80}
\frac{d}{dt}\|\na^2\varphi\|_{L^2}^2+\|\na^2(u,\varphi)\|_{L^2}^2&\lesssim\|(u_{t},u,\na u,\na\varphi,\na\phi)\|_{L^2}^2+\|\chi_{\sss 0}\na^2(u,\varphi)\|_{L^2}^2+\|\chi_{\sss j}\pa\na(u,\varphi)\|_{L^2}^2\nonumber\\
&\quad+\|\chi_{\sss j}\pa_{y_3}(\tfrac{D\diver\varphi}{Dt},\diver\varphi)\|_{L^2}^2;
\end{align}
\begin{align}\label{3.81}
\frac{d}{dt}\|\na^3\varphi\|_{L^2}^2+\|\na^3(u,\varphi)\|_{L^2}^2&\lesssim\|(u_{t},u,\na u,\na\varphi,\na\phi)\|_{H^1}^2+\|\chi_{\sss 0}\na^3(u,\varphi)\|_{L^2}^2+\|\chi_{\sss j}\pa\na^2(u,\varphi)\|_{L^2}^2\nonumber\\
&\quad+\|\chi_{\sss j}\pa_{y_3}\na(\tfrac{D\diver\varphi}{Dt},\diver\varphi)\|_{L^2}^2;
\end{align}
\begin{align}\label{3.81'}
\frac{d}{dt}\|\chi_{\sss j}\pa\na^2\varphi\|_{L^2}^2+\|\chi_{\sss j}\pa\na^{2}(u,\varphi)\|_{L^2}^2&\lesssim\|(u_{t},\na u,\na\varphi)\|_{H^{1}}^2+\|\chi_{\sss j}(\pa^2\na u,\pa\pa_{y_3}(\tfrac{D\diver\varphi}{Dt}),\pa \na\diver\varphi)\|_{L^2}^2\nonumber\\
&\quad+\|\na\phi\|_{L^2}^2+\de\|\na^3 (u,\varphi)\|_{L^2}^2.
\end{align}
\end{Lemma}
\begin{proof}
First, we shall prove the higher-order dissipation estimates of $\varphi$:
\begin{align}
&\frac{d}{dt}\|\na^2\varphi\|_{L^2}^2+\|\na^2\varphi\|_{L^2}^2
\lesssim\|\na^2(u-\frac{\varphi}{\mu})\|_{L^2}^2+\de\|u\|_{H^2}^2;\label{3.18}\\
&\frac{d}{dt}\|\na^3\varphi\|_{L^2}^2+\|\na^3\varphi\|_{L^2}^2
\lesssim\|\na^3(u-\frac{\varphi}{\mu})\|_{L^2}^2+\de\|u\|_{H^3}^2;\label{3.19}\\
&\frac{d}{dt}\|\chi_{\sss j}\pa\na^2\varphi\|_{L^2}^2+\|\chi_{\sss j}\pa\na^2\varphi\|_{L^2}^2
\lesssim\|\chi_{\sss j}\pa\na^2(u-\frac{\varphi}{\mu})\|_{L^2}^2+\de\|u\|_{H^3}^2.\label{3.19'}
\end{align}
The proofs of these are similar, we only prove one of them here. In fact, applying $\pa\na^2$ to Eq. $\eqref{2.8}_{2}$, multiplying it by $\pa\na^2\varphi\chi_{\sss j}^2$ and integrating over $\Omega_{j}^{-1}$, we have for any $\epsilon>0$,
\begin{align*}
\frac{1}{2}&\frac{d}{dt}\int_{\Omega_{j}^{-1}}|\pa\na^2\varphi\chi_{\sss j}|^2 \,dy+\frac{1}{\mu}\int_{\Omega_{j}^{-1}}|\pa\na^2\varphi\chi_{\sss j}|^2 \,dy\\
&=-\int_{\Omega_{j}^{-1}}\pa\na^2(u-\frac{\varphi}{\mu})\cdot\pa\na^2\varphi \chi_{\sss j}^2 \,dy-\int_{\Omega_{j}^{-1}}\pa\na^2(u\cdot\na\varphi)\cdot\pa\na^2\varphi\chi_{\sss j}^2 \,dy\\
&\lesssim\|\chi_{\sss j}\pa\na^2(u-\frac{\varphi}{\mu})\|_{L^2}\|\chi_{\sss j}\pa\na^2\varphi\|_{L^2}+\|\na\varphi\|_{L^\infty}
\|\na^3u\|_{L^2}\|\chi_{\sss j}\pa\na^2\varphi\|_{L^2}\nonumber\\
&\quad+\|\na^2\varphi\|_{L^6}
\|\na^2u\|_{L^3}\|\chi_{\sss j}\pa\na^2\varphi\|_{L^2}+(\|u\|_{L^\infty}+\|\na u\|_{L^\infty})\|\na^3\varphi\|_{L^2} \|\chi_{\sss j}\pa\na^2\varphi\|_{L^2}\nonumber\\
&\lesssim\de(\|\chi_{\sss j}\pa\na^2\varphi\|_{L^2}^2+\|u\|_{H^3}^2)+\epsilon\|\chi_{\sss j}\pa\na^2\varphi\|_{L^2}^2+\|\chi_{\sss j}\pa\na^2(u-\frac{\varphi}{\mu})\|_{L^2}^2.
\end{align*}
Taking the above $\epsilon>0$ to be small, since $\de$ is small, we deduce \eqref{3.19'}. Similarly we can prove \eqref{3.18}--\eqref{3.19}.

Then, we shall prove the higher-order dissipation estimates of $(u-\frac{\varphi}{\mu})$:
\begin{align}\label{3.72}
\|\na^{2}(u-\frac{\varphi}{\mu})\|_{L^2}^2
&\lesssim \|(u_{t},\na\phi,u,\na u,\na \varphi,\chi_{\sss 0}\na^2u,\chi_{\sss j}\pa\na u,\chi_{\sss j}\pa_{y_3}(\tfrac{D\diver\varphi}{Dt}))\|_{L^2}^2\nonumber\\
&\quad+\|(\chi_{\sss 0}\na^2\varphi,\chi_{\sss j}\pa\na\varphi,\chi_{\sss j}\pa_{y_3}\diver\varphi)\|_{L^2}^2+\de\|\na^2(u,\varphi)\|_{L^2}^2;
\end{align}
\begin{align}\label{3.73}
&\|\na^{3}(u-\frac{\varphi}{\mu})\|_{L^2}^2
\lesssim \|(u_{t},\na\phi,u,\na u,\na\varphi)\|_{H^1}^2+\|(\chi_{\sss 0}\na^3u,\chi_{\sss j}\pa\na^2 u,\chi_{\sss j}\pa_{y_3}\na(\tfrac{D\diver\varphi}{Dt}))\|_{L^2}^2\nonumber\\
&\qquad\qquad\qquad+\|(\chi_{\sss 0}\na^3\varphi,\chi_{\sss j}\pa\na^2\varphi,\chi_{\sss j}\pa_{y_3}\na\diver\varphi)\|_{L^2}^2+\de(\|\na^3 u\|_{L^2}^2+\|\na^2\varphi\|_{H^{1}}^2);
\end{align}
\begin{align}\label{3.74'}
&\|\chi_{\sss j}\pa\na^{2}(u-\frac{\varphi}{\mu})\|_{L^2}^2\lesssim \|(u_{t},\na u,\na\varphi)\|_{H^{1}}^2+\|\na\phi\|_{L^2}^2+\|\chi_{\sss j}(\pa^2\na u,\pa\pa_{y_3}(\tfrac{D\diver\varphi}{Dt}),\pa \na\diver\varphi)\|_{L^2}^2+\de\|\na^3(u,\varphi)\|_{L^2}^2.
\end{align}
In fact, together Eq. $\eqref{2.8}_{1}$ with $\eqref{2.7}$, we have
\begin{align}\label{3.75}
\begin{cases}
\diver (u-\frac{\varphi}{\mu})=-\frac{D\diver\varphi}{Dt}-(\na u)^{T}:\na\varphi-\frac{1}{\mu}\diver \varphi,\\
-\mu\Delta(u-\frac{\varphi}{\mu})=-u_{t}-u+(\mu+\lambda) \na\diver u-\na \diver\varphi+\na\phi+R_1,\\
(u-\frac{\varphi}{\mu})\mid_{\pa\Omega}=0.
\end{cases}
\end{align}
Applying Lemma \ref{le2.3} to the boundary-value problem \eqref{3.75}, we obtain
\begin{align}\label{3.76}
\|\na^{2}(u-\frac{\varphi}{\mu})\|_{L^2}^2
&\lesssim \|(u_{t},u,\na\phi,R_1)\|_{L^2}^2+\|\tfrac{D\diver\varphi}{Dt}\|_{H^{1}}^2+\|\diver  \varphi\|_{H^{1}}^2+\|(\na u)^{T}:\na\varphi\|_{H^{1}}^2+\|(\na u,\na\varphi)\|_{L^2}^2\nonumber\\
&\lesssim\|(u_{t},u,\na\phi,\na u,\na\varphi)\|_{L^2}^2+\|\tfrac{D\diver\varphi}{Dt}\|_{H^{1}}^2
+\|\diver \varphi\|_{H^{1}}^2+\de(\|\na^2 u\|_{L^2}^2+\|\na \varphi\|_{H^{1}}^2),
\end{align}
with the fact
\begin{align*}
\|\na\diver u\|_{L^2}\le\|\na(\tfrac{D\diver\varphi}{Dt})\|_{L^2}+\|\na[(\na u)^{T}:\na\varphi]\|_{L^2}.
\end{align*}
Note that
\begin{align}\label{3.76'}
\|\tfrac{D\diver\varphi}{Dt}\|_{L^2}\le\|\diver u\|_{L^2}+\|(\na u)^{T}:\na\varphi\|_{L^2}\lesssim\|\na u\|_{L^2};
\end{align}
\begin{align}
\|\na(\tfrac{D\diver\varphi}{Dt})\|_{L^2}&\lesssim \|\chi_{\sss 0}\na(\tfrac{D\diver\varphi}{Dt})\|_{L^2}
+\|\chi_{\sss j}\pa(\tfrac{D\diver\varphi}{Dt})\|_{L^2}+\|\chi_{\sss j}\pa_{y_3}(\tfrac{D\diver\varphi}{Dt})\|_{L^2}\nonumber\\
&\lesssim\|\chi_{\sss 0}\na^2u\|_{L^2}+\|\chi_{\sss j}\pa\na u\|_{L^2}+\|\na u\|_{H^1}\|\na^2 \varphi\|_{H^1}+\|\chi_{\sss j}\pa_{y_3}(\tfrac{D\diver\varphi}{Dt})\|_{L^2};
\end{align}
\begin{align}\label{3.76''}
\|\na\diver\varphi\|_{L^2}\lesssim\|(\chi_{\sss 0}\na^2\varphi,\chi_{\sss j}\pa\na\varphi,\chi_{\sss j}\pa_{y_3}\diver\varphi)\|_{L^2}.
\end{align}
Hence, by the estimates \eqref{3.76}--\eqref{3.76''}, we can get \eqref{3.72}. Similarly, we can obtain \eqref{3.73} as follows:
\begin{align}\label{3.77}
\|\na^{3}(u-\frac{\varphi}{\mu})\|_{L^2}^2
&\lesssim \|(u_{t},u,\na\phi,R_1)\|_{H^{1}}^2+\|\tfrac{D\diver\varphi}{Dt}\|_{H^{2}}^2
+\|\diver\varphi\|_{H^{2}}^2+\|(\na u,\na\varphi)\|_{L^2}^2+\de(\|\na u\|_{H^{2}}^2+\|\na^2\varphi\|_{H^{1}}^2)\nonumber\\
&\lesssim \|(u_{t},\na\phi,u,\na u,\na\varphi)\|_{H^1}^2+\|(\chi_{\sss 0}\na^3u,\chi_{\sss j}\pa\na^2 u,\chi_{\sss j}\pa_{y_3}\na(\tfrac{D\diver\varphi}{Dt}))\|_{L^2}^2\nonumber\\
&\quad+\|(\chi_{\sss 0}\na^3\varphi,\chi_{\sss j}\pa\na^2\varphi,\chi_{\sss j}\pa_{y_3}\na\diver\varphi)\|_{L^2}^2+\de(\|\na^3 u\|_{L^2}^2+\|\na^2\varphi\|_{H^{1}}^2),
\end{align}
where we have estimated
\begin{align*}
\|\na^2\diver\varphi\|_{L^2}\lesssim\|(\chi_{\sss 0}\na^3\varphi,\chi_{\sss j}\pa\na^2\varphi,\chi_{\sss j}\pa_{y_3}\na\diver\varphi)\|_{L^2}.
\end{align*}
In order to estimate the term $\|\chi_{\sss j}\pa\na^2 u\|^2$ on the right-hand side of \eqref{3.77}, we shall take $\chi_{\sss j}\pa$ on \eqref{3.75} to obtain
\begin{align}\label{3.78}
\begin{cases}
\diver [\chi_{\sss j}\pa(u-\frac{\varphi}{\mu})]=-\chi_{\sss j}\pa(\tfrac{D\diver\varphi}{Dt})+\na\chi_{\sss j}\cdot\pa u-\chi_{\sss j}\pa[(\na u)^{T}:\na\varphi]-\frac{1}{\mu}\diver (\chi_{\sss j}\pa\varphi),\\
-\mu\Delta[\chi_{\sss j}\pa(u-\frac{\varphi}{\mu})]=-2\mu\na\chi_{\sss j}\cdot\na[\pa(u-\frac{\varphi}{\mu})]-\mu\Delta\chi_{\sss j}\pa(u-\frac{\varphi}{\mu})\\
\qquad\qquad\qquad\qquad+\chi_{\sss j}\pa [-u_{t}-u+(\mu+\lambda) \na\diver u-\na \diver\varphi+\na\phi+R_1],\\
\chi_{\sss j}\pa (u-\frac{\varphi}{\mu})\mid_{\pa\Omega_{j}^{-1}}=0.
\end{cases}
\end{align}
Then applying Lemma \ref{le2.3} to the boundary-value problem \eqref{3.78}, we obtain
\begin{align}\label{3.79}
\|\chi_{\sss j}\pa\na^{2}(u-\frac{\varphi}{\mu})\|_{L^2}^2
&\lesssim \|\chi_{\sss j}\pa(u_{t},u,\na\phi,R_1)\|_{L^2}^2+\|\chi_{\sss j}\pa (\tfrac{D\diver\varphi}{Dt})\|_{H^1}^2+\|\na\chi_{\sss j}\cdot\pa u\|_{H^1}^2+\|\chi_{\sss j}\pa[(\na u)^{T}:\na\varphi]\|_{H^1}^2\nonumber\\
&\quad+\|\pa\na(u-\frac{\varphi}{\mu})\|_{L^2}^2+\|\pa(u-\frac{\varphi}{\mu})\|_{L^2}^2+\|\diver(\chi_{\sss j}\pa \varphi)\|_{H^1}^2+\|\chi_{\sss j}\pa\na(u,\varphi)\|_{L^2}^2\nonumber\\
&\lesssim \|(u_{t},\na u,\na \varphi)\|_{H^{1}}^2+\|\na\phi\|_{L^2}^{2}+\|\chi_{\sss j}\pa^2\na u\|_{L^2}^2+\|\chi_{\sss j}\pa\pa_{y_3}(\tfrac{D\diver\varphi}{Dt})\|_{L^2}^2+\|\chi_{\sss j}\pa \na\diver\varphi\|_{L^2}^2\nonumber\\
&\quad+\de(\|\na u\|_{H^{2}}^2+\|\na^2\varphi\|_{H^{1}}^2),
\end{align}
with the facts
\begin{align*}
\|\chi_{\sss j}\pa\na(\tfrac{D\diver\varphi}{Dt})\|_{L^2}&\lesssim\|\chi_{\sss j}\pa^2(\tfrac{D\diver\varphi}{Dt})\|_{L^2}+\|\chi_{\sss j}\pa\pa_{y_3}(\tfrac{D\diver\varphi}{Dt})\|_{L^2}\nonumber\\
&\lesssim\|\chi_{\sss j}\pa^2\na u\|_{L^2}+\|\chi_{\sss j}\pa^2 [(\na u)^{T}:\na\varphi]\|_{L^2}+\|\chi_{\sss j}\pa\pa_{y_3}(\tfrac{D\diver\varphi}{Dt})\|_{L^2}
\end{align*}
and
\begin{align*}
\|\chi_{\sss j}\pa\na\phi\|_{L^2}\le\|\na^2\phi\|_{L^2}\lesssim\|\De\phi\|_{L^2}+\|\na\phi\|_{L^2}
\lesssim\|\na\varphi\|_{L^2}+\|\na\varphi\|_{L^\infty}\|\na\varphi\|_{L^2}
+\|\na\phi\|_{L^2}\lesssim\|\na\varphi\|_{L^2}+\|\na\phi\|_{L^2}.
\end{align*}
Thus we deduce \eqref{3.74'} from \eqref{3.79}.

Finally, let $\eta>0$ be a small but fixed constant. Computing $\eqref{3.18}\times\eta+\eqref{3.72}$, $\eqref{3.19}\times\eta+\eqref{3.73}$ and $\eqref{3.19'}\times\eta+\eqref{3.74'}$, respectively, we deduce \eqref{3.80}--\eqref{3.81'}.
\end{proof}

\section{Proof of Theorem \ref{th1.2}}\label{se4}

In this section, we will establish the a priori estimates and then complete the proof of Theorem \ref{th1.2}. For clarity, we use two tables below to illustrate the energy estimates for the half-space and the exterior domain with a compact boundary.

\begin{table}[htbp]
\caption{Energy Estimates in Half-spaces}
\centering
\begin{tabular}{p{35pt}p{85pt}p{120pt}p{138pt}}
\toprule
Lemma & Energy $\mathcal{E}(t)$ & Dissipation $\mathcal{D}(t)$ & Key terms in $\mathcal{B}(t)$\\
\midrule
3.1 & $u,\na\varphi,\na\phi$ & $u,\na u$ & -- \smallskip\\
3.2 & $u_t,\na\varphi_t,\na\phi_t$ & $u_t,\na u_t$ & -- \smallskip\\
3.3 & $\varphi$ & $\na(\varphi,\phi)$ & $u,\na u$ \smallskip\\
4.1 & $\pa(u,\na\varphi,\na\phi)$ & $\pa u,\pa\na(u,\varphi,\phi)$ & $\na(u,\varphi,u_t)$ \smallskip\\
 & $\pa^2(u,\na\varphi,\na\phi)$ & $\pa^2u,\pa^2\na(u,\varphi,\phi)$ & $\na(u,u_t)$ \smallskip\\
4.2 & $\pa_{x_3}\diver\varphi$ & $\pa_{x_3}(\tfrac{D\diver\varphi}{Dt},\diver\varphi)$ & $\pa\na(u,\varphi)$\smallskip\\
 & $\pa\pa_{x_3}\diver\varphi$ & $\pa\pa_{x_3}(\tfrac{D\diver\varphi}{Dt},\diver\varphi)$ & $\pa^2\na(u,\varphi)$\smallskip\\
 & $\pa_{x_3}^{2}\diver\varphi$ & $\pa_{x_3}^{2}(\tfrac{D\diver\varphi}{Dt},\diver\varphi)$ & $\pa\pa_{x_3}\na (u,\varphi)$\smallskip\\
4.3 & $\pa(u,\na u)$ & $\pa u_t$ & $\na^2\phi,\na^3\varphi$ \smallskip\\
4.4 & $\na^2\varphi$ & $\na^2(u,\varphi)$ & $\pa_{x_3}(\tfrac{D\diver\varphi}{Dt},\diver\varphi),\pa\na( u,\varphi)$ \smallskip\\

& $\pa \na^2\varphi$ & $\pa\na^2(u,\varphi)$ & $\pa(\pa_{x_3}(\tfrac{D\diver\varphi}{Dt}),  \na\diver\varphi)$ \smallskip\\

 & $\na^3\varphi$ & $\na^3(u,\varphi)$ & $\pa_{x_3}\na(\tfrac{D\diver\varphi}{Dt},\diver\varphi),\pa\na^2(u,\varphi)$ \\
\bottomrule
\end{tabular}
\end{table}

\begin{table}[htbp]
\caption{Energy Estimates in Exterior Domains}
\centering
\begin{tabular}{p{35pt}p{85pt}p{120pt}p{138pt}}
\toprule
Lemma & Energy $\mathcal{E}(t)$ & Dissipation $\mathcal{D}(t)$ & Key terms in $\mathcal{B}(t)$\\
\midrule
3.1 & $u,\na\varphi,\na\phi$ & $u,\na u$ & -- \smallskip\\
3.2 & $u_t,\na\varphi_t,\na\phi_t$ & $u_t,\na u_t$ & -- \smallskip\\
3.3 & $\varphi$ & $\na(\varphi,\phi)$ & $u,\na u$ \smallskip\\
5.1 & $\chi_{\sss 0}(\na u,\De\varphi,\De\phi)$ & $\chi_{\sss 0}\na u,\chi_{\sss 0}\na^{2}(u,\varphi,\phi)$ & $\na(u,\varphi,\phi,u_t)$ \smallskip\\
 & $\chi_{\sss 0}(\na^2u,\na\De\varphi,\na\De\phi)$ & $\chi_{\sss 0}\na^{2}u,\chi_{\sss 0}\na^{3}(u,\varphi,\phi)$ & $\na(u,u_t),\na^2(u,\varphi,\phi)$ \smallskip\\
5.2 & $\chi_{\sss j}\pa(u,\na\varphi,\na\phi)$ & $\chi_{\sss j}\pa u,\chi_{\sss j}\pa\na( u,\varphi,\phi)$ & $\na(u,\varphi,\phi,u_t)$ \smallskip\\
 & $\chi_{\sss j}\pa^2(u,\na\varphi,\na\phi)$ & $\chi_{\sss j}\pa^{2}u,\chi_{\sss j}\pa^{2}\na( u,\varphi,\phi)$ & $\na(u,u_t),\na^2(u,\varphi,\phi)$ \smallskip\\
5.3 & $\chi_{\sss j}\pa_{y_3}\diver\varphi$ & $\chi_{\sss j}\pa_{y_3}(\tfrac{D\diver\varphi}{Dt},\diver\varphi)$ & $\chi_{\sss j}\pa\na(u,\varphi)$\smallskip\\
 & $\chi_{\sss j}\pa\pa_{y_3}\diver\varphi$ & $\chi_{\sss j}\pa\pa_{y_3}(\tfrac{D\diver\varphi}{Dt},\diver\varphi)$ & $\chi_{\sss j}\pa^2\na (u,\varphi),\na^2(u,\varphi)$\smallskip\\
 & $\chi_{\sss j}\pa_{y_3}^{2}\diver\varphi$ & $\chi_{\sss j}\pa_{y_3}^{2}(\tfrac{D\diver\varphi}{Dt},\diver\varphi)$ & $\chi_{\sss j}\pa\pa_{y_3}\na (u,\varphi),\na^2(u,\varphi)$\smallskip\\
5.4 & $\chi_{\sss j}\pa(u,\na u)$ & $\chi_{\sss j}\pa u_t$ & $\na^2(u,\phi),\na^3\varphi$ \smallskip\\
5.5 & $\na^2\varphi$ & $\na^2(u,\varphi)$ & $\chi_{\sss j}\pa_{y_3}(\tfrac{D\diver\varphi}{Dt},\diver\varphi),\chi_{\sss j}\pa\na(u,\varphi)$ \smallskip\\
& $\chi_{\sss j}\pa \na^2\varphi$ & $\chi_{\sss j}\pa \na^2(u,\varphi)$ & $\chi_{\sss j}\pa(\pa_{y_3}(\tfrac{D\diver\varphi}{Dt}),\na\diver\varphi)$ \smallskip\\
& $\na^3\varphi$ & $\na^3(u,\varphi)$ & $\chi_{\sss j}\pa_{y_3}\na(\tfrac{D\diver\varphi}{Dt},\diver\varphi),\chi_{\sss j}\pa\na^2(u,\varphi)$ \\
\bottomrule
\end{tabular}
\end{table}

It is not difficult to see that the energy estimates in Sections 3-5 can be expressed as the following form
\begin{align*}
\frac{d}{dt}\mathcal{E}(t)+\mathcal{D}(t)\lesssim \mathcal{B}(t)+\de \mathcal{S}(t),\quad \de\ll1,
\end{align*}
where $\mathcal{B}(t)$ includes some bad large terms. However, we observe from the table that the bad terms $\mathcal{B}(t)$ appeared in some row can be absorbed by the dissipation $\mathcal{D}(t)$ located in other rows after multiplying them by a small constant.
Because of the equivalence in norms between $x$-domain and $y$-domain, we omit the transformation of the domains of integration without causing confusion in an exterior domain with a compact boundary.

Let $\gamma>0$ be a suitably small constant in the below, which may vary from line to line.

\textbf{Step 1:} {\it Establish the lower-order energy estimates for $(u,\na\varphi,\na\phi)$.}

Multiplying \eqref{3.17} of Lemma \ref{le3.3} by $\gamma$, and then adding it to \eqref{3.1} of Lemma \ref{le3.1}, together with \eqref{3.9} of Lemma \ref{le3.2}, since $\de$ is small, we have
\begin{align}\label{zong2}
\frac{d}{dt}&\Big[\|(u,\diver\varphi,\na\varphi,\na\phi,u_{t},\diver\varphi_{t},\na\varphi_{t}, \na\phi_{t})\|_{L^2}^2+\gamma\int_{\Omega}(\frac12|\varphi|^2-u\cdot\varphi)\,dx\Big]\nonumber\\
&+\|(u,\nabla u,\nabla\varphi,\nabla\phi,u_t,\na u_t)\|_{L^2}^2\lesssim\de\|\na^2(u,\varphi)\|_{L^2}^2.
\end{align}

\textbf{Step 2:} {\it Construct the complete energy estimates for $(u,\na\varphi,\na\phi)$ including the estimates for the higher-order derivatives.}

Here, we need to consider the half-space case and the exterior domain case separately.

\textbf{(I)} {\it The half-space case.}

Computing $\ga\times[\eqref{3.26'}+\eqref{3.27'}]+\eqref{zong2}$, since $\de$ is small, we obtain
\begin{align}\label{'zong3}
&\frac{d}{dt}\|(u,u_{t},\pa u,\pa^2u)\|_{L^2}^2+\frac{d}{dt}\|(\diver\varphi, \na\varphi,\diver \varphi_{t}, \na\varphi_{t}, \pa\diver \varphi,\pa\na\varphi,\pa^2 \diver \varphi,\pa^2 \na\varphi)\|_{L^2}^2\nonumber\\
&\quad+\frac{d}{dt}\Big[\|(\na\phi,\na\phi_{t},\pa\na\phi,\pa^2\na\phi)\|_{L^2}^2
+\gamma\int_{\Omega}(\frac12|\varphi|^2-u\cdot\varphi)\,dx\Big]\nonumber\\
&\qquad+\|(u,\na u,\na\varphi,\na\phi,u_t,\na u_t)\|_{L^2}^2+\|(\pa u,\pa^2u,\pa\na u,\pa^2\na u,\pa\na\varphi,\pa^2\na\varphi,\pa\na\phi,\pa^2\na\phi)\|_{L^2}^2\nonumber\\
&\qquad\quad\lesssim\de\|\na^2(u,\varphi)\|_{H^1}^2.
\end{align}
Computing $\ga^2\times\eqref{3.80'}+\ga\times\eqref{3.56'}+\eqref{'zong3}$, we have
\begin{align}\label{''zong3}
&\frac{d}{dt}\|(u,u_{t},\pa u,\pa^2 u)\|_{L^2}^2+\frac{d}{dt}\|(\diver \varphi, \na\varphi,\diver \varphi_{t}, \na\varphi_{t}, \pa\diver \varphi,\pa\na\varphi,\pa^2 \diver \varphi,\pa^2 \na\varphi)\|_{L^2}^2\nonumber\\
&\quad+\frac{d}{dt}\Big[\|(\na\phi,\na\phi_{t},\pa\na\phi,\pa^2\na\phi)\|_{L^2}^2
+\gamma\int_{\Omega}(\frac12|\varphi|^2-u\cdot\varphi)\,dx+\|\pa_{x_3}\diver \varphi\|_{L^2}^2+\|\na^2\varphi\|_{L^2}^2\Big]\nonumber\\
&\qquad+\|(u,\na u,\na\varphi,\na\phi,u_t,\na u_t)\|_{L^2}^2+\|(\pa u,\pa^2u,\pa\na u,\pa^2\na u,\pa\na\varphi,\pa^2\na\varphi,\pa\na\phi,\pa^2\na\phi)\|_{L^2}^2\nonumber\\
&\qquad+\|\pa_{x_3}(\tfrac{D\diver\varphi}{Dt})\|_{L^2}^2+\|\pa_{x_3}\diver\varphi\|_{L^2}^2+\|\na^2(u,\varphi)\|_{L^2}^2
\lesssim\de\|\na^3(u,\varphi)\|_{L^2}^2.
\end{align}
In order that the terms $\|\pa\pa_{x_3}\na(u,\varphi)\|_{L^2}^2$ on the right-hand side of $\eqref{3.57'}_{\kappa=0}$ can be absorbed by the left-hand side of \eqref{3.81'''}, the terms $\|(\pa^2\na u,\pa\pa_{x_3}(\tfrac{D\diver\varphi}{Dt}),\pa \na\diver\varphi)\|_{L^2}^2$ on the right-hand side of \eqref{3.81'''} can be absorbed by the left-hand side of $\eqref{3.57'}_{\kappa=1}$, the terms $\|\pa^2\na(u,\varphi)\|_{L^2}^2$ on the right-hand side of  $\eqref{3.57'}_{\kappa=1}$ can be absorbed by the left-hand side of \eqref{''zong3}, we compute $\ga^3\times\eqref{3.57'}_{\kappa=0}+\ga^2\times\eqref{3.81'''}+\ga\times\eqref{3.57'}_{\kappa=1}+\eqref{''zong3}$ to obtain
\begin{align}\label{'''zong3}
&\frac{d}{dt}\|(u,u_{t},\pa u,\pa^2 u)\|_{L^2}^2+\frac{d}{dt}\|(\diver\varphi, \na\varphi,\diver \varphi_{t}, \na\varphi_{t}, \pa\diver \varphi,\pa\na\varphi,\pa^2 \diver \varphi,\pa^2 \na\varphi)\|_{L^2}^2\nonumber\\
&+\frac{d}{dt}\Big[\|(\na\phi,\na\phi_{t},\pa\na\phi,\pa^2\na\phi)\|_{L^2}^2
+\gamma\int_{\Omega}(\frac12|\varphi|^2-u\cdot\varphi)\,dx+\|\pa_{x_3}\diver \varphi\|_{L^2}^2+\|\na^2\varphi\|_{L^2}^2\Big]\nonumber\\
&+\frac{d}{dt}\|(\pa\na^2\varphi,\pa\pa_{x_3}\diver \varphi,\pa_{x_3}^2\diver \varphi)\|_{L^2}^2\nonumber\\
&\quad+\|(u,\na u,\na\varphi,\na\phi,u_t,\na u_t)\|_{L^2}^2+\|(\pa u,\pa^2u,\pa\na u,\pa^2\na u,\pa\na\varphi,\pa^2\na\varphi,\pa\na\phi,\pa^2\na\phi)\|_{L^2}^2\nonumber\\
&\quad+\|\pa_{x_3}(\tfrac{D\diver\varphi}{Dt})\|_{L^2}^2+\|\pa_{x_3}\diver\varphi\|_{L^2}^2+\|\na^2(u,\varphi)\|_{L^2}^2
+\|\pa\na^2(u,\varphi)\|_{L^2}^2\nonumber\\
&\quad+\|(\pa\pa_{x_3}(\tfrac{D\diver\varphi}{Dt}),\pa_{x_3}^2(\tfrac{D\diver\varphi}{Dt}),
\pa\pa_{x_3}\diver \varphi,\pa_{x_3}^2\diver \varphi)\|_{L^2}^2
\lesssim\de\|\na^3(u,\varphi)\|_{L^2}^2.
\end{align}
Multiplying \eqref{3.81''} of Lemma \ref{le3.10'} by $\gamma$, and adding it to \eqref{'''zong3} so that the terms
\begin{align*}
\|(\pa\na^2u,\pa\na^2\varphi,\pa_{x_3}\na(\tfrac{D\diver\varphi}{Dt}),\pa_{x_3}\na\diver\varphi)\|_{L^2}^2
\end{align*}
on the right-hand side of \eqref{3.81''} can be absorbed by the left-hand side of \eqref{'''zong3}, thus, we have
\begin{align}\label{''''zong3}
&\frac{d}{dt}\|(u,u_{t},\pa u,\pa^2 u)\|_{L^2}^2+\frac{d}{dt}\|(\diver \varphi, \na\varphi,\diver \varphi_{t}, \na\varphi_{t}, \pa\diver \varphi,\pa\na\varphi,\pa^2 \diver \varphi,\pa^2 \na\varphi)\|_{L^2}^2\nonumber\\
&\quad+\frac{d}{dt}\Big[\|(\na\phi,\na\phi_{t},\pa\na\phi,\pa^2\na\phi)\|_{L^2}^2
+\gamma\int_{\Omega}(\frac12|\varphi|^2-u\cdot\varphi)\,dx+\|\pa_{x_3}\diver \varphi\|_{L^2}^2+\|\na^2\varphi\|_{L^2}^2\Big]\nonumber\\
&\quad+\frac{d}{dt}\|(\pa\na^2\varphi,\pa\pa_{x_3}\diver \varphi,\pa_{x_3}^2\diver \varphi,\na^3\varphi)\|_{L^2}^2\nonumber\\
&\qquad+\|(u,\na u,\na\varphi,\na\phi,u_t,\na u_t)\|_{L^2}^2+\|(\pa u,\pa^2u,\pa\na u,\pa^2\na u,\pa\na\varphi,\pa^2\na\varphi,\pa\na\phi,\pa^2\na\phi)\|_{L^2}^2\nonumber\\
&\qquad+\|\pa_{x_3}(\tfrac{D\diver\varphi}{Dt})\|_{L^2}^2+\|\pa_{x_3}\diver\varphi\|_{L^2}^2
+\|\na^2(u,\varphi)\|_{H^1}^2+\|\pa\na^2(u,\varphi)\|_{L^2}^2\nonumber\\
&\qquad+\|(\pa\pa_{x_3}(\tfrac{D\diver\varphi}{Dt}),\pa_{x_3}^2(\tfrac{D\diver\varphi}{Dt}),
\pa\pa_{x_3}\diver \varphi,\pa_{x_3}^2\diver \varphi)\|_{L^2}^2
\le0.
\end{align}
Computing $\ga\times\eqref{3.53'}+\eqref{''''zong3}$, we obtain
\begin{align}\label{'''''zong3}
&\frac{d}{dt}\|(u,u_{t},\pa u,\pa\na u,\pa^2 u,\pa\diver u)\|_{L^2}^2+\frac{d}{dt}\|(\diver\varphi, \na\varphi,\diver\varphi_{t}, \na\varphi_{t}, \pa\diver\varphi,\pa\na\varphi,\pa^2 \diver\varphi,\pa^2 \na\varphi)\|_{L^2}^2\nonumber\\
&+\frac{d}{dt}\Big[\|(\na\phi,\na\phi_{t},\pa\na\phi,\pa^2\na\phi)\|_{L^2}^2
+\gamma\int_{\Omega}(\frac12|\varphi|^2-u\cdot\varphi)\,dx+\|\pa_{x_3}\diver \varphi\|_{L^2}^2+\|\na^2\varphi\|_{L^2}^2\Big]\nonumber\\
&+\frac{d}{dt}\|(\pa\na^2\varphi,\pa\pa_{x_3}\diver\varphi,\pa_{x_3}^2\diver\varphi,\na^3\varphi)\|_{L^2}^2\nonumber\\
&\quad+\|(u,\na u,\na\varphi,\na\phi,u_t,\pa u_t,\na u_t)\|_{L^2}^2+\|(\pa u,\pa^2u,\pa\na u,\pa^2\na u,\pa\na\varphi,\pa^2\na\varphi,\pa\na\phi,\pa^2\na\phi)\|_{L^2}^2\nonumber\\
&\quad+\|\pa_{x_3}(\tfrac{D\diver\varphi}{Dt})\|_{L^2}^2+\|\pa_{x_3}\diver\varphi\|_{L^2}^2+\|\na^2(u,\varphi)\|_{H^1}^2
+\|\pa\na^2(u,\varphi)\|_{L^2}^2\nonumber\\
&\quad+\|(\pa\pa_{x_3}(\tfrac{D\diver\varphi}{Dt}),\pa_{x_3}^2(\tfrac{D\diver\varphi}{Dt}),
\pa\pa_{x_3}\diver \varphi,\pa_{x_3}^2\diver\varphi)\|_{L^2}^2
\le0.
\end{align}
We define
\begin{align*}
\mathcal{W}(t):=\|(u,u_{t},\pa\na u)(t)\|_{L^2}^2+\|(\na\varphi,\na\varphi_{t}, \na^2\varphi,\na^3\varphi)(t)\|_{L^2}^2+\|\na(\phi,\phi_{t})(t)\|_{L^2}^2.
\end{align*}
Then \eqref{'''''zong3} implies
\begin{align}\label{''''''zong3}
&\mathcal{W}(t)+\|\varphi\|_{L^2}^2
+C\int_{0}^{t}(\mathcal{W}(\tau)+\|\pa_{x_3}^2u\|_{L^2}^2+\|\na^3u\|_{L^2}^2)\,d\tau\lesssim \mathcal{W}(0)+\|\varphi_0\|_{L^2}^2.
\end{align}
Thus, we easily check that
\begin{align}\label{4.6'}
\mathcal{W}(t)\sim\|u\|_{L^2}^2+\|\pa\na u\|_{L^2}^2+\|\na\varphi\|_{H^2}^2+\|\na\phi\|_{H^2}^2+\|(u_{t},\na\varphi_{t},\na\phi_{t})\|_{L^2}^2.
\end{align}
By Eq. $\eqref{2.8}_{1}$, we easily estimate
\begin{align}\label{4.7'}
\|\pa_{x_3}^2u\|_{L^2}^2\lesssim\|(\pa\na u,u_{t},u,\na \diver \varphi,\De\varphi,\na\phi,R_1)\|_{L^2}^2\lesssim \mathcal{W}(t).
\end{align}
Combining \eqref{''''''zong3}--\eqref{4.7'} with \eqref{equivalent} and Eq. $\eqref{2.8}_3$, by Lemma \ref{le2.1}, there exists a functional $\mathcal{H}(t)$ satisfying
\begin{align*}
\mathcal{H}(t)\sim \|(\rho-1,u,\mathbb{F}-\mathbb{I})\|_{H^2}^2+\|\na\phi\|_{H^3}^2
+\|(\rho_{t},u_{t},\mathbb{F}_t,\na\phi_{t})\|_{L^2}^2
\end{align*}
such that
\begin{align*}
\mathcal{H}(t)+\|\varphi\|_{L^2}^2+C\int_{0}^{t}(\mathcal{H}(\tau)+\|\na^3u\|_{L^2}^2)\,d\tau\lesssim \mathcal{H}(0)
\lesssim \|(\rho_{0}-1, u_{0},\mathbb{F}_{0}-\mathbb{I})\|_{H^2}^2+\|\varphi_{0}\|_{L^2}^2.
\end{align*}

\textbf{(II)} {\it The exterior domain case.}

Computing $\ga\times[\eqref{3.26}+\eqref{3.27}]+\eqref{zong2}$, since $\de$ is small, we obtain
\begin{align}\label{zong3}
&\frac{d}{dt}\|(u,u_{t},\chi_{\sss 0}\na u,\chi_{\sss 0}\na^2 u)\|_{L^2}^2+\frac{d}{dt}\|(\diver \varphi, \na\varphi,\diver \varphi_{t}, \na\varphi_{t}, \chi_{\sss 0}\na\diver \varphi,\chi_{\sss 0}\De\varphi,\chi_{\sss 0}\na^2 \diver \varphi,\chi_{\sss 0}\na\De\varphi)\|_{L^2}^2\nonumber\\
&\quad+\frac{d}{dt}\Big[\|(\na\phi,\na\phi_{t},\chi_{\sss 0}\De\phi,\chi_{\sss 0}\na\De\phi)\|_{L^2}^2
+\gamma\int_{\Omega}(\frac12|\varphi|^2-u\cdot\varphi)\,dx\Big]\nonumber\\
&\qquad+\|(u,\na u,\na\varphi,\na\phi,u_t,\na u_t)\|_{L^2}^2+\|\chi_{\sss 0}(\na u,\na^2u,\na^3u,\na^2\varphi,\na^2\phi,\na^3\varphi,\na^3\phi)\|_{L^2}^2\nonumber\\
&\qquad\quad\lesssim\ga\|\na^2(u,\varphi)\|_{L^2}^2+\de\|\na^3(u,\varphi)\|_{L^2}^2.
\end{align}
Computing $\ga\times[\eqref{3.49}+\eqref{3.50}]+\eqref{zong3}$, we obtain
\begin{align}\label{zong4}
&\frac{d}{dt}\|(u,u_{t},\chi_{\sss 0}\na u,\chi_{\sss 0}\na^2 u)\|_{L^2}^2+\frac{d}{dt}\|(\diver \varphi, \na\varphi,\diver \varphi_{t}, \na\varphi_{t}, \chi_{\sss 0}\na\diver \varphi,\chi_{\sss 0}\De\varphi,\chi_{\sss 0}\na^2 \diver \varphi,\chi_{\sss 0}\na\De\varphi)\|_{L^2}^2\nonumber\\
&\quad+\frac{d}{dt}\Big[\|(\na\phi,\na\phi_{t},\chi_{\sss 0}\De\phi,\chi_{\sss 0}\na\De\phi)\|_{L^2}^2
+\gamma\int_{\Omega}(\frac12|\varphi|^2-u\cdot\varphi)\,dx\Big]\nonumber\\
&\quad+\frac{d}{dt}\|\chi_{\sss j}(\pa u,\pa^2 u,\pa \diver \varphi,\pa^2 \diver \varphi,\pa\na\varphi,\pa^2\na\varphi,\pa\na\phi,\pa^2\na\phi)\|_{L^2}^2\nonumber\\
&\qquad+\|(u,\na u,\na\varphi,\na\phi,u_t,\na u_t)\|_{L^2}^2+\|\chi_{\sss 0}(\na u,\na^2u,\na^3u,\na^2\varphi,\na^2\phi,\na^3\varphi,\na^3\phi)\|_{L^2}^2\nonumber\\
&\qquad+\|\chi_{\sss j}(\pa u,\pa\na u,\pa^2 u,\pa^2\na u,\pa\na\varphi,\pa^2\na\varphi,\pa\na\phi,\pa^2\na\phi)\|_{L^2}^2\nonumber\\
&\qquad\quad\lesssim\ga\|\na^2(u,\varphi)\|_{L^2}^2+\de\|\na^3(u,\varphi)\|_{L^2}^2.
\end{align}
Computing $\ga^{\frac{1}{2}}\times\eqref{3.80}+\ga^{\frac{1}{4}}\times\eqref{3.56}+\eqref{zong4}$, we obtain
\begin{align}\label{zong5}
&\frac{d}{dt}\|(u,u_{t},\chi_{\sss 0}\na u,\chi_{\sss 0}\na^2 u)\|_{L^2}^2+\frac{d}{dt}\|(\diver \varphi, \na\varphi,\diver \varphi_{t}, \na\varphi_{t}, \chi_{\sss 0}\na\diver \varphi,\chi_{\sss 0}\De\varphi,\chi_{\sss 0}\na^2 \diver \varphi,\chi_{\sss 0}\na\De\varphi,\na^2\varphi)\|_{L^2}^2\nonumber\\
&\quad+\frac{d}{dt}\Big[\|(\na\phi,\na\phi_{t},\chi_{\sss 0}\De\phi,\chi_{\sss 0}\na\De\phi)\|_{L^2}^2
+\gamma\int_{\Omega}(\frac12|\varphi|^2-u\cdot\varphi)\,dx\Big]\nonumber\\
&\quad+\frac{d}{dt}\|\chi_{\sss j}(\pa u,\pa^2 u,\pa \diver \varphi,\pa^2 \diver \varphi,\pa\na\varphi,\pa^2\na\varphi,\pa\na\phi,\pa^2\na\phi)\|_{L^2}^2+\frac{d}{dt}\|\chi_{\sss j}\pa_{y_3}\diver\varphi\|_{L^2}^2\nonumber\\
&\qquad+\|(u,\na u,\na\varphi,\na\phi,u_t,\na u_t)\|_{L^2}^2+\|\chi_{\sss 0}(\na u,\na^2u,\na^3u,\na^2\varphi,\na^2\phi,\na^3\varphi,\na^3\phi)\|_{L^2}^2\nonumber\\
&\qquad+\|\chi_{\sss j}(\pa u,\pa\na u,\pa^2 u,\pa^2\na u,\pa\na\varphi,\pa^2\na\varphi,\pa\na\phi,\pa^2\na\phi)\|_{L^2}^2\nonumber\\
&\qquad+\|\chi_{\sss j}\pa_{y_3}(\tfrac{D\diver\varphi}{Dt})\|_{L^2}^2
+\|\chi_{\sss j}\pa_{y_3}\diver\varphi\|_{L^2}^2+\|\na^2(u,\varphi)\|_{L^2}^2\lesssim\de\|\na^3(u,\varphi)\|_{L^2}^2.
\end{align}
In order that the terms $\|\chi_{\sss j}\pa\pa_{y_3}\na(u,\varphi)\|_{L^2}^2$ on the right-hand side of $\eqref{3.57}_{\kappa=0}$ can be absorbed by the left-hand side of \eqref{3.81'}, the terms $\|\chi_{\sss j}\pa(\pa_{y_3}(\tfrac{D\diver\varphi}{Dt}),\na\diver\varphi)\|_{L^2}^2$ on the right-hand side of \eqref{3.81'} can be absorbed by the left-hand side of $\eqref{3.57}_{\kappa=1}$, the terms $\|\chi_{\sss j}\pa^2\na(u,\varphi)\|_{L^2}^2$ on the right-hand side of $\eqref{3.57}_{\kappa=1}$ can be absorbed by the left-hand side of \eqref{zong5}, we compute $\ga^3\times\eqref{3.57}_{\kappa=0}+\ga^2\times\eqref{3.81'}+\ga\times\eqref{3.57}_{\kappa=1}+\eqref{zong5}$ to have
\begin{align}\label{zong6}
&\frac{d}{dt}\|(u,u_{t},\chi_{\sss 0}\na u,\chi_{\sss 0}\na^2 u)\|_{L^2}^2+\frac{d}{dt}\|(\diver \varphi, \na\varphi,\diver \varphi_{t}, \na\varphi_{t}, \chi_{\sss 0}\na\diver \varphi,\chi_{\sss 0}\De\varphi,\chi_{\sss 0}\na^2 \diver \varphi,\chi_{\sss 0}\na\De\varphi,\na^2\varphi)\|_{L^2}^2\nonumber\\
&\quad+\frac{d}{dt}\Big[\|(\na\phi,\na\phi_{t},\chi_{\sss 0}\De\phi,\chi_{\sss 0}\na\De\phi)\|_{L^2}^2
+\gamma\int_{\Omega}(\frac12|\varphi|^2-u\cdot\varphi)\,dx\Big]\nonumber\\
&\quad+\frac{d}{dt}\|\chi_{\sss j}(\pa u,\pa^2 u,\pa \diver \varphi,\pa^2 \diver \varphi,\pa\na\varphi,\pa^2\na\varphi,\pa\na\phi,\pa^2\na\phi)\|_{L^2}^2\nonumber\\
&\quad+\frac{d}{dt}\|\chi_{\sss j}[\pa_{y_3}\diver\varphi,\pa\pa_{y_3}\diver\varphi,\pa_{y_3}^{2}\diver\varphi,\pa\na^2\varphi]\|_{L^2}^{2}\nonumber\\
&\qquad+\|(u,\na u,\na\varphi,\na\phi,u_t,\na u_t)\|_{L^2}^2+\|\chi_{\sss 0}(\na u,\na^2u,\na^3u,\na^2\varphi,\na^2\phi,\na^3\varphi,\na^3\phi)\|_{L^2}^2\nonumber\\
&\qquad+\|\chi_{\sss j}(\pa u,\pa\na u,\pa^2 u,\pa^2\na u,\pa\na\varphi,\pa^2\na\varphi,\pa\na\phi,\pa^2\na\phi)\|_{L^2}^2\nonumber\\
&\qquad+\|\chi_{\sss j}\pa_{y_3}(\tfrac{D\diver\varphi}{Dt})\|_{L^2}^2
+\|\chi_{\sss j}\pa_{y_3}\diver\varphi\|_{L^2}^2+\|\na^2(u,\varphi)\|_{L^2}^2\nonumber\\
&\qquad+\|\chi_{\sss j}\partial\partial_{y_3}(\tfrac{D\diver\varphi}{Dt})\|_{L^2}^{2}
+\|\chi_{\sss j}\partial\partial_{y_3}\diver\varphi\|_{L^2}^{2}+\|\chi_{\sss j}\partial\nabla^{2}(u,\varphi)\|_{L^2}^2\nonumber\\
&\qquad+\|\chi_{\sss j}\partial_{y_3}^{2}(\tfrac{D\diver\varphi}{Dt})\|_{L^2}^{2}
+\|\chi_{\sss j}\partial_{y_3}^{2}\diver\varphi\|_{L^2}^{2}\lesssim\de\|\na^3(u,\varphi)\|_{L^2}^2.
\end{align}
Computing $\ga\times\eqref{3.81}+\eqref{zong6}$, since $\de$ is small, we obtain
\begin{align}\label{zong7}
&\frac{d}{dt}\|(u,u_{t},\chi_{\sss 0}\na u,\chi_{\sss 0}\na^2 u)\|_{L^2}^2+\frac{d}{dt}\|(\diver \varphi, \na\varphi,\diver \varphi_{t}, \na\varphi_{t}, \chi_{\sss 0}\na\diver \varphi,\chi_{\sss 0}\De\varphi,\chi_{\sss 0}\na^2 \diver \varphi,\chi_{\sss 0}\na\De\varphi,\na^2\varphi,\na^3\varphi)\|_{L^2}^2\nonumber\\
&\quad+\frac{d}{dt}\Big[\|(\na\phi,\na\phi_{t},\chi_{\sss 0}\De\phi,\chi_{\sss 0}\na\De\phi)\|_{L^2}^2
+\gamma\int_{\Omega}(\frac12|\varphi|^2-u\cdot\varphi)\,dx\Big]\nonumber\\
&\quad+\frac{d}{dt}\|\chi_{\sss j}(\pa u,\pa^2 u,\pa \diver \varphi,\pa^2 \diver \varphi,\pa\na\varphi,\pa^2\na\varphi,\pa\na\phi,\pa^2\na\phi)\|_{L^2}^2\nonumber\\
&\quad+\frac{d}{dt}\|\chi_{\sss j}[\pa_{y_3}\diver\varphi,\pa\pa_{y_3}\diver\varphi,\pa_{y_3}^{2}\diver\varphi,\pa\na^2\varphi]\|_{L^2}^{2}\nonumber\\
&\qquad+\|(u,\na u,\na\varphi,\na\phi,u_t,\na u_t)\|_{L^2}^2+\|\chi_{\sss 0}(\na u,\na^2u,\na^3u,\na^2\varphi,\na^2\phi,\na^3\varphi,\na^3\phi)\|_{L^2}^2\nonumber\\
&\qquad+\|\chi_{\sss j}(\pa u,\pa\na u,\pa^2 u,\pa^2\na u,\pa\na\varphi,\pa^2\na\varphi,\pa\na\phi,\pa^2\na\phi)\|_{L^2}^2\nonumber\\
&\qquad+\|\chi_{\sss j}\pa_{y_3}(\tfrac{D\diver\varphi}{Dt})\|_{L^2}^2
+\|\chi_{\sss j}\pa_{y_3}\diver\varphi\|_{L^2}^2+\|\na^2(u,\varphi)\|_{L^2}^2+\|\na^3(u,\varphi)\|_{L^2}^2\nonumber\\
&\qquad+\|\chi_{\sss j}\partial\partial_{y_3}(\tfrac{D\diver\varphi}{Dt})\|_{L^2}^{2}
+\|\chi_{\sss j}\partial\partial_{y_3}\diver\varphi\|_{L^2}^{2}+\|\chi_{\sss j}\partial\nabla^{2}(u,\varphi)\|_{L^2}^2\nonumber\\
&\qquad+\|\chi_{\sss j}\partial_{y_3}^{2}(\tfrac{D\diver\varphi}{Dt})\|_{L^2}^{2}
+\|\chi_{\sss j}\partial_{y_3}^{2}\diver\varphi\|_{L^2}^{2}\le0.
\end{align}
Computing $\ga\times\eqref{3.53}+\eqref{zong7}$, we get
\begin{align}\label{zong8}
\frac{d}{dt}&\|(u,u_{t})\|_{L^2}^2+\frac{d}{dt}\|\chi_{\sss 0}(\na u,\na^2 u)\|_{L^2}^2+\frac{d}{dt}\|\chi_{\sss j}(\partial u,\partial^2 u,\partial\na u,\partial\diver u)\|_{L^2}^2\nonumber\\
&+\frac{d}{dt}\|(\diver\varphi, \na\varphi,\na^2\varphi,\na^3\varphi, \diver\varphi_{t}, \na\varphi_{t})\|_{L^2}^2+\frac{d}{dt}\|\chi_{\sss 0}(\na\diver \varphi,\De\varphi,\nabla^2 \diver \varphi,\nabla\De\varphi)\|_{L^2}^2\nonumber\\
&+\frac{d}{dt}\|\chi_{\sss j}[\pa_{y_3}\diver\varphi,\pa\diver\varphi,\pa\na\varphi,\pa^2 \diver\varphi,\pa^2\na\varphi,\pa\na^2\varphi,\pa\pa_{y_3}\diver\varphi,\pa_{y_3}^{2}\diver\varphi]\|_{L^2}^2\nonumber\\
&+\frac{d}{dt}\|(\nabla\phi,\nabla\phi_{t})\|_{L^2}^2+\frac{d}{dt}\|\chi_{\sss 0}(\De\phi,\nabla\De\phi)\|_{L^2}^2+\frac{d}{dt}\|\chi_{\sss j}(\partial\na\phi,\partial^2\na\phi)\|_{L^2}^2
+\frac{d}{dt}\Big[\gamma\int_{\Omega}(\frac12|\varphi|^2-u\cdot\varphi)\,dx\Big]\nonumber\\
&\quad+\|(u,\na u,\na^2u,\na^3u,u_t,\na u_t)\|_{L^2}^2+\|\chi_{\sss 0}(\na^2u,\na^3u)\|_{L^2}^2+\|\chi_{\sss j}(\pa u,\pa\na u,\pa^2u,\pa^2\na u,\partial\nabla^{2}u,\partial u_{t})\|_{L^2}^2\nonumber\\
&\quad+\|(\nabla\varphi,\nabla^2\varphi,\nabla^3\varphi)\|_{L^2}^2+\|\chi_{\sss 0}(\na^2\varphi,\nabla^3\varphi)\|_{L^2}^2\nonumber\\
&\quad+\|\chi_{\sss j}(\pa\nabla\varphi,\pa^2\nabla\varphi,\pa\nabla^2\varphi,\pa_{y_3}\diver\varphi,
\partial\partial_{y_3}\diver\varphi,\partial_{y_3}^{2}\diver\varphi)\|_{L^2}^2\nonumber\\
&\quad+\|\nabla\phi\|_{L^2}^2+\|\chi_{\sss 0}(\na^2\phi,\nabla^3\phi)\|_{L^2}^2+\|\chi_{\sss j}(\partial\nabla\phi,\partial^2\nabla\phi)\|_{L^2}^2\nonumber\\
&\quad+\|\chi_{\sss j}[\pa_{y_3}(\tfrac{D\diver\varphi}{Dt}),\partial\partial_{y_3}(\tfrac{D\diver\varphi}{Dt}),
\partial_{y_3}^{2}(\tfrac{D\diver\varphi}{Dt})]\|_{L^2}^{2}
\le0.
\end{align}
We define
\begin{align*}
\mathcal{Y}(t):=\|(u,u_{t},\chi_{\sss j}\pa\na u,\chi_{\sss 0}\na^2u)(t)\|_{L^2}^2+\|(\na\varphi,\na\varphi_{t}, \na^2\varphi,\na^3\varphi)(t)\|_{L^2}^2+\|\na(\phi,\phi_{t})(t)\|_{L^2}^2.
\end{align*}
Then \eqref{zong8} implies
\begin{align}\label{zong9}
&\mathcal{Y}(t)+\|\varphi\|_{L^2}^2+C\int_{0}^{t}(\mathcal{Y}(\tau)+\|\chi_{\sss j}\pa_{y_{3}}^2u\|_{L^2}^2+\|\na^3u\|_{L^2}^2)\,d\tau\lesssim \mathcal{Y}(0)+\|\varphi_0\|_{L^2}^2.
\end{align}
Next, we easily check that
\begin{align}\label{4.6}
\mathcal{Y}(t)\sim\|u\|_{L^2}^2+\|\chi_{\sss 0}\na^2u\|_{L^2}^2+\|\chi_{\sss j}\pa\na u\|_{L^2}^2+\|\na\varphi\|_{H^2}^2+\|\na\phi\|_{H^2}^2+\|(u_{t},\na\varphi_{t},\na\phi_{t})\|_{L^2}^2.
\end{align}
By Eq. $\eqref{2.8}_{1}$, we easily estimate
\begin{align}\label{4.7}
\|\chi_{\sss j}\pa_{y_{3}}^2u\|_{L^2}^2\lesssim\|(\chi_{\sss j}\pa\na u,\chi_{\sss 0}\na^2u,u_{t},u,\na \diver \varphi,\De\varphi,\na\phi,R_1)\|_{L^2}^2\lesssim \mathcal{Y}(t).
\end{align}
Combining \eqref{zong9}--\eqref{4.7} with \eqref{equivalent} and Eq. $\eqref{2.8}_3$, by Lemma \ref{le2.1}, there exists a functional $\mathcal{H}(t)$ satisfying
\begin{align*}
\mathcal{H}(t)\sim \|(\rho-1,u,\mathbb{F}-\mathbb{I})\|_{H^2}^2+\|\na\phi\|_{H^2}^2
+\|(\rho_{t},u_{t},\mathbb{F}_t,\na\phi_{t})\|_{L^2}^2
\end{align*}
such that
\begin{align*}
\mathcal{H}(t)+\|\varphi\|_{L^2}^2+C\int_{0}^{t}(\mathcal{H}(\tau)+\|\na^3u\|_{L^2}^2)\,d\tau\lesssim \mathcal{H}(0)
\lesssim \|(\rho_{0}-1, u_{0},\mathbb{F}_{0}-\mathbb{I})\|_{H^2}^2+\|\varphi_{0}\|_{L^2}^2.
\end{align*}
\textbf{Step 3:} From the above two steps, we have proved the following a priori estimates:
\begin{Proposition}[A priori estimates]\label{apriori}
Let $\Omega\subset\r3$ be a half-space or an exterior domain with a compact boundary $\pa\Omega\in\mathcal{C}^3$. Let $T>0$. Assume that for sufficiently small $\de>0$,
\begin{align*}
\sup_{0\le t\le T}\big[\|(\rho-1, u, \mathbb{F}-\mathbb{I})(t)\|_{H^2}+\|\De^{-1}\diver \mathbb{F}^{-1}(t)\|_{L^2}\big]<\de.
\end{align*}
Then we have for any $t\in[0,T]$,
\begin{align*}
&\|(\rho-1, u, \mathbb{F}-\mathbb{I})(t)\|_{H^2}+ \|\na\phi(t)\|_{H^3}+\|(\rho_{t},u_{t},\mathbb{F}_t,\na\phi_t)(t)\|_{L^2}
+\|\varphi(t)\|_{L^2}\\
&\quad\le C_1(\|(\rho_{0}-1, u_{0}, \mathbb{F}_{0}-\mathbb{I})\|_{H^2}+\|\varphi_{0}\|_{L^2})
\end{align*}
if $\rho_+=1$ and
\begin{align*}
&\|(\rho-1, u, \mathbb{F}-\mathbb{I})(t)\|_{H^2}+\|\phi(t)\|_{H^4}
+\|(\rho_{t},u_{t},\mathbb{F}_t,\phi_t,\na\phi_t)(t)\|_{L^2}+\|\varphi(t)\|_{L^2}\\
&\quad\le C_1 (\|(\rho_{0}-1, u_{0}, \mathbb{F}_{0}-\mathbb{I})\|_{H^2}+\|\varphi_{0}\|_{L^2})
\end{align*}
if $\rho_{+}=e^{-\phi}$. Here $C_1>1$ is some fixed constant.
\end{Proposition}
\textbf{Step 4:} For completeness, we list the local existence and uniqueness of the strong solution of the problem \eqref{1.1}--\eqref{1.2} and omit its proof, cf. \cite{Kagei-Kawashima2006}.

\begin{Proposition}\label{th1.1}
Let $\Omega\subset\r3$ be a half-space or an exterior domain with a compact boundary $\pa\Omega\in\mathcal{C}^3$. Assume that $(\rho_{0}-1, u_{0}, \mathbb{F}_{0}-\mathbb{I})\in H^2(\Omega)$ and $\varphi_0:=\De^{-1}\diver\mathbb{F}_0^{-1}\in L^2(\Omega)$ satisfying $\inf_{x\in\overline{\Omega}}\{\rho_0(x)\}>0$ and the compatible conditions given in \eqref{compatible}.

(i) Then the problem \eqref{1.1}--\eqref{1.2} with $\rho_{+}=1$ admits a unique local solution $(\rho, u, \mathbb{F}, \na\phi)$ satisfying for some $T>0$,
\begin{align}\label{regularity-1-local}
\begin{cases}
\rho\in \mathcal{C}([0,T];H^2(\Omega)),\quad \rho_{t}\in \mathcal{C}([0,T];H^1(\Omega)),\\
u\in \mathcal{C}([0,T];H^2(\Omega)\cap H^1_0(\Omega))\cap L^2([0,T];H^3(\Omega)),\\
u_{t}\in \mathcal{C}([0,T];L^2(\Omega))\cap L^2([0,T];H^1_0(\Omega)),\\
\mathbb{F}\in \mathcal{C}([0,T];H^2(\Omega)),\quad \mathbb{F}_{t}\in \mathcal{C}([0,T];H^1(\Omega)),\\
\displaystyle\inf_{(x,t)\in\overline{\Omega}\times[0,T]}\rho(x,t)\ge\frac{1}{2}+\frac{1}{2}
\inf_{x\in\overline{\Omega}}\{\rho_0(x)\}>0
\end{cases}
\end{align}
and
\begin{align*}
\na\phi\in \mathcal{C}([0,T];H^3(\Omega)),\quad \na\phi_{t}\in \mathcal{C}([0,T];H^2(\Omega)).
\end{align*}
Moreover, it holds that
\begin{align*}
\sup_{0\le t\le T}\big(\|(\rho-1, u, \mathbb{F}-\mathbb{I})(t)\|_{H^2}+\|\na\phi(t)\|_{H^3}
+\|\varphi(t)\|_{L^2}\big)\le C_2\big(\|(\rho_0-1, u_0, \mathbb{F}_0-\mathbb{I})\|_{H^2}+\|\varphi_0\|_{L^2}\big).
\end{align*}

(ii) Then the problem \eqref{1.1}--\eqref{1.2} with $\rho_{+}=e^{-\phi}$ admits a unique global solution $(\rho, u, \mathbb{F}, \phi)$ satisfying \eqref{regularity-1-local}
and
\begin{align*}
\phi\in \mathcal{C}([0,T];H^4(\Omega)),\quad \phi_{t}\in \mathcal{C}([0,T];H^3(\Omega)).
\end{align*}
Moreover, it holds that
\begin{align*}
\sup_{0\le t\le T}\big(\|(\rho-1, u, \mathbb{F}-\mathbb{I})(t)\|_{H^2}+\|\phi(t)\|_{H^4}
+\|\varphi(t)\|_{L^2}\big)\le C_2\big(\|(\rho_0-1, u_0, \mathbb{F}_0-\mathbb{I})\|_{H^2}+\|\varphi_0\|_{L^2}\big).
\end{align*}

Here the above $C_2>1$ is some fixed constant.
\end{Proposition}

Then the local solution given in Proposition \ref{th1.1} can be extended to the global one by combining the a priori estimates given in Proposition \ref{apriori} with a standard continuous argument, cf. \cite{Matsumura-Nishida1980}. Therefore, we complete the proof of Theorem \ref{th1.2}.

\appendix

\section{Derivation of Models}\label{appendix}

We consider the viscoelastic two-fluids: one is made up of negatively charged particles (called negative fluids for convenience) and the other is made up of positively charged particles (called positive fluids). We want to propose a PDE model to describe the dynamics of this kind of charged two-fluids under the self-consistent electrostatic field on a unbounded domain corresponding to the unbounded problems considered in this paper. We will adopt the so-called energetic variational approach developed in \cite{Forster2013,Giga-Kirshtein-Liu2018}, which is proved to be effectively to derive mathematical models for complex fluid dynamics.

The energetic variational approach is mainly based on energy dissipation laws, Maximum Dissipation Principle, Least Action Principle, and force balance laws. In order to apply the energetic variational approach, we need to calculate the variation of the action functional on the flow map and the variation of the dissipation functional on the velocity. For this reason, we first introduce the flow map $x_\pm(X,t)$ as follows.

For a given velocity field $u_\pm(x_\pm(X,t),t)$, the flow map $x_\pm(X,t)$ is determined by the following initial value problem:
\begin{align}\label{0-20220629}
\begin{cases}
\displaystyle\frac{d}{dt}x_\pm(X_\pm,t)=u_\pm(x_\pm(X,t),t),\quad t>0,\\
x_\pm(X_\pm,0)=X,
\end{cases}
\end{align}
where $X$, $x_\pm\in\Omega\subseteq\r3$ denote the Lagrangian coordinate and Eulerian coordinate of the particle, respectively. Here the subscripts $-$ and $+$ represent the negative and positive particles, respectively.
Note that we first have at hand two mass conservation equations
\begin{align}\label{1-20220629}
\begin{cases}
\pa_t\rho_{-} +\diver (\rho_{-} u_{-})=0,\\
\pa_t\rho_{+} +\diver (\rho_{+} u_{+})=0
\end{cases}
\end{align}
and the Poisson equation by Gauss' law
\begin{align}\label{2-20220629}
\De \phi=\rho_{-}-\rho_{+}.
\end{align}

For viscoelastic electrically conducting fluids occupying unbounded domains $\Omega$, the total energy should contain the kinetic energy and the Helmholtz free energy. So, we start with the following energy dissipation law:
\begin{align}\label{3-20220629}
\frac{d}{dt}E^{\mbox{total}}=-\triangle,
\end{align}
where
\begin{align*}
E^{\mbox{total}}:=\int_{\Omega}\Big[\frac{1}{2}\rho_{+}|u_{+}|^2+\frac12\rho_{-}|u_{-}|^2+\omega_{+}(\rho_{+})+\omega_{-}(\rho_{-})+\frac 12c_{+}^2\rho_{+}|\mathbb{F}_{+}|^2+\frac 12c_{-}^2\rho_{-}|\mathbb{F}_{-}|^2+\frac12|\na\phi|^2\Big]\,dx
\end{align*}
and
\begin{align*}
\triangle:=\int_{\Omega}\Big[\mu_{+}|\na u_{+}|^2+(\mu_{+}+\la_{+})|\diver u_{+}|^2+\al_{+}\rho_{+}|u_{+}|^2+\mu_{-}|\na u_{-}|^2+(\mu_{-}+\la_{-})|\diver u_{-}|^2+\al_{-}\rho_{-}|u_{-}|^2\Big]\,dx.
\end{align*}
The symbols $E^{\mbox{total}}$ and $\triangle$ denote the total energy and the entropy production, respectively. Note that the entropy production is generated by fluid viscosity and skin friction with the fixed boundary $\pa\Omega$, where the terms $\al_{\pm}\rho_{\pm}|u_{\pm}|^2=\al_{\pm}\rho_{\pm}|u_{\pm}-0|^2$ are well understood as the friction induced by the relative motion with the motionless physical boundary $\pa\Omega$.

By the Maximum Dissipation Principle \cite{Forster2013}, taking the variation (for any smooth $\tilde{u}_{-}$ with compact support) with respect to $u_{-}$ yields
\begin{align*}
0&=\left.\frac{d}{d\va}\right|_{\va=0}\frac12\triangle(u_{-}+\va\tilde{u}_{-})\nonumber\\
&=\left.\frac{d}{d\va}\right|_{\va=0}\frac12\int_{\Omega}\Big[\mu_{-}|\na u_{-}+\va\na  \tilde{u}_{-}|^2+(\mu_{-}+\la_{-})|\diver u_{-}+\va\diver \tilde{u}_{-}|^2+\al_{-}\rho_{-}|u_{-}+\va\tilde{u}_{-}|^2\Big]\,dx\nonumber\\
&=\int_{\Omega}[\mu_{-}\na u_{-}:\na\tilde{u}_{-}+(\mu_{-}+\la_{-})\diver u_{-}\diver\tilde{u}+\al_{-}\rho_{-} u_{-}\cdot\tilde{u}_{-}]\,dx\nonumber\\
&=\int_{\Omega}[-\mu_{-}\De u_{-}-(\mu_{-}+\la_{-})\na\diver u_{-}+\al_{-}\rho_{-} u_{-}]\cdot \tilde{u}_{-}\,dx.
\end{align*}
Since $ \tilde{u}_{-}$ is arbitrary, we obtain
\begin{align}\label{4-20220629}
-\mu\De u_{-}-(\mu_{-}+\la_{-})\na\diver u_{-}+\al_{-}\rho_{-} u_{-}&=0,\nonumber\\
\mbox{and}\ F^{-}_{dissipative}&=\frac{\de(\tfrac12\triangle)}{\de u_{-}}=-\mu_{-}\De u_{-}-(\mu_{-}+\la_{-})\na\diver u_{-}+\al_{-}\rho_{-} u_{-},
\end{align}
where $F^-_{dissipative}$ denotes the dissipative force for negative fluids.

Now, we need to find the conservative force from the total energy $E^{\mbox{total}}$. Solving the Poisson equation \eqref{2-20220629} in $\Omega$, we obtain
\begin{align}\label{5-20220629}
\phi(x)=\int_{\Omega}G(x-y)(\rho_{+}-\rho_{-})(y)\,dy,
\end{align}
where $G(\cdot)$ is the Green's kernel. So, by integrating by parts and \eqref{5-20220629}, we have
\begin{align}\label{6-20220629}
\int_{\Omega}\frac12|\na\phi|^2\,dx=\frac12\int_{\Omega}(\rho_{-}-\rho_{+})(x)\int_{\Omega}G(x-y)(\rho_{-}-\rho_{+})(y)\,dydx.
\end{align}
Note that the relation between Eulerian coordinates $x_{-}$ and Lagrangian coordinates $X$, see \cite{Forster2013}, it holds that
\begin{align}\label{7-20220629}
\rho_{-}(x_{-}(X,t),t)=\frac{\rho_{-0}(X)}{\det(\widetilde{\mathbb{F}}_{-})}\quad \mbox{and}\quad \widetilde{\mathbb{F}}_{-}(X,t)=\mathbb{F}_{-}(x_{-}(X,t),t)=\frac{\pa x_{-}(X,t)}{\pa X}.
\end{align}
Given the total energy $E^{\mbox{total}}$ in \eqref{3-20220629}, by \eqref{6-20220629} and \eqref{7-20220629}, we can define the following action functionals:
\begin{align*}
&\mathcal{A}_{-}(x_{-}(X,t)):=\int_0^{t^\ast}\int_{\Omega}\Big[\frac{\rho_{-}}{2}|u_{-}|^2
-\frac{c_{-}^2\rho_{-}}{2}|\mathbb{F}_{-}|^2\Big]\,dx_{-}dt=\int_0^{t^\ast}\int_{\Omega} \Big[\frac12\rho_{-0}(X)|x_{-t}|^2-\frac{c_{-}^2}{2}\rho_{-0}(X)|\widetilde{\mathbb{F}}_{-}|^2\Big]\,dXdt;\\
&\mathcal{B}_{-}(\rho_{-}):=-\int_0^{t^\ast}\int_{\Omega}\omega_{-}(\rho_{-})\,dxdt
-\frac12\int_0^{t^\ast}\iint_{\Omega\times\Omega} G(x-y)(\rho_{-}-\rho_{+})(x)(\rho_{-}-\rho_{+})(y)\,dydxdt.
\end{align*}
Then, the conservative force is
\begin{align}\label{8-20220629}
F^{-}_{conservative}=\frac{\de\mathcal{A}_{-}}{\de x_{-}}+\rho_{-}\na\frac{\de \mathcal{B}_{-}}{\de\rho_{-}},
\end{align}
where $\frac{\de \mathcal{B}_{-}}{\de\rho_{-}}$ can be looked as a coupling potential induced by the pressure and the Coulomb force due to the fact $\frac{\de\mathcal{B}_{-}}{\de x_{-}}=\rho_{-}\na\frac{\de \mathcal{B}_{-}}{\de\rho_{-}}$.

By the Least Action Principle \cite{Forster2013}, taking the variation (for any smooth $y_{-}(X,t)=\tilde{y}_{-}(x_{-}(X,t),t)$ with compact support) with respect to the flow map $x_{-}$ yields
\begin{align*}
0&=\left.\frac{d}{d\va}\right|_{\va=0}\mathcal{A}_{-}(x_{-}(X,t)+\va y_{-}(X,t))\nonumber\\
&=\left.\frac{d}{d\va}\right|_{\va=0}\int_0^{t^\ast}\int_{\Omega}\Big[\frac12\rho_{-0}(X)|x_{-t}(X,t)+\va y_{-t}(X,t)|^2-\frac{c_{-}^2}{2}\rho_{-0}(X)\left|\frac{\pa x_{-}(X,t)}{\pa X}+\va\frac{\pa y_{-}(X,t)}{\pa X}\right|^2\Big]\,dXdt\nonumber\\
&=\int_0^{t^\ast}\int_{\Omega}\big[-(\rho_{-} u_{-})_t-\diver(\rho_{-} u_{-}\otimes u_{-})+c_{-}^2\diver(\rho_{-} \mathbb{F}_{-}\mathbb{F}_{-}^T)\big]\tilde{y}_{-}\,dxdt.
\end{align*}
Since $\tilde{y}_{-}$ is arbitrary, we obtain
\begin{align}\label{9-20220629}
-&(\rho_{-} u_{-})_t-\diver(\rho_{-} u_{-}\otimes u_{-})+c_{-}^2\diver(\rho_{-} \mathbb{F}_{-}\mathbb{F}_{-}^T)=0\nonumber\\
&\Rightarrow\frac{\de\mathcal{A}}{\de x_{-}}=-(\rho_{-} u_{-})_t-\diver(\rho_{-} u_{-}\otimes u_{-})+c_{-}^2\diver(\rho_{-} \mathbb{F}_{-}\mathbb{F}_{-}^T).
\end{align}
By the Least Action Principle \cite{Forster2013} again, taking the variation (for any smooth $\tilde{v}_{-}$ with compact support) with respect to $\rho_{-}$ yields
\begin{align*}
0&=\left.\frac{d}{d\va}\right|_{\va=0}\mathcal{B}_{-}(\rho_{-}+\va\tilde{v}_{-})\nonumber\\
&=\left.\frac{d}{d\va}\right|_{\va=0}\Big[\int_0^{t^\ast}\int_{\Omega} -\omega(\rho_{-}+\va\tilde{v}_{-})\,dxdt\Big]\nonumber\\
&\quad-\left.\frac{d}{d\va}\right|_{\va=0}\Big[\frac{1}{2}\int_0^{t^\ast}\iint_{\Omega\times\Omega} G(x-y)(\rho_{-}+\va \tilde{v}_{-}-\rho_{+})(x)(\rho_{-}+\va \tilde{v}_{-}-\rho_{+})(y)\,dydxdt\Big]\nonumber\\
&=\int_0^{t^\ast}\int_{\Omega}-\omega_{-}'(\rho_{-})\tilde{v}_{-}\,dxdt
-\frac{1}{2}\int_0^{t^\ast}\iint_{\Omega\times\Omega} G(x-y)(\rho_{-}-\rho_{+})(y)\,dy\tilde{v}_{-}(x)\,dxdt\nonumber\\
&\quad-\frac{1}{2}\int_0^{t^\ast}\iint_{\Omega\times\Omega} G(x-y)(\rho_{-}-\rho_{+})(x)\tilde{v}_{-}(y)\,dydxdt\nonumber\\
&=\int_0^{t^\ast}\int_{\Omega}\left(-\omega_{-}'(\rho_{-})+\phi(x)\right)\tilde{v}_{-}(x)\,dxdt.
\end{align*}
Since $\tilde{v}_{-}$ is arbitrary, we obtain
\begin{align}\label{10-20220629}
&-\omega_{-}'(\rho_{-})+\phi(x)=0\Rightarrow\frac{\de\mathcal{B}_{-}}{\de \rho_{-}}=-\omega_{-}'(\rho_{-})+\phi(x)\nonumber\\
&\Rightarrow\rho_{-}\na\frac{\de\mathcal{B}_{-}}{\de\rho_{-}}=-\na p_{-}+\rho_{-}\na\phi,\ p_{-}:=\rho_{-}\omega_{-}'-\omega_{-}.
\end{align}
So, by \eqref{8-20220629}--\eqref{10-20220629}, we obtain
\begin{align}\label{11-20220629}
F_{conservative}^-=-(\rho_{-} u_{-})_t-\diver(\rho_{-} u_{-}\otimes u_{-})+c_{-}^2\diver(\rho_{-} \mathbb{F}_{-}\mathbb{F}_{-}^T)-\na p_{-}+\rho_{-}\na\phi.
\end{align}
By \eqref{4-20220629} and \eqref{11-20220629}, the total force balance gives
\begin{align*}
F_{conservative}^-=F_{dissipative}^-,
\end{align*}
that is,
\begin{align}\label{12-20220629}
&(\rho_{-} u_{-})_t+\diver(\rho_{-} u_{-}\otimes u_{-})-c_{-}^2\diver(\rho_{-} \mathbb{F}_{-}\mathbb{F}_{-}^T)+\na p_{-}-\rho_{-}\na\phi\nonumber\\
&\quad=\mu_{-}\De u_{-}+(\mu_{-}+\la_{-})\na\diver u_{-}-\al_{-}\rho_{-} u_{-},
\end{align}
where $p_{-}:=\rho_{-}\omega_{-}'-\omega_{-}$.

Similarly, we can derive
\begin{align}\label{13-20220629}
&(\rho_{+} u_{+})_t+\diver(\rho_{+} u_{+}\otimes u_{+})-c_{+}^2\diver(\rho_{+} \mathbb{F}_{+}\mathbb{F}_{+}^T)+\na p_{+}+\rho_{+}\na\phi\nonumber\\
&\quad=\mu_{+}\De u_{+}+(\mu_{+}+\la_{+})\na\diver u_{+}-\al_{+}\rho_{+} u_{+},
\end{align}
where $p_{+}:=\rho_{+}\omega_{+}'-\omega_{+}$.

Now, we collect \eqref{1-20220629}, \eqref{2-20220629}, \eqref{12-20220629} and \eqref{13-20220629} to obtain a viscoelastic two-fluid system
\begin{align}\label{14-20220629}
\begin{cases}
\rho_{-t} +\diver (\rho_{-} u_{-})=0,\\
(\rho_{-} u_{-})_t+\diver(\rho_{-} u_{-}\otimes u_{-})+\na p_{-}\\
\quad=\mu_{-}\De u_{-}+(\mu_{-}+\la_{-})\na\diver u_{-}+c_{-}^2\diver(\rho_{-} \mathbb{F}_{-}\mathbb{F}_{-}^T)+\rho_{-}\na\phi-\al_{-}\rho_{-} u_{-},\\
\mathbb{F}_{-t}+u_{-}\cdot\na \mathbb{F}_{-}=\na u_{-}\mathbb{F}_{-},\\
\rho_{+t} +\diver (\rho_{+} u_{+})=0,\\
(\rho_{+} u_{+})_t+\diver(\rho_{+} u_{+}\otimes u_{+})+\na p_{+}\\
\quad=\mu_{+}\De u_{+}+(\mu_{+}+\la_{+})\na\diver u_{+}+c_{+}^2\diver(\rho_{+} \mathbb{F}_{+}\mathbb{F}_{+}^T)-\rho_{+}\na\phi-\al_{+}\rho_{+} u_{+},\\
\mathbb{F}_{+t}+u_{+}\cdot\na \mathbb{F}_{+}=\na u_{+}\mathbb{F}_{+},\\
\De \phi=\rho_{-}-\rho_{+}.
\end{cases}
\end{align}
In fact, we can go back the energy dissipation law \eqref{3-20220629} by multiplying
the first six equations in \eqref{14-20220629} by $\omega_{-}'(\rho_{-}),u_{-},c_{-}^2\rho_{-} \mathbb{F}_{-},\omega_{+}'(\rho_{+}),u_{+},c_{+}^2\rho_{+}\mathbb{F}_{+}$, respectively, summing them up and then integrating over $\Omega$.

Next, we analyze the dynamics of the positive fluid. When the positive fluid becomes a steady state, the equations for the positive fluid become
\begin{align}\label{1-20220706}
\begin{cases}
\diver (\rho_{+} u_{+})=0,\\
\rho_{+} u_{+}\cdot\na u_{+}-c_{+}^2\diver(\rho_{+} \mathbb{F}_{+}\mathbb{F}_{+}^T)+\na p_{+}+\rho_{+}\na\phi=\mu_{+}\De u_{+}+(\mu_{+}+\la_{+})\na\diver u_{+}-\al_{+}\rho_{+} u_{+},\\
u_{+}\cdot\na \mathbb{F}_{+}=\na u_{+}\mathbb{F}_{+}.
\end{cases}
\end{align}
Multiplying Eq. $\eqref{1-20220706}_2$, Eq. $\eqref{1-20220706}_3$ by $u_{+}$, $c_{+}^2\rho_{+}\mathbb{F}_{+}$, respectively, summing them up and integrating over $\Omega$ by parts, we deduce
\begin{align*}
\int_{\Omega}\Big[\mu_{+}|\na u_{+}|^2+(\mu_{+}+\la_{+})|\diver u_{+}|^2+\al_{+}\rho_{+}|u_{+}|^2\Big]\,dx=0,
\end{align*}
which together with \eqref{0-20220629} implies
\begin{align*}
u_{+}=0,\quad \mathbb{F}_{+}=\mathbb{I}.
\end{align*}
Thus the stationary system \eqref{1-20220706} is reduced to a single equation
\begin{align}\label{2-20220706}
\na p_{+}(\rho_{+})=-\rho_{+}\na\phi.
\end{align}
It is easy to check that the equation \eqref{2-20220706} has a constant solution $\rho_{+}\equiv\bar\rho>0$ and $\phi_{+}\equiv0$. Under certain conditions, the equation \eqref{2-20220706} also has a nonconstant solution. Under the simple case of $p_{+}(\rho_{+})=\rho_{+}$, one can derive
\begin{align*}
\rho_{+}=e^{-\phi},
\end{align*}
which gives a Boltzmann distribution for positive charged particles. Thus, the system \eqref{14-20220629} is reduced to
\begin{align*}
\begin{cases}
\rho_{-t} +\diver (\rho_{-} u_{-})=0,\\
(\rho_{-} u_{-})_t+\diver(\rho_{-} u_{-}\otimes u_{-})-c_{-}^2\diver(\rho_{-} \mathbb{F}_{-}\mathbb{F}_{-}^T)+\na p_{-}-\rho_{-}\na\phi=\mu_{-}\De u_{-}+(\mu_{-}+\la_{-})\na\diver u_{-}-\al_{-}\rho_{-} u_{-},\\
\mathbb{F}_{-t}+u_{-}\cdot\na \mathbb{F}_{-}=\na u_{-}\mathbb{F}_{-},\\
\De \phi=\rho_{-}-\rho_{+},\quad \rho_{+}=\bar\rho\ \mbox{or}\ e^{-\phi},
\end{cases}
\end{align*}
which is a closed system. This is exactly the system \eqref{1.1} by removing the subscript $-$.

\section*{Acknowledgements}

This work was partially supported by National Key R\&D Program of China (No. 2021YFA1002900), Guangdong Provincial Pearl River Talents Program (No. 2017GC010407), Guangdong Province Basic and Applied Basic Research Fund (No. 2021A1515010235), Guangzhou City Basic and Applied Basic Research Fund (No. 202102020436) and the NSF of China (Nos. 11701264 and 11971179).

\bigskip

{\bf Data Availability:} The datasets generated during and/or analysed during the current study are available from the corresponding author on reasonable request.

\bibliography{bib}

\end{document}